\documentclass[12pt]{amsart}
\usepackage[margin=1in]{geometry}
\usepackage[all]{xy}
\usepackage{amsmath,amsfonts,amssymb,mathdots,mathrsfs,esint,mathtools,enumitem,mnsymbol,amsthm}
\usepackage{verbatim}
\usepackage[table]{xcolor}

\newtheorem{theorem}{Theorem}
\newtheorem{lemma}{Lemma}[section]
\newtheorem{proposition}[lemma]{Proposition}
\newtheorem{corollary}[lemma]{Corollary}
\newtheorem{claim}[lemma]{Claim}
\newtheorem{remark}{Remark}[section]

\numberwithin{equation}{section}

\newcommand{\gtorncw}{\succnsim}
\newcommand{\UnipotentOrbit}{\ensuremath{\mathcal{O}}}
\newcommand{\RingOfIntegers}{\ensuremath{\mathfrak{O}}}
\DeclareMathOperator*{\Res}{Res}
\newcommand{\Fsquares}{\ensuremath{F{^{*2}}}}
\newcommand{\Frs}{\ensuremath{F{^{*r}}}}
\newcommand{\Repart}{\ensuremath{\text{Re}}}
\newcommand{\setdifference}{\ensuremath{\mathrm{-}}}
\newcommand{\matha}{\ensuremath{\mathrm{a}}}
\newcommand{\mathc}{\ensuremath{\mathrm{c}}}
\newcommand{\mathe}{\ensuremath{\mathrm{e}}}

\newcommand{\mathm}{\ensuremath{\mathrm{m}}}

\newcommand{\quasisplit}{quasi-split} 
\newcommand{\nonarchimedean}{non-Archimedean} 
\newcommand{\archimedean}{Archimedean} 
\newcommand{\GL}[1]{GL_{#1}} 
\newcommand{\GLF}[2]{GL_{#1}(#2)} 
\newcommand{\CGLF}[2]{\cover{GL}_{#1}(#2)} 
\newcommand{\rmodulo}[2]{#1 / #2} 
\newcommand{\glnidx}{\ensuremath{n}}

\newcommand{\Induced}[3]{\ensuremath{Ind_{#1}^{#2}(#3)}}
\newcommand{\absdet}[1]{\ensuremath{|\det{#1}|}}
\newcommand{\half}{\ensuremath{\frac12}}

\newcommand{\lmodulo}[2]{\ensuremath{#1 \backslash #2}}
\newcommand{\rconj}[1]{\ensuremath{{}^{#1}}}
\newcommand{\isomorphic}{\ensuremath{\cong}}
\newcommand{\transpose}[1]{\ensuremath{\rconj{t}#1}}
\newcommand{\setof}[2]{\ensuremath{\{ #1 : #2 \}}}
\newcommand{\Z}{\ensuremath{\mathbb{Z}}}
\newcommand{\C}{\ensuremath{\mathbb{C}}}
\newcommand{\R}{\ensuremath{\mathbb{R}}}
\newcommand{\Adele}{\ensuremath{\mathbb{A}}}
\newcommand{\rhs}{right-hand side} 
\newcommand{\glnrep}{\ensuremath{\tau}}

\newcommand{\induced}[3]{\ensuremath{Ind_{#1}^{#2}(#3)}}
\newcommand{\cinduced}[3]{\ensuremath{ind_{#1}^{#2}(#3)}}
\newcommand{\cover}[1]{\ensuremath{\widetilde{#1}}}

\begin{document}
\title[The double cover of $GSpin(2\glnidx+1)$ and small representations]{The double cover of odd general spin groups, small representations and applications}
\author{Eyal Kaplan}
\email{kaplaney@gmail.com}

{\let\thefootnote\relax\footnote{Author partially supported by the ISF Center of Excellence grant \#1691/10.}}
\begin{abstract}
We construct local and global metaplectic double covers of odd general spin groups, using the cover of Matsumoto of spin groups. Following Kazhdan and Patterson, a local exceptional representation is the unique irreducible quotient of a principal series representation, induced from a certain exceptional character. The global exceptional representation is obtained as the multi-residue of an Eisenstein series, it is an automorphic representation and decomposes as the restricted tensor product of local exceptional representations. As in the case of the small representation of $SO_{2\glnidx+1}$ of Bump, Friedberg and Ginzburg, exceptional representations enjoy the vanishing of a large class of twisted Jacquet modules (locally), or Fourier coefficients (globally). Consequently they are useful in many settings, including lifting problems and Rankin-Selberg integrals. We describe one application, to a calculation of a co-period integral.
\end{abstract}

\maketitle

Let $G_{\glnidx}=GSpin_{2\glnidx+1}$ be the split odd general spin group of rank $\glnidx+1$. Its derived group is $G'_{\glnidx}=Spin_{2\glnidx+1}$, the simple split simply-connected algebraic group of type $B_{\glnidx}$. The group $G_{\glnidx}$ occurs as a Levi subgroup of $G'_{\glnidx+1}$. For a local field $F$ of characteristic $0$, let
$\cover{G'}_{\glnidx+1}(F)$ be the metaplectic double cover of $G'_{\glnidx+1}(F)$ defined by Matsumoto \cite{Mats}. We can obtain a double cover $\cover{G}_{\glnidx}(F)$ of $G_{\glnidx}(F)$ by restriction.

Following Banks, Levi and Sepanski \cite{BLS} we define a section $\mathfrak{s}$ and a $2$-cocycle $\sigma$ of $G'_{\glnidx+1}(F)$, representing the cohomology class in $H^2(G'_{\glnidx+1}(F),\{\pm1\})$ of $\cover{G'}_{\glnidx+1}(F)$. We show that the restriction of $\sigma$ to $G_{\glnidx}(F)\times G_{\glnidx}(F)$ satisfies a block-compatibility relation, with respect to standard Levi subgroups. This is a useful condition for studying parabolically induced representations.

Fix a Borel subgroup in $G_{\glnidx}(F)$. The preimage $\cover{T}_{\glnidx+1}(F)$ of the maximal torus $T_{\glnidx+1}(F)$ in the cover is a two step nilpotent subgroup, its irreducible genuine representations are parameterized by genuine characters of its center $C_{\cover{T}_{\glnidx+1}(F)}$. The analogous theory for covers of $\GL{\glnidx}$ was developed by Kazhdan and Patterson \cite{KP}, who studied a special class of genuine characters which they called ``exceptional". In our setting these are characters $\chi$ of $C_{\cover{T}_{\glnidx+1}(F)}$ satisfying $\chi({\alpha^{\vee}}^*(x^{\mathfrak{l}(\alpha)}))=|x|$ for all simple roots $\alpha$ of $G_{\glnidx}$ and $x\in F^*$, where
${\alpha^{\vee}}^*$ is a certain lift of the coroot $\alpha^{\vee}$ to the cover and $\mathfrak{l}(\alpha)$ is the length of $\alpha$.

We use an exceptional character $\chi$ to construct a genuine principal series representation of $\cover{G}_{\glnidx}(F)$, which has a unique irreducible quotient denoted $\Theta=\Theta_{G_{\glnidx},\chi}$. The quotient $\Theta$ is an exceptional representation, or a small representation in the terminology of Bump, Friedberg and Ginzburg \cite{BFG}. Our purpose is to develop a theory of these representations.

Let $\mathcal{O}$ be a unipotent class of $G_{\glnidx}$, it corresponds to a partition of $2\glnidx+1$ for which an even number appears with an even multiplicity. Let $V_{\UnipotentOrbit}$ be the corresponding unipotent subgroup. We consider certain characters of $V_{\UnipotentOrbit}(F)$, called ``generic". 
Roughly, a character $\psi$ of $V_{\UnipotentOrbit}(F)$ is generic if it is in general position.
The definitions are similar to those of Bump, Friedberg and Ginzburg \cite{BFG} (see also \cite{CM,Cr}) for $SO_{2\glnidx+1}$ and are given in Section~\ref{subsubsection:vanishing and results}. The following result characterizes the sense of ``smallness" of $\Theta$.
\begin{theorem}\label{theorem:local vanishing result}
Assume that $F$ is a $p$-adic field with an odd residual characteristic.
Let $\UnipotentOrbit_{0}=(2^{\glnidx}1)$ if $\glnidx$ is even, otherwise $\UnipotentOrbit_{0}=(2^{\glnidx-1}1^3)$. Let $\UnipotentOrbit$ be any class greater than or non comparable with $\UnipotentOrbit_0$ and $\psi$ be a generic character of
$V_{\UnipotentOrbit}(F)$. The twisted Jacquet module of $\Theta$ with respect to $V_{\UnipotentOrbit}(F)$ and $\psi$ is zero.
\end{theorem}
Theorem~\ref{theorem:local vanishing result} is proved in Section~\ref{subsubsection:vanishing and results}. To remove the restriction on the residual characteristic, we only need to know that exceptional representations of the double cover of $\GLF{\glnidx}{F}$ (with the \cite{KP} parameter $\mathc=0$) do not have Whittaker models, for $\glnidx\geq3$ and even residual characteristic, this is an expected result (\cite{BG} p.~145, see \cite{FKS} Lemma~6).

The local theory has a global counterpart. Let $F$ be a number field with an ad\`{e}les ring $\Adele$. Let $\Theta=\Theta_{G_{\glnidx},\chi}$ be the exceptional representation of $\cover{G}_{\glnidx}(\Adele)$ defined with respect to a global exceptional character $\chi$. It has an automorphic realization as the multi-residue of an Eisenstein series with respect to the Borel subgroup. Also $\Theta\isomorphic\prod'_{\nu}\Theta_{G_{\glnidx},\chi_{\nu}}$ (restricted tensor product). The analogs of the Jacquet modules are Fourier coefficients, over the quotient $\lmodulo{V_{\UnipotentOrbit}(F)}{V_{\UnipotentOrbit}(\Adele)}$, with respect to generic characters of $V_{\UnipotentOrbit}(\Adele)$ trivial on $V_{\UnipotentOrbit}(F)$. See Section~\ref{subsubsection:global vanishing results} for the definitions.
The local-global principle (see e.g. \cite{JR} Proposition~1) immediately implies the following global
corollary of Theorem~\ref{theorem:local vanishing result}.
\begin{theorem}\label{theorem:global vanishing result}
Let $\UnipotentOrbit$ be any orbit greater than or non comparable with $\UnipotentOrbit_0$. Any Fourier coefficient with respect to $\UnipotentOrbit$ and a generic character vanishes identically on the space of $\Theta$.
\end{theorem}

A minimal representation is a representation supported on the minimal coadjoint orbit. If $\glnidx\geq4$,
Vogan \cite{V} (Theorem~2.13) proved that $SO_{2\glnidx+1}$ (or its cover groups) do not afford such representations.
Bump, Friedberg and Ginzburg \cite{BFG} constructed a local and global ``small" representation $\Theta_{SO_{2\glnidx+1}}$ of
$SO_{2\glnidx+1}$. It is a representation of a cover $\cover{SO}_{2\glnidx+1}(F)$, this cover was obtained by restriction of the $4$-fold cover of $SL_{2\glnidx+1}(F)$ of Matsumoto \cite{Mats}. In the cases $\glnidx=2,3$ it is in fact the minimal representation. It is small in the sense that it is supported on the orbit $\UnipotentOrbit_0$ \cite{BFG,BFG2}.
The use of the $4$-fold cover implies a minor technical restriction on the field, namely that $-1$ is a square.

The arguments of Vogan \cite{V} apply also to $G_{\glnidx}$, that is, for $\glnidx\geq4$ there is no minimal representation. It is reasonable to call $\Theta$ a small representation of $G_{\glnidx}$.

Our local and global results are parallel to those of \cite{BFG,BFG2} and are obtained using similar methods. For example, because the unipotent subgroups of $G_{\glnidx}$ are in bijection with those of $SO_{2\glnidx+1}$, manipulations on Jacquet modules are similar. One notable difference, is in the restriction of the cover to Levi subgroups. In contrast with the cover of $SO_{2\glnidx+1}$, here direct factors of Levi subgroups do not commute in the cover. This implies that representations of Levi subgroups cannot be studied using the usual tensor product. In this property, as well as in other details, the cover of $G_{\glnidx}$ is more related to the cover of $\GL{\glnidx}$, than to that of $SO_{2\glnidx+1}$.

Let $Q_{k}=M_k\ltimes U_k$ be a maximal parabolic subgroup with a Levi part $M_k$ isomorphic to $\GL{k}\times G_{\glnidx-k}$. In the particular case of $k=\glnidx$, $\CGLF{\glnidx}{F}$ and $\cover{G}_{0}(F)$ do commute and we can define a tensor product. Let $\chi^{(1)}$ be an exceptional character in the sense of Kazhdan and Patterson \cite{KP} and let $\Theta_{\GL{\glnidx},\chi^{(1)}}$ be the corresponding global exceptional representation of the double cover $\CGLF{\glnidx}{\Adele}$ of \cite{KP}. Also, let $\chi^{(2)}$ be a genuine character of $\cover{G}_0(\Adele)$ (this cover is split). Assume $\chi=\chi^{(1)}\otimes\chi^{(2)}$, for the precise meaning of this equality see
Section~\ref{subsubsection:global induction transitivity}. We compute the constant term of an automorphic form $\theta$ in the space of $\Theta$ along the unipotent radical $U_{\glnidx}$ and prove the following result.
\begin{theorem}\label{theorem:the constant term}
The function $m\mapsto\theta^{U_{\glnidx}}(m)$ on $\cover{M}_{\glnidx}(\Adele)$ belongs to the space of
\begin{align*}
\Theta_{\GL{\glnidx},\absdet{}^{-1/2}\chi^{(1)}}\otimes\Theta_{G_{0},\chi^{(2)}}.
\end{align*}
\end{theorem}
The definition of the constant term $\theta^{U_{\glnidx}}$ and the proof of the theorem occupy Section~\ref{subsubsection:the constant term}. Note that $\absdet{}^{-1/2}\chi^{(1)}$ is also an exceptional character.

For $SO_{2\glnidx+1}$ and any $k$, the mapping
$m\mapsto\theta^{U_{k}}(m)$ belongs to the space of the tensor product $\Theta_{\GL{k}}\otimes\Theta_{SO_{2(\glnidx-k)+1}}$, for a uniquely determined exceptional representation $\Theta_{\GL{\glnidx}}$ ($\Theta_{SO_{2(\glnidx-k)+1}}$ is unique). This was conjectured in \cite{BFG} and proved for $k=1$, and in general proved in \cite{me7}. We mention that for the current applications the case of $k=\glnidx$ is sufficient. For the lifting results of \cite{BFG2} (see below) the constant term was not used.

We will study representations of $M_k$ using a ``larger" induced representation, similar to the construction of Kable \cite{Kable}, see Section~\ref{subsubsection:tensor product}. A metaplectic tensor product for cover groups of $\GL{\glnidx}$ has been studied in \cite{FK,Su2,Kable,Mezo,Tk2}, but will not be used here. Refer to the discussion in Section~\ref{subsubsection:tensor discussion}.

There are several applications to our work. Essentially, one can simply replace $\Theta_{SO_{2\glnidx+1}}$ with $\Theta$. This has the benefit of removing the restriction on the field with respect to the $4$-th roots of unity. 

We describe one application, whose details are given in Section~\ref{section:application 1 coperiod}. In general,
let $G$ be a split reductive algebraic $F$-group, where $F$ is a number field. Let $Q=M\ltimes U$ be a maximal
parabolic subgroup of $G$ with a Levi part $M$ and let $\tau$ be an irreducible unitary cuspidal globally generic automorphic representation of $M(\Adele)$. Denote the central character of $\tau$ by $\omega_{\tau}$.
Denote by $E(g;\rho,s)$ the Eisenstein series corresponding to an element $\rho$ in the space of the representation of $G(\Adele)$ induced from $\tau$; $g\in G(\Adele)$ and $s\in\C$. For $s_0\in\C$, let $E_{s_0}(g;\rho)$ denote the residue of $E(g;\rho,s)$ at $s_0$.
The space spanned by the residues $E_{s_0}(\cdot;\rho)$ is called the residual representation $E_{\glnrep}$.

Periods of automorphic forms are often related to poles of $L$-functions and to questions of functoriality. Ginzburg, Jiang and Soudry \cite{GJS} described such relations in a general setup
and considered several examples. They conjectured (\cite{GJS} Conjecture~1.4) that the pole at $s=1$ of a partial $L$-function corresponding to $\tau$ is related to the non-triviality of certain period integrals, the existence of a residual representation and the existence of a representation $\tau_0$ such that $\tau$ is the Langlands functorial transfer of $\tau_0$.

Among the examples given in \cite{GJS} is the case of $G=SO_{2\glnidx+1}$ and $M=\GL{\glnidx}$. Let $A^+$ be the subgroup of id\`{e}les of $F$ whose finite components are trivial, and \archimedean\ components are equal, real and positive. Assume that $\omega_{\tau}$ is trivial on $A^+$. The pole of the Eisenstein series at $s=1/2$ is determined by the presence of a pole of the partial symmetric square $L$-function at $s=1$.
In \cite{me7} we elaborated on this case and proved a result relating a co-period integral
\begin{align*}
\int_{\lmodulo{SO_{2\glnidx+1}(F)}{SO_{2\glnidx+1}(\Adele)}}E_{1/2}(g;\rho)\theta(g)\theta'(g)dg
\end{align*}
to the ``theta period" integral
\begin{align*}
\int_{\lmodulo{\GL{\glnidx}(F)}{\GL{\glnidx}(\Adele)^1}}\rho(b)\theta^{U_{\glnidx}}(b){\theta'}^{U_{\glnidx}}(b)db.
\end{align*}
Here $\theta$ (resp. $\theta'$) is an automorphic form in the space of $\Theta_{SO_{2\glnidx+1},\vartheta}$
(resp. $\Theta_{SO_{2\glnidx+1},\vartheta^{-1}}$), where $\vartheta$ is a character of order $4$ of the group of $4$-th roots of unity ($\vartheta$ is implicit in the notation $\Theta_{SO_{2\glnidx+1}}$); $\GLF{\glnidx}{\Adele}^1$ is the kernel of $\absdet{}$ on $\GLF{\glnidx}{\Adele}$.

Bump and Ginzburg \cite{BG} constructed a Rankin-Selberg integral representing the partial symmetric square $L$-function. In particular, they showed that if this function has a pole at $s=1$, a certain period integral does not vanish. Their period integral was related in \cite{me7}, under an additional assumption on $\omega_{\glnrep}$, to the theta period above.

Using the exceptional representation $\Theta$, this result can be put in a more general setting of $G_{\glnidx}$ and in particular, will hold for any number field. Let $\tau$ be as above (but without the assumption $\omega_{\tau}|_{A^+}=1$) and $\eta$ be a unitary Hecke character.
The Eisenstein series $E(g;\rho,s)$ is now defined with respect to $G_{\glnidx}$ and the parabolic subgroup $Q_{\glnidx}$. According to Hundley and Sayag \cite{HS},
the series $E(g;\rho,s)$ is holomorphic at $\Re(s)>0$ except perhaps for a simple pole at $s=1/2$. The existence of this pole is determined by the presence of a pole of the partial $L$-function $L^S(s,\tau,Sym^2\otimes\eta)$ at $s=1$ 
(in \cite{HS} the twisting is with respect to $\eta^{-1}$ but their conventions are different, see Remark~\ref{remark:eta inverse in HS}). Takeda \cite{Tk} constructed a Rankin-Selberg integral for this function and proved that if $S$ is large enough and
$\omega_{\tau}^2\eta^{\glnidx}\ne1$, then $L^S(s,\tau,Sym^2\otimes\eta)$ is holomorphic at $s=1$. In particular, the series is holomorphic at $s=1/2$ unless $\omega_{\tau}^2\eta^{\glnidx}$ is trivial on $A^+$.

Denote the center of $G_{\glnidx}(\Adele)$ by $C_{G_{\glnidx}(\Adele)}$. We select global exceptional characters $\chi$ and $\chi'$ such that $\chi\cdot\chi'\cdot\eta=1$ on $C_{G_{\glnidx}(\Adele)}$. Let $\theta$ (resp. $\theta'$) belong to the space of $\Theta_{G_{\glnidx},\chi}$ (resp. $\Theta_{G_{\glnidx},\chi'}$). Here is our result, which follows by a minor modification to \cite{me7}.
\begin{theorem}\label{theorem:co period GSpin}
Consider the co-period integral
\begin{align*}
\mathcal{I}(E_{1/2}(\cdot;\rho),\theta,\theta')=\int_{\lmodulo{C_{G_{\glnidx}(\Adele)}G_{\glnidx}(F)}{G_{\glnidx}(\Adele)}}E_{1/2}(g;\rho)\theta(g)\theta'(g)dg.
\end{align*}
Assume that $\omega_{\tau}^2\eta^{\glnidx}$ is trivial on $A^+$. Then the following holds.
\begin{enumerate}[leftmargin=*]
  \item\label{thm:gspin part 2} There is a normalization of measures (explicitly given in the proof) such that
\begin{align*}
\mathcal{I}(E_{1/2}(\cdot;\rho),\theta,\theta')=\int_{K}\int_{\lmodulo{\GLF{\glnidx}{F}}{\GLF{\glnidx}{\Adele}^1}}\rho(bk)\theta^{U_{\glnidx}}(bk){\theta'}^{U_{\glnidx}}(bk)dbdk.
\end{align*}
Here $K$ is the product of local maximal compact subgroups. 
\item\label{thm:gspin part 3}
The co-period $\mathcal{I}(E_{1/2}(\cdot;\rho),\theta,\theta')$ is nonzero for some $(\rho,\theta,\theta')$ if and only if \begin{align*}
\int_{\lmodulo{\GL{\glnidx}(F)}{\GL{\glnidx}(\Adele)^1}}\rho_1(m)\theta_1(b)\theta_1'(b)db\ne0
\end{align*}
for some cusp form $\rho_1$ in the space of $\tau$, and $\theta_1$ (resp. $\theta_1'$) in the space of $\Theta_{\GL{\glnidx},\absdet{}^{-1/2}\chi^{(1)}}$ (resp. $\Theta_{\GL{\glnidx},\absdet{}^{-1/2}{\chi'}^{(1)}}$).
\end{enumerate}
\end{theorem}
The proof is given in Section~\ref{section:application 1 coperiod}.
Note that according to Theorem~\ref{theorem:the constant term}, the function $\theta^{U_{\glnidx}}$ belongs to the space
of an exceptional representation on $\CGLF{\glnidx}{\Adele}$.


Theorem~\ref{theorem:co period GSpin} motivates a local counterpart, which 
will be used as an ingredient in a proof of a conjecture of Lapid and Mao on Whittaker-Fourier
coefficients \cite{LM2}, for even orthogonal groups. Let $\varepsilon$ be Arthur's elliptic tempered parameter for $SO_{2\glnidx}$ and $\mathcal{S}_{\varepsilon}$ be the corresponding group \cite{A3}. Assume we have an irreducible automorphic cuspidal ($\psi$-)generic representation $\pi$ of $SO_{2\glnidx}(\Adele)$ in the $A$-packet associated to $\varepsilon$ ($\pi$ is expected to be unique). Further let $W$ and $W^{\wedge}$ be two global Whittaker-Fourier coefficients on the spaces of $\pi$ and $\pi^{\wedge}$ ($\pi^{\wedge}$ - the contragradient representation). The conjecture of Lapid and Mao relates the product $W(e)W^{\wedge}(e)$ to the size of $\mathcal{S}_{\varepsilon}$. The exceptional representation of $G_{\glnidx}$ can perhaps be also used to prove the conjecture for $GSpin_{2n}$.

The Minimal representation for the group $SO_7$, which is a representation of $\cover{SO}_7$, was constructed and studied by Roskies \cite{Roskies}, Sabourin \cite{Sabourin} and Torasso \cite{Torasso}. It was used by Bump, Friedberg and Ginzburg \cite{BFG3} to construct a Rankin-Selberg integral for the $14$-dimensional irreducible representation of the $L$-group of $SO_7$, corresponding to the third fundamental weight.
Bump, Friedberg and Ginzburg \cite{BFG2} used $\Theta_{SO_{2\glnidx+1}}$ to construct a lift, with certain functorial properties, from genuine automorphic representations of $\cover{SO}_{k}(\Adele)$ to $\cover{SO}_{m}(\Adele)$, for integers $k$ and $m$ of different parity.
Both results can be extended to the context of $G_{\glnidx}$. In particular, our results lead to a lift from
genuine automorphic representations of $\cover{GSpin_{k}}(\Adele)$ to $\cover{GSpin_{m}}(\Adele)$.


Loke and Savin \cite{LokeSavin2} constructed local and global exceptional representations for simply connected Chevalley groups. Their approach was different from \cite{KP}. They started with defining a global automorphic representation $\pi$, which was invariant under the action of the Weyl group. Their local and global exceptional representations were obtained by unramified twists of $\pi_{\nu}$ or $\pi$. The global representation was also realized as a multi-residue of an Eisenstein series. Their exposition is elegant and applicable to a wide range of groups.

Our results have some overlap with an ongoing work of Loke and Savin\footnote{private communication.}, we thank them for informing us about their work. They and the author were working independently. The approach and techniques are different. For example, they do not use a cocycle. 

Minimal representations have been studied and used by many authors. The fundamental example is the Weil representation of $\cover{Sp}_{\glnidx}$, which was used in the theta correspondence between a pair of dual reductive groups, to lift representations from one group to another \cite{Pd}. The Weil representation also played an important role in the descent method, in the construction of Fourier-Jacobi coefficients \cite{GRS3,JSd1,Soudry4,RGS}. Among the works on minimal representations are \cite{Savin2,Savin,GanSavin,LokeSavin}.
Works on minimal representations for simply laced groups include
\cite{V,KZ,KS,BK,GRS,KP3}. Minimal representations enjoy the vanishing of a large class of Fourier coefficients, which makes them valuable for applications involving lifts and Rankin-Selberg integrals \cite{GRS6,G2,GJS2}.

Early works on the metaplectic groups include the work of Weil \cite{We} on $\cover{Sp}_n$, Kubota \cite{Kubota,Kubota2} who studied the $r$-fold cover of $\GL{2}$ and Moore \cite{Moore} and Steinberg \cite{Stein}, who studied central extensions of simple Chevelley groups and also considered metaplectic groups. Matsumoto \cite{Mats} constructed the metaplectic $r$-fold cover of any simple simply connected split group $G$ over a local field. For $GL_{n}$, the metaplectic groups were constructed and studied by Kazhdan and Patterson \cite{KP}. Over a $p$-adic field containing $2r$ different $2r$-th roots of unity and such that $2r$ is coprime to the residue characteristic, McNamara \cite{McNamara} constructed $\cover{G}$ for any split reductive group $G$, using the results of Brylinski and Deligne \cite{BD} and Finkelberg and Lysenko \cite{FL}. Sun \cite{Su} studied 
metaplectic covers of $G$, defined using covers of its derived group $G'$, assuming that $G$ is also connected and $G'$ is a simple simply connected Chevalley group.

Banks, Levy and Sepanski \cite{BLS} elaborated on the work of Matsumoto \cite{Mats}, by describing an explicit section and a $2$-cocycle $\sigma_G$ representing the corresponding cohomology class in $H^2(G(F),\mu_r)$ of the cover of \cite{Mats}, where $F$ is a local field and $\mu_r$ is the subgroup of $r$ $r$-th roots of unity. They proved several compatibility results, which make their cocycle a convenient choice. For example, if $H$ is a ``standard" subgroup of $G$, which means that $H$ is a simple simply connected split group generated in $G$ by certain data, the restriction of $\sigma$ to $H(F)\times H(F)$ is $\sigma_H$. The cocycle $\sigma_{SL_{\glnidx+1}}$ for $SL_{\glnidx+1}(F)$ was used in \cite{BLS} to define a cocycle on $\GLF{\glnidx}{F}$, which (in contrast with the cocycle of \cite{KP}) is block-compatible. 

The group $GSpin_m$ has been the focus of study of a few recent works. Asgari \cite{Asg2,Asg} studied its local $L$-functions, Asgari and Shahidi \cite{AsgSha,AsgSha2} proved functoriality results and Hundley and Sayag \cite{HS} extended the descent construction to $GSpin_m$. 


\addtocontents{toc}{\protect\setcounter{tocdepth}{1}}

\subsection*{Acknowledgments} I wish to express my gratitude to Erez Lapid for suggesting this project to me and for helpful and inspiring conversations. I thank Daniel Bump, Solomon Friedberg, David Ginzburg, Anthony Kable, Paul Mezo, Gordan Savin and Shuichiro Takeda for helpful conversations. Part of this work was done while the author was participating in the workshop ``Advances in the theory of automorphic forms and their L-functions" in the Erwin Schr\"{o}dinger Institute in Vienna. I would like to thank Joachim Schwermer and the institute for their kind hospitality. 

\addtocontents{toc}{\protect\setcounter{tocdepth}{3}}

\tableofcontents

\section{Preliminaries}\label{section:preliminaries}

\subsection{The groups}\label{subsection:the groups}
Let $F$ be a field of characteristic $0$. 
For any $r\geq1$, let $\mu_{r}=\mu_{r}(F)$ be the subgroup of the $r$-th roots of unity in $F$.
Put $F^{*r}=(F^*)^r$. If $F$ is any local field, let $(,)_{r}$ be the Hilbert symbol of order $r$ of $F$, and we usually denote by $\psi$ a fixed nontrivial additive character of $F$. Then $\gamma_{\psi}$ is the normalized Weil factor associated to $\psi$ (\cite{We} Section~14, $\gamma_{\psi}(a)$ is $\gamma_F(a,\psi)$ in the notation of \cite{Rao}, $\gamma_{\psi}(\cdot)^4=1$). If $F$ is a local $p$-adic field, its ring of integers is $\RingOfIntegers$, the maximal ideal is $\mathcal{P}=\varpi\RingOfIntegers$ and $|\varpi|^{-1}=q=|\rmodulo{\RingOfIntegers}{\mathcal{P}}|$.

In the group $\GL{\glnidx}$, fix  the Borel subgroup $B_{\GL{\glnidx}}=T_{\GL{\glnidx}}\ltimes N_{\GL{\glnidx}}$ of upper triangular invertible matrices, where $T_{\GL{\glnidx}}$ is
the diagonal torus. 
Denote by $I_{\glnidx}$ the identity matrix of $\GL{\glnidx}(F)$.

We define the special odd orthogonal group
\begin{align*}
SO_{2\glnidx+1}(F)=\setof{g\in SL_{2\glnidx+1}(F)}{\transpose{g}J_{2\glnidx+1}g=J_{2\glnidx+1}},
\end{align*}
where $\transpose{g}$ is the transpose of $g$ and for any $k\geq1$, $J_{k}\in \GLF{k}{F}$ is the matrix with $1$ on the anti-diagonal and $0$ elsewhere.
Fix the Borel subgroup $B_{SO_{2\glnidx+1}}=T_{SO_{2\glnidx+1}}\ltimes N_{SO_{2\glnidx+1}}$ where $B_{SO_{2\glnidx+1}}=
B_{\GL{2\glnidx+1}}\cap SO_{2\glnidx+1}$ and $T_{SO_{2\glnidx+1}}$ is the torus. If $t\in T_{SO_{2\glnidx+1}}(F)$, $t=diag(t_1,\ldots,t_{\glnidx},1,t_{\glnidx}^{-1},\ldots,t_1^{-1})$, where $diag(\cdots)$ denotes a diagonal or block diagonal matrix. Denote by $\epsilon_i$, $1\leq i\leq\glnidx$, the $i$-th coordinate function, $\epsilon_i(t)=t_i$.

Let $Spin_{2\glnidx+1}$ be the
simple split simply-connected algebraic group of type $B_{\glnidx}$. It is the algebraic double cover of $SO_{2\glnidx+1}$. 
The standard Borel subgroup $B_{Spin_{2\glnidx+1}}$ of $Spin_{2\glnidx+1}$ is the preimage of $B_{SO_{2\glnidx+1}}$, $B_{Spin_{2\glnidx+1}}=T_{\glnidx}\ltimes N_{Spin_{2\glnidx+1}}$.
Each $\epsilon_i$ can be pulled back to $T_{\glnidx}$, 
this pull back will still be denoted $\epsilon_i$. Denote the set of roots of $Spin_{2\glnidx+1}$ by ${\Sigma}_{Spin_{2\glnidx+1}}$ and the positive roots by ${\Sigma}_{Spin_{2\glnidx+1}}^+$. The set of simple roots of $Spin_{2\glnidx+1}$ is
$\Delta_{Spin_{2\glnidx+1}}=\setof{\alpha_i}{1\leq i\leq\glnidx}$, where $\alpha_i=\epsilon_i-\epsilon_{i+1}$ for $1\leq i\leq\glnidx-1$ and $\alpha_{\glnidx}=\epsilon_{\glnidx}$.

Define $\epsilon_i^{\vee}$ with respect to the standard $\Z$-pairing $(,)$, i.e., $(\epsilon_i,\epsilon_j^{\vee})=\delta_{i,j}$.
The set $\Delta_{\glnidx}^{\vee}=\setof{\alpha_i}{1\leq i\leq\glnidx}$ of simple coroots is given by $\alpha_i^{\vee}=\epsilon_i^{\vee}-\epsilon_{i+1}^{\vee}$ for $1\leq i\leq\glnidx-1$ and $\alpha_{\glnidx}^{\vee}=\epsilon_{\glnidx}^{\vee}$.
Because
$Spin_{2\glnidx+1}$ is simply connected, any $t\in T_{\glnidx}(F)$ can be written uniquely as $t=\prod_{i=1}^{\glnidx}\alpha_i^{\vee}(t_i)$ for $t_i\in F^*$. For a description of the Levi subgroups of $Spin_{2\glnidx+1}$ see Mati\'{c} \cite{Mt}.

The group $GSpin_{2\glnidx+1}$ is an $F$-split connected reductive algebraic group, which can be defined using a based root
datum as in \cite{Asg,AsgSha,HS}. It is also embedded in $Spin_{2\glnidx+3}$ as the Levi part of the parabolic subgroup
corresponding to $\Delta_{Spin_{2\glnidx+3}}\setdifference\{{\alpha_1}\}$ (see \cite{Mt}). Since we will be constructing
a cover of $GSpin_{2\glnidx+1}(F)$ using a cover of $Spin_{2\glnidx+3}(F)$, it is natural for us to view
$GSpin_{2\glnidx+1}(F)$ as  this subgroup. Henceforth we adapt this identification. The simple roots of
$GSpin_{2\glnidx+1}$ are $\{\alpha_2,\ldots,\alpha_{\glnidx+1}\}$. In the degenerate case $\glnidx=0$,
$GSpin_{2\glnidx+1}=\GL{1}$.

The group $Spin_{2\glnidx+1}$ is the derived group of $GSpin_{2\glnidx+1}$. Denote $G_{\glnidx}=GSpin_{2\glnidx+1}$ and $G_{\glnidx}'=Spin_{2\glnidx+1}$. Additionally, let $\Sigma_{G_{\glnidx}}$ (resp. $\Sigma_{G_{\glnidx}}^+$) denote the set of roots (resp. positive roots) of
$G_{\glnidx}$, determined according to the embedding $G_{\glnidx}<G'_{\glnidx+1}$. The set of simple roots of $G_{\glnidx}$ is $\Delta_{G_{\glnidx}}=\Delta_{G'_{\glnidx+1}}\setdifference\{{\alpha_1}\}$.
The corresponding Borel subgroup of $G_{\glnidx}$ is $B_{\glnidx}=T_{\glnidx+1}\ltimes N_{\glnidx}$. For $0\leq k\leq\glnidx$, denote by $Q_k=M_k\ltimes U_k$ the standard maximal parabolic subgroup of $G_{\glnidx}$ with a Levi part $M_k$
isomorphic to $\GL{k}\times G_{\glnidx-k}$. The modulus character of $Q(F)$ for a parabolic subgroup $Q<G_{\glnidx}$ is denoted by $\delta_{Q(F)}$. 
Let $W_{\glnidx}$ be the Weyl group
of $G_{\glnidx}$. 
The longest element of $W_{\glnidx}$ is denoted $\mathbf{w_0}$. If $F$ is $p$-adic, let $K=G_{\glnidx}(\RingOfIntegers)$ be the hyperspecial subgroup.

We will also encounter the \quasisplit\ group $GSpin_{2\glnidx}$.
Let $SO_{2\glnidx}$ be a \quasisplit\ even orthogonal group, split over $F$ or a quadratic extension of $F$. The group $Spin_{2\glnidx}$ is the simply connected algebraic double cover of $SO_{2\glnidx}$. In the split case, it is the simple simply-connected algebraic group of type $D_{\glnidx}$. In the non-split case, its relative root system is of type $B_{\glnidx-1}$. Regarding $SO_{2\glnidx}$ as a subgroup of $\GL{2\glnidx}$, define
the Borel subgroup $B_{SO_{2\glnidx}}=SO_{2\glnidx}\cap B_{\GL{2\glnidx}}$. Then $B_{Spin_{2\glnidx}}$ is the preimage of $B_{SO_{2\glnidx}}$. This fixes a set of simple roots. The group $GSpin_{2\glnidx}$ can be defined as the Levi subgroup of the maximal parabolic subgroup of $Spin_{2(\glnidx+1)}$ corresponding to the subset of simple roots obtained by removing the first root (\cite{KimKim} Section~2.3). For a definition using a based root datum see \cite{Asg,AsgSha,HS}.

In general if $G$ is a group, denote by $C_G$ the center of $G$. If $H<G$, $C(G,H)$ is the centralizer of $H$ in $G$.
For any two elements $x,y\in G$, $[x,y]=xyx^{-1}y^{-1}$ denotes their commutator. Also put $\rconj{x}y=xyx^{-1}$ and $\rconj{x}H=\setof{\rconj{x}h}{h\in H}$.

\subsection{Properties of $GSpin_{2\glnidx+1}$}\label{subsection:properties of GSpin}
We collect a few structure properties of $G_{\glnidx}$ that will be used throughout. The center of $G_{\glnidx}$ is
\begin{align}\label{eq:center of the group}
C_{G_{\glnidx}(F)}=\setof{\prod_{i=1}^{\glnidx}\alpha_i^{\vee}(t_1^2)\alpha_{\glnidx+1}^{\vee}(t_1)}{t_1\in F^*}.
\end{align}
If $t=\prod_{i=1}^{\glnidx+1}\alpha_i^{\vee}(t_i)\in T_{\glnidx+1}(F)$,
\begin{align}\label{eq:conjugation of torus by w_0}
\rconj{\mathbf{w_0}}t
=\prod_{i=1}^{\glnidx+1}\alpha_i^{\vee}(t_i^{-1})\cdot
\prod_{i=1}^{\glnidx}\alpha_i^{\vee}(t_1^2)\alpha_{\glnidx+1}^{\vee}(t_1).
\end{align}
To see this one may work over the algebraic closure of $F$, where $t$ is a product of elements $\alpha_i^{\vee}(y_i)$ with
$2\leq i\leq\glnidx+1$, which are inverted by $\mathbf{w_0}$, and an element of the center.

Let $M$ be a Levi subgroups of $G_{\glnidx}$. As in classical groups, $M$ is isomorphic
to $\GL{k_1}\times\ldots\times\GL{k_l}\times G_{\glnidx-k}$, $k=k_1+\ldots+k_l$, with $0\leq k\leq\glnidx$ (\cite{Asg}). 
We define an isomorphism of $\GL{k}\times G_{\glnidx-k}$ with $M_k$. This isomorphism can be used to define an isomorphism between $\GL{k_1}\times\ldots\times\GL{k_l}\times G_{\glnidx-k}$ and $M$ using the standard embedding of
$\GL{k_1}\times\ldots\times\GL{k_l}$ in $\GL{k}$.

The derived group $SL_k$ of $\GL{k}$ is generated by the root subgroups of $\setof{\alpha_i}{2\leq i\leq k}$. To complete
the description of the embedding of $\GL{k}$, let $\eta_1^{\vee},\ldots,\eta_{k}^{\vee}$ be the standard cocharacters of
$T_{\GL{k}}$ and map $\eta_i^{\vee}\mapsto \epsilon_{i+1}^{\vee}-\epsilon_{1}^{\vee}$ for $1\leq i\leq k$.
Under this embedding, the image of a torus element $\prod_{i=1}^k\eta_i^{\vee}(a_i)$ of $\GLF{k}{F}$ is
\begin{align}\label{eq:image of standard A_k in GSpin}
\prod_{i=1}^k\alpha_i^{\vee}(\prod_{j=i}^ka_j^{-1}).
\end{align}

Regarding $G_{\glnidx-k}$, the set $\setof{\alpha_i}{k+2\leq i\leq \glnidx+1}$ identifies $G'_{\glnidx-k}$ and if $\theta_1,\ldots,\theta_{\glnidx-k+1}$ are the characters of $T_{\glnidx-k+1}$, define $\theta_1^{\vee}\mapsto\epsilon_1^{\vee}$ and for $2\leq i\leq \glnidx-k+1$, $\theta_i^{\vee}\mapsto \epsilon_{k+i}^{\vee}$.
Denote by $\beta_1^{\vee},\ldots,\beta_{\glnidx-k+1}^{\vee}$ the simple coroots of $T_{\glnidx-k+1}$. We have $\beta_i^{\vee}=\theta_i^{\vee}-\theta_{i+1}^{\vee}$ for $1\leq i\leq\glnidx-k$ and
$\beta_{\glnidx-k+1}^{\vee}=2\theta_{\glnidx-k+1}^{\vee}$. When $k=0$, $\beta_i^{\vee}=\alpha_i^{\vee}$ for all $1\leq i\leq\glnidx+1$. The image of $\prod_{i=1}^{\glnidx-k+1}\beta_i^{\vee}(t_i)$ in $G_{\glnidx}(F)$ is
\begin{align} \label{eq:image of standard T_n-k+1 in GSpin}
\begin{dcases}
\prod_{i=1}^{k}\alpha_i^{\vee}(t_1)\prod_{i=1}^{\glnidx-k+1}\alpha_{k+i}^{\vee}(t_i)&\qquad k<\glnidx,\\
\prod_{i=1}^{\glnidx}\alpha_i^{\vee}(t_1^2)\alpha_{\glnidx+1}^{\vee}(t_1)&\qquad k=\glnidx.\end{dcases}
\end{align}
For $k=\glnidx$ the image of $\beta_1^{\vee}$ is $C_{G_{\glnidx}(F)}$.

The restriction of the projection $G'_{\glnidx}\rightarrow SO_{2\glnidx+1}$ to unipotent subgroups is an isomorphism. In particular, the unipotent radical $U_k$ corresponds to the unipotent radical $U'_k$ of the standard parabolic subgroup $Q'_k$ of $SO_{2\glnidx+1}$, whose Levi part $M'_k$ is isomorphic to $\GL{k}\times SO_{2(\glnidx-k)+1}$. In coordinates,
\begin{align*}
U'_k(F)=\left\{\left(\begin{array}{ccc}I_k&u_1&u_2\\&I_{2(\glnidx-k)+1}&u_1'\\&&I_k\end{array}\right)\in SO_{2\glnidx+1}(F)\right\}\qquad(u_1'=-J_{2(\glnidx-k)+1}\transpose{u_1}J_k).
\end{align*}
If $b_0\in\GLF{k}{F}$, $b$ is the image of $b_0$ in $M_k(F)$ and $b'$ is the image of $b_0$ in $M'_k(F)$ ($b'=diag(b,I_{2(\glnidx-k)+1},J_k\transpose{b}^{-1}J_k)$), the following diagram commutes
\begin{displaymath}
    \xymatrix{
        U_k \ar[d]\ar[r]^{\mapsto\rconj{b}} &
        U_k \\
        U'_k \ar[r]^{\mapsto\rconj{b'}}       & U'_k.\ar[u] }
\end{displaymath}
Here the horizontal arrows denote conjugation.

For calculations it will sometimes be more convenient to use the cocharacters of $T_{\GL{\glnidx}}$ and the coroot of $T_1$ to write an element in $T_{\glnidx+1}$, instead of the coroots of $T_{\glnidx+1}$. Then we write
\begin{align}\label{eq:convenient coordinates of torus}
t=\prod_{i=1}^{\glnidx}\eta_i^{\vee}(a_i)\beta_1^{\vee}(t_1)
\end{align}
and the image of $t$ in $G_{\glnidx}(F)$ is
\begin{align}\label{eq:image of convenient coordinates of the torus in GSpin}
\prod_{i=1}^{\glnidx}\alpha_i^{\vee}((\prod_{j=i}^{\glnidx}a_j^{-1})t_1^2)\alpha_{\glnidx+1}(t_1).
\end{align}

In these coordinates,
\begin{align}\label{eq:formula for conjugation by w_0 with convenient coordinates}
&\rconj{\mathbf{w_0}}t=\prod_{i=1}^{\glnidx}\eta_i^{\vee}(a_i^{-1})\beta_1^{\vee}((\prod_{i=1}^{\glnidx}a_i)^{-1}t_1),\\
\label{eq:formula for delta Borel with convenient coordinates}
&\delta_{B_{\glnidx}(F)}(t)=\delta_{B_{SO_{2\glnidx+1}}(F)}(diag(a_1,\ldots,a_{\glnidx},1,a_{\glnidx}^{-1},\ldots,a_1^{-1}))=\prod_{i=1}^{\glnidx}|a_i|^{2(\glnidx-i)+1}.
\end{align}

We use the character $\epsilon_1$ to define the following ``canonical" character $\Upsilon$ of $G_{\glnidx}$, which is the extension of $-\epsilon_1$ to $G_{\glnidx}$. The extension exists because $G_{\glnidx}(F)=\setof{\alpha_1^{\vee}(x)}{x\in F^*}\ltimes G'_{\glnidx}(F)$ and $\epsilon_1|_{T_{\glnidx+1}(F)\cap G'_{\glnidx}(F)}$ is trivial. 
We call $\Upsilon$ canonical because its restriction to $\GL{k}$ is $\det$, as is evident from \eqref{eq:image of standard A_k in GSpin}. For example if $t$ is given by \eqref{eq:convenient coordinates of torus},
$\Upsilon(t)=t_1^{-2}\prod_{i=1}^{\glnidx}a_i$.
 The restriction of $\Upsilon$ to ${G_{\glnidx-k}}$ is the corresponding character of $G_{\glnidx-k}$. For $\glnidx=0$, 
 $\Upsilon(x)=x^{-2}$.

\subsection{Further definitions and notation}\label{subsection:Further definitions and notation}

\subsubsection{Central extensions}\label{subsection:central extensions}
Let $G$ be a group. We recall a few notions of abstract central extensions (i.e., ignoring topologies). A central extension of $G$ by a group $A$ is a short exact sequence of groups
\begin{align*}
1\rightarrow A\xrightarrow{\iota} \cover{G}\xrightarrow{p} G\rightarrow 1
\end{align*}
such that $\iota(A)<C_{\cover{G}}$. For any subset $X\subset G$, denote $\cover{X}=p^{-1}(X)$.
Assume we have a fixed faithful character $\vartheta:A\rightarrow\C^*$. Let $H<G$.
A representation $\pi$ of $\cover{H}$ is called $\vartheta$-genuine if $\pi(\iota(a)h)=\vartheta(\iota(a))\pi(h)$ for all
$a\in A$ and $h\in \cover{H}$. In particular if $A=\mu_2$, such a representation is simply called genuine. Any representation $\pi'$ of $G$ can be pulled back to
a nongenuine representation of $\cover{G}$ by composing it with $p$.

A section of $H$ is a mapping $\varphi:H\rightarrow \cover{G}$ such that
$\varphi(1)=1$ and $p\circ\varphi=id_H$. If $H_1<H$, we say that $\varphi$ splits $H_1$, or $H_1$ is split under $\varphi$,
if the restriction $\varphi|_{H_1}$ is a homomorphism. In this case, because $\varphi(H_1)\cap\iota(A)=\{1\}$, $\cover{H}_1=\iota(A)\cdot\varphi(H_1)$ (an inner direct product). Now any $\vartheta$-genuine representation of $\cover{H}_1$ is uniquely determined by its restriction to $\varphi(H_1)$ and conversely, any representation of $\varphi(H_1)$ can be extended uniquely to a $\vartheta$-genuine representation of $\cover{H}_1$. Of course, the splitting $\varphi$ might not be unique, but once fixed,
a representation of $H_1$ has a unique extension to a $\vartheta$-genuine representation of $\cover{H}_1$.

A $2$-cocycle defined on $G$ is a mapping $\sigma:G\times G\rightarrow A$ such that $\sigma(1,1)=1$ and for all $g,g',g''\in G$,
\begin{align}\label{eq:2-cocycle property}
&\sigma(g,g')\sigma(gg',g'')=\sigma(g,g'g'')\sigma(g',g'').
\end{align}
If $\varphi:G\rightarrow\cover{G}$ is a section, then the function $(g,g')\mapsto \iota^{-1}(\varphi(g)\varphi(g')\varphi(gg')^{-1})$ is a $2$-cocycle.
Once we fix $\sigma$, the elements of $\cover{G}$ can be written as pairs $(g,a)$ where $g\in G$ and $a\in A$. In this realization, the multiplication in $\cover{G}$ is defined by
\begin{align*}
(g,a)\cdot(g',a')=(gg',\sigma(g,g')aa').
\end{align*}
\begin{remark}
Note one subtlety of the notation. For clarity reasons, we
write $\cover{X}(F)$ instead of $\widetilde{X(F)}$, but always mean $p^{-1}(X(F))$. 
\end{remark}

\subsubsection{Representations}\label{subsection:representations}
Let $G$ be an $\ell$-group (see \cite{BZ1}~1.1). Throughout, unless mentioned otherwise, representations of $G$ will be complex, smooth and admissible. If $\pi$ is a representation of $G$,
the representation contragradient to $\pi$ is denoted $\pi^{\wedge}$.
If $\pi$ is a representation of $H<G$ and $g\in G$, $\rconj{g}\tau$ is the representation of $\rconj{g}H$ defined on the
space of $\tau$ by $\rconj{g}\tau(x)=\tau(\rconj{g^{-1}}x)$. 

Parabolically induced representations will always be normalized as in \cite{BZ2} (1.8). We use $Ind$ to denote regular induction, $ind$ signifies compact induction. The (normalized) Jacquet functor $j_{U,\psi}$ is defined as in \cite{BZ2} (1.8), where
$U$ is a unipotent 
subgroup 
and $\psi$ is a character of $U$. If $\psi$ is trivial, we
write $j_{U}=j_{U,\psi}$. 


\subsubsection{The Hilbert symbol}\label{subsection:The Hilbert symbol}
For an integer $r>1$, let $c=(,)_r^{-1}$. We recall that $c$ is an anti-symmetric bi-character, i.e., $c(1,1)=1$, $c(xy,z)=c(x,z)c(y,z)$ and
$c(x,y)=c(y,x)^{-1}$. 
Also, $c(x,y)=1$ for all $y\in F^*$ if and only if
$x\in {\Frs}$. If $|r|=1$ in $F^*$, $c$ is trivial on
$\RingOfIntegers^*\times\RingOfIntegers^*$ and $c(x,y)=1$ for all $y\in\RingOfIntegers^*$
if and only if $x\in\RingOfIntegers^*F^{*r}$ (
\cite{We2} Section~XIII.5 Proposition~6).

\subsubsection{The Weil factor}\label{subsection:The Weil symbol}
Let $\psi$ be given and $\gamma_{\psi}$ be the Weil factor (see Section~\ref{subsection:the groups}). We recall that
$\gamma_{\psi}(xy)=\gamma_{\psi}(x)\gamma_{\psi}(y)(x,y)_2$, and if $a\in F^*$ and $\psi_a(x)=\psi(ax)$, $\gamma_{\psi_a}(x)=(a,x)_2 \gamma_{\psi}(x)$. Moreover, $\gamma_{\psi_{a}}=\gamma_{\psi_{b}}$ if and only if $ab^{-1}\in{\Fsquares}$. Also $\gamma_{\psi}(x^2)=1$ and $\gamma_{\psi}(x^{-1})=\gamma_{\psi}(x)$. See the appendix of Rao \cite{Rao}.

\section{Local theory}\label{section:local nonarchimedean theory}

\subsection{The double cover of $G_{\glnidx}(F)$}\label{subsection:the double cover of G_n}
\subsubsection{Definition of the $r$-fold cover}\label{subsubsection:Definition of the cover}
Let $F$ be a local field of characteristic $0$ and let $r>1$ be an integer. Assume that $F$ contains all the $r$-th roots of unity, i.e. $|\mu_r|=r$. Let $\cover{G'}_{\glnidx+1}(F)$ be the $r$-fold cover of $G'_{\glnidx+1}(F)$, constructed by Matsumoto \cite{Mats} using $c=(,)_{r}^{-1}$ as the Steinberg symbol. The group $\cover{G'}_{\glnidx+1}(F)$ fits into an exact sequence
\begin{align*}
1\rightarrow{\mu_r}\xrightarrow{\iota} \cover{G'}_{\glnidx+1}(F)\xrightarrow{p} G'_{\glnidx+1}(F)\rightarrow 1.
\end{align*}
Let $\cover{G}_{\glnidx}(F)=p^{-1}(G_{\glnidx}(F))$, we have an exact sequence
\begin{align*}
1\rightarrow{\mu_r}\xrightarrow{\iota} \cover{G}_{\glnidx}(F)\xrightarrow{p} G_{\glnidx}(F)\rightarrow 1
\end{align*}
and $\cover{G}_{\glnidx}(F)$ is an $r$-fold cover of $G_{\glnidx}(F)$.

Although most of this work will focus on the double cover, we start with the more general setting. From Section~\ref{subsection:Subgroups of the torus} onward, except for Section~\ref{subsection:the double cover of G_n Adele}, we will assume $r=2$.

Banks, Levy and Sepanski \cite{BLS} gave an explicit description of a section $\mathfrak{s}$ of $G'_{\glnidx+1}(F)$ and used it to construct a $2$-cocycle $\sigma_{G'_{\glnidx+1}}$, representing the cohomology class of $\cover{G'}_{\glnidx+1}(F)$ in $H^2(G'_{\glnidx+1}(F),\mu_r)$. We obtain a section of $G_{\glnidx}(F)$ and a $2$-cocycle by restricting $\mathfrak{s}$ and $\sigma_{G'_{\glnidx+1}}$.

We briefly describe the construction and results of \cite{BLS}.
For $\alpha\in{\Sigma}_{G'_{\glnidx+1}}$, let $\mathcal{U}_{\alpha}$ be the corresponding root subgroup. Fix an isomorphism
$n_{\alpha}:F\rightarrow \mathcal{U}_{\alpha}$ (based on an explicit decomposition of the Chevalley algebra corresponding to
$G'_{\glnidx+1}(F)$). For $x\in F^*$, put $w_{\alpha}(x)=n_{\alpha}(x)n_{-\alpha}(-x^{-1})n_{\alpha}(x)$.
Note that $\alpha^{\vee}(x)=h_{\alpha}(x)$ in the notation of \cite{BLS} ($h_{\alpha}(x)=w_{\alpha}(x)w_{\alpha}(1)^{-1}$).
If $\alpha'\in{\Sigma}_{G'_{\glnidx+1}}$, let $\langle\alpha,{\alpha'}^{\vee}\rangle$ be the standard pairing between characters and cocharacters, where ${\alpha'}^{\vee}$ is the coroot corresponding to $\alpha'$.

Given $\alpha\in{\Sigma}_{G'_{\glnidx+1}}$, let $n_{\alpha}^*:F\rightarrow \cover{G'}_{\glnidx+1}(F)$ be the canonical lift of Steinberg \cite{Steinberg}. Using these lifts one defines $w_{\alpha}^*(x)$ and ${\alpha^{\vee}}^*(x)$: $w_{\alpha}^*(x)=n_{\alpha}^*(x)n_{-\alpha}^*(-x^{-1})n_{\alpha}^*(x)$ and ${\alpha^{\vee}}^*(x)=w_{\alpha}^*(x)w_{\alpha}^*(1)$.

Let $\alpha\in\Delta_{G'_{\glnidx+1}}$. For any $x,y\in F^*$, define $c_{\alpha}(x,y)=c(x,y)$ if $\glnidx>0$ and $\alpha$ is a long root, otherwise ($\glnidx=0$ or $\alpha=\alpha_{\glnidx+1}$) put $c_{\alpha}(x,y)=c(x^2,y)$.
Also set $w_{\alpha}=w_{\alpha}(-1)$ and
$w_{\alpha}^*=w_{\alpha}^*(-1)$.

\begin{remark}\label{remark:minor deviation from BLS in case n=0}
In the case $\glnidx=0$ there is only one root $\alpha_1$, which is by definition a long root. In this case in \cite{BLS} it was defined that $c_{\alpha}(x,y)=c(x,y)$. The present definition seems to be correct because it preserves the compatibility with restriction to ``standard" subgroups, see Section~\ref{subsubsection:Restriction to standard subgroups}.
\end{remark}

Following Matsumoto \cite{Mats}, Banks, Levi and Sepanski described a list of identities in the cover group describing the multiplication laws between the elements. The following partial list of identities in $\cover{G'}_{\glnidx+1}(F)$ (\cite{BLS} Section~1) will be used repeatedly in computations:
\begin{align}
&{\alpha^{\vee}}^*(x){\alpha^{\vee}}^*(y)=\iota(c_{\alpha}(x,y)){\alpha^{\vee}}^*(xy),\label{eq:torus alpha and alpha}\\
&{\alpha^{\vee}}^*(x){{\alpha'}^{\vee}}^*(y)=\iota(c_{\alpha}(x,y^{\langle\alpha,{\alpha'}^{\vee}\rangle})){{\alpha'}^{\vee}}^*(y){\alpha^{\vee}}^*(x),\label{eq:torus alpha and alpha'}\\ \label{eq:weyl element and torus}
&w_{\alpha}^*{{\alpha'}^{\vee}}^*(x){w_{\alpha}^*}^{-1}={\alpha^{\vee}}^*(x^{-\langle\alpha,{\alpha'}^{\vee}\rangle}){{\alpha'}^{\vee}}^*(x),
\\\label{eq:weyl element inverse}
&{w_{\alpha}^*}^{-1}=w_{\alpha}^*{\alpha^{\vee}}^*(-1)\iota(c_{\alpha}(-1,-1))^{r-1},
\\ \label{eq:weyl element and unipotent of the same root}
&{w_{\alpha''}^*}(x)n_{\alpha''}^*(z){w_{\alpha''}^*}(x)^{-1}=n_{-\alpha''}^*(-x^{-2}z),
\\ \label{eq:weyl element and unipotent}
&w_{\alpha}^*n_{\alpha}^*(x)=n_{\alpha}^*(-x^{-1}){\alpha^{\vee}}^*(x^{-1})
n_{-\alpha}^*(x^{-1}),\\
\label{eq:torus and unipotent}
&n_{\alpha'''}^*(z){{\alpha}^{\vee}}^*(y)={{\alpha}^{\vee}}^*(y)n_{\alpha'''}^*(zy^{-\langle\alpha''',{\alpha'}^{\vee}\rangle}),\\
\label{eq:weyl simple and unipotent positive}
&{w_{\alpha}^*}^{-1}n_{\alpha'''}^*(z)w_{\alpha}^*=n_{\mathbf{s}_{\alpha}\alpha'''}(d_{\alpha,\alpha'''}z).
\end{align}
Here $\alpha,\alpha'\in \Delta_{G'_{\glnidx+1}}$, $x,y\in F^*$, $z\in F$, $\alpha''\in{\Sigma}_{G'_{\glnidx+1}}$, $\alpha'''\in{\Sigma}_{G'_{\glnidx+1}}^+$; In \eqref{eq:weyl simple and unipotent positive} $\alpha\ne\alpha'''$, $\mathbf{s}_{\alpha}$ is the reflection along $\alpha$ in the Weyl group of $G'_{\glnidx+1}$ and $d_{\alpha,\alpha'''}=\pm1$ depending only on $G'_{\glnidx+1}(F)$. 

Banks, Levy and Sepanski \cite{BLS} defined the following section $\mathfrak{s}$ of $G'_{\glnidx+1}(F)$. First, it is the only splitting of $N_{\glnidx+1}(F)$ and satisfies $\mathfrak{s}(n_{\alpha}(x))=n_{\alpha}^*(x)$ for all $\alpha\in{\Sigma}_{G_{\glnidx+1}'}^+$. 
On $T_{\glnidx+1}(F)$ it satisfies 
\begin{align}\label{eq:section of BLS on torus}
\mathfrak{s}(\prod_{i=1}^{\glnidx+1}\alpha_i^{\vee}(t_i))={\alpha_{\glnidx+1}^{\vee}}^{*}(t_{\glnidx+1}){\alpha_{\glnidx}^{\vee}}^{*}(t_{\glnidx})\cdot\ldots\cdot{\alpha_{1}^{\vee}}^{*}(t_{1})\prod_{i=1}^{\glnidx+1}\iota(c_{\alpha_i}(t_i,t_i)).
\end{align}
Let $\mathfrak{W}_{\glnidx+1}$ be the set of elements $w=w_{\alpha_{i_1}}w_{\alpha_{i_2}}\ldots w_{\alpha_{i_k}}$, where $\alpha_{i_1},\ldots,\alpha_{i_k}\in\Delta_{G'_{\glnidx+1}}$ and $w$ is assumed to be a reduced expression. For such $w$, put $l(w)=k$. As sets, $\mathfrak{W}_{\glnidx+1}\isomorphic W_{\glnidx+1}$. 
We use boldface to denote the elements of $W_{\glnidx+1}$. If $\mathbf{w}\in W_{\glnidx+1}$, its representative in $\mathfrak{W}_{\glnidx+1}$ is $w$.

The section $\mathfrak{s}$ is extended to $G'_{\glnidx+1}(F)$ with the following properties (\cite{BLS} Lemmas~2.2 and 2.3):
\begin{align}\label{eq:props of BLS section image on a simple reflection}
&\mathfrak{s}(w_{\alpha})=w_{\alpha}^*,\\ \label{eq:props of BLS section separate W if length is additive}
&\mathfrak{s}(ww')=\mathfrak{s}(w)\mathfrak{s}(w'),\\ \label{eq:props of BLS section separate N T and W}
&\mathfrak{s}(utwu')=\mathfrak{s}(u)\mathfrak{s}(t)\mathfrak{s}(w)\mathfrak{s}(u').
\end{align}
Here $\alpha\in\Delta_{G'_{\glnidx+1}}$, $w,w'\in\mathfrak{W}_{\glnidx+1}$ satisfy $l(ww')=l(w)+l(w')$, $u,u'\in N_{G'_{\glnidx+1}}(F)$ and $t\in T_{\glnidx+1}(F)$.

Finally, the $2$-cocycle $\sigma=\sigma_{G'_{\glnidx+1}}$ is defined by $\sigma(g,g')=\mathfrak{s}(g)\mathfrak{s}(g')\mathfrak{s}(gg')^{-1}$. The image of $\sigma$ belongs to
 $\iota(\mu_r)$ but we implicitly compose it with $\iota^{-1}$. 
The following formulas hold (\cite{BLS} Proposition~2.4):
\begin{align}
&\sigma(ugu'',g'u')=\sigma(g,u''g'),\label{eq:cocycle move N around}\\
&\sigma(t,w)=1,\label{eq:cocycle T and W}
\end{align}
where $u,u',u''\in N_{G'_{\glnidx+1}}(F)$, $t\in T_{\glnidx+1}(F)$ and $w\in\mathfrak{W}_{\glnidx+1}$.

We pull back the character $\Upsilon$ to a nongenuine character of $\cover{G}_{\glnidx}(F)$, still denoted $\Upsilon$.


For any $g,g'\in G'_{\glnidx+1}(F)$ denote $[g,g']_{\sigma}=\sigma(g,g')\sigma(g',g)^{-1}$. Let
$x,x'\in \cover{G'}_{\glnidx+1}(F)$ be with $p(x)=g$ and $p(x')=g'$. If $[g,g']=1$, then $\mathfrak{s}(gg')=\mathfrak{s}(g'g)$ and
\begin{align}\label{eq:cocycle trick}
[x,x']=\mathfrak{s}(g)\mathfrak{s}(g')\mathfrak{s}(g)^{-1}\mathfrak{s}(g')^{-1}=[g,g']_{\sigma}.
\end{align}
In particular $x$ and $x'$ commute if and only if $[g,g']=1$ and $[g,g']_{\sigma}=1$.

The following claim describes the value of the cocycle on the torus.
\begin{claim}\label{claim:cocycle of spin on the torus}
Let $t=\prod_{i=1}^{\glnidx+1}\alpha_i^{\vee}(t_i)$, $t'=\prod_{i=1}^{\glnidx+1}\alpha_i^{\vee}(t'_i)$. Then
\begin{align*}
\sigma(t,t')=c(t_{\glnidx+1}^2,t_{\glnidx+1}')c(t_{\glnidx},{t'}_{\glnidx+1}^{-2})\prod_{i=1}^{\glnidx}c(t_i,t'_i)\prod_{i=1}^{\glnidx-1}c(t_i,{t'}_{i+1}^{-1}).
\end{align*}
In particular if $\glnidx=0$, $\sigma(t,t')=c(t_{1}^2,t_{1}')$ and for $\glnidx=1$,
$\sigma(t,t')=c(t_{2}^2,t_{2}')c(t_{1},{t'}_{2}^{-2})c(t_1,t_1')$.
\end{claim}
\begin{proof}
The proof follows from \eqref{eq:section of BLS on torus}, \eqref{eq:torus alpha and alpha}, \eqref{eq:torus alpha and alpha'} and the fact that for $1\leq i<i'\leq\glnidx+1$, if $i<i'-1$ then $\langle\alpha_i,\alpha_{i'}^{\vee}\rangle=0$; if $i'\leq\glnidx$,
$\langle\alpha_{i'-1},\alpha_{i'}^{\vee}\rangle=-1$; and $\langle\alpha_{\glnidx},\alpha_{\glnidx+1}^{\vee}\rangle=-2$.
\end{proof}

Note that a similar construction works for $GSpin_{2\glnidx}$. As explained in Section~\ref{subsection:the groups} this group can be regarded as a subgroup of $Spin_{2(\glnidx+1)}$, then the cover $\cover{GSpin_{2\glnidx}}(F)$ can be obtained by restriction from $\cover{Spin_{2(\glnidx+1)}}(F)$. Here $GSpin_{2\glnidx}$ will appear a priori as a subgroup of some $G_{m}$, so the cover is obtained by restricting $\cover{G}_{m}(F)$.

\subsubsection{Restriction to ``standard" subgroups}\label{subsubsection:Restriction to standard subgroups}
Let $H$ be a simple simply-connected algebraic $F$-group, which is $F$-split. Let $\Delta\subset\Delta_{G'_{\glnidx+1}}$. Further assume that $H(F)$ is the subgroup of $G'_{\glnidx+1}(F)$ generated by the elements $\alpha^{\vee}(x)$, $w_{\alpha}$ and $n_{\alpha'}(y)$, where $\alpha\in\Delta$; $x\in F^*$; $\alpha'\in {\Sigma}_{G'_{\glnidx+1}}^+$ is spanned by the roots in $\Delta$; and $y\in F$. Then $H(F)$ was called a ``standard" subgroup of $G'_{\glnidx+1}(F)$ in \cite{BLS}. 
According to \cite{BLS} (Lemma~2.5 and its proof), the restriction of $\mathfrak{s}$ to $H(F)$ gives the section defined by \cite{BLS} on $H(F)$ and $\sigma|_{H(F)\times H(F)}$ is the $2$-cocycle of \cite{BLS} on $H(F)$. Moreover, if $H_1(F),\ldots,H_k(F)$ is a collection of standard subgroups that are mutually commuting and $\sigma_i$ is the $2$-cocycle on $H_i(F)$, by Theorem~2.7 of \cite{BLS},
\begin{align*}
\sigma(g_1\ldots g_k,g_1'\ldots g_k')=\prod_{i=1}^k\sigma_i(g_i,g_i').
\end{align*}
This property was used in \cite{BLS} to define their block-compatible cocycle for $\GLF{\glnidx}{F}$, see Section~\ref{subsubsection:Block-compatibility}.

For example if $\glnidx>0$, according to the embedding $M_{\glnidx-1}<G_{\glnidx}$, the group $G_1'(F)$ is the standard subgroup generated by $\{\alpha_{\glnidx+1}^{\vee}(x),w_{\alpha_{\glnidx+1}},n_{\alpha_{\glnidx+1}}(y)\}$. 
If $\sigma_{G'_1}$ is the cocycle on $G'_1(F)$, 
\begin{align*}
\sigma_{G'_1}(\alpha_{\glnidx+1}^{\vee}(t_1),\alpha_{\glnidx+1}^{\vee}(t_1'))=
\sigma(\alpha_{\glnidx+1}^{\vee}(t_1),\alpha_{\glnidx+1}^{\vee}(t_1'))=c_{\alpha_{\glnidx+1}}(t_1,t_1').
\end{align*}
Since $\glnidx>0$,
$c_{\alpha_{\glnidx+1}}(t_1,t_1')=c(t_1^2,t_1')$. Therefore, when $\glnidx=0$ we must define
$c_{\alpha_1}(x,y)=c(x^2,y)$ (see Remark~\ref{remark:minor deviation from BLS in case n=0}). 
\subsubsection{Splitting of the cover for $r=2$ and $\glnidx\leq1$}\label{subsubsection:Splitting of the double cover for n<=1}
We have the following minimal cases.
\begin{claim}\label{claim:cover minimal cases}
Assume $r=2$. For $\glnidx\leq1$, the cover $\cover{G}_{\glnidx}(F)$ splits under $\mathfrak{s}$. Moreover, for $\glnidx=0$ the restriction of $\sigma$ to $G_0(F)\times G_0(F)$ is trivial.
\end{claim}
\begin{proof}
In the case $\glnidx=0$, $\sigma|_{G_0(F)\times G_0(F)}=1$ because $c_{\alpha_1}(x,y)=c(x^2,y)=1$.
Assume $\glnidx=1$. Since $G_1=\GL{2}$, if the cocycle is nontrivial, then it is equal up to a coboundary to one of the cocycles 
defined by Kazhdan and Patterson \cite{KP}, all of which are nontrivial on $SL_2(F)$.
However, $\sigma|_{G_1'\times G_1'}=\sigma_{G_1'}$ (by Section~\ref{subsubsection:Restriction to standard subgroups})
and $\sigma_{G_1'}=1$, as can be seen 
using \eqref{eq:torus alpha and alpha}-\eqref{eq:torus and unipotent}, note that we have only one simple root $\alpha_2$ and in this case $c_{\alpha_2}=1$.
\end{proof}
\subsubsection{A splitting of the hyperspecial subgroup}\label{subsubsection:A splitting of the hyperspecial subgroup}
Assume $|r|=1$ in $F$ (in particular, $F$ is $p$-adic) and $q>3$. By Moore \cite{Moore} (p.~54-56), there is a unique splitting $\kappa$ of $G_{\glnidx+1}'(\RingOfIntegers)$, which is in particular a splitting of $K(=G_{\glnidx}(\RingOfIntegers))$. Denote $K^*=\kappa(K)$.
As in \cite{KP} (Proposition~0.1.3), we have the following relation between $\kappa$ and $\mathfrak{s}$.
\begin{claim}\label{claim:relations section s and section kappa}
The sections $\mathfrak{s}$ and $\kappa$ agree on $K\cap N_{\glnidx}(F)$, $K\cap T_{\glnidx+1}(F)$ and $\mathfrak{W}_{\glnidx}$ ($\mathfrak{W}_{\glnidx}\subset K$).
\end{claim}
\begin{proof}
Let $\alpha\in\Delta_{G'_{\glnidx+1}}$. First assume $\alpha\ne\alpha_{\glnidx+1}$ and let $SL_2$ be embedded in $G_{\glnidx+1}'$ along $\alpha$. Let $\sigma_{SL_2}$ be the $2$-cocycle of \cite{BLS} on $SL_2(F)$. According to \cite{BLS} (Corollary~3.8 and Lemma~2.5), $\sigma_{SL_2}$ is the cocycle of Kubota \cite{Kubota} on $SL_2(F)$. Kubota \cite{Kubota2} (Theorem~2) proved that the mapping $\gamma:SL_2(\RingOfIntegers)\rightarrow\mu_r$, which is $1$ on
$\left(\begin{smallmatrix}*&*\\x&y\end{smallmatrix}\right)$ if $|x|\in \{0,1\}$ and otherwise equals $c(x,y)$, satisfies $\sigma_{SL_2}(k,k')=\gamma(k)\gamma(k')\gamma(kk')^{-1}$.

Again by \cite{BLS} (Lemma~2.5), $\mathfrak{s}(k)\mathfrak{s}(k')\mathfrak{s}(kk')^{-1}=\sigma_{SL_2}(k,k')$ whence
$\varphi(k)=\mathfrak{s}(k)/\iota(\gamma(k))$ is a splitting of $SL_2(\RingOfIntegers)$. Hence $\varphi$ coincides with $\kappa$ on $SL_2(\RingOfIntegers)$ (\cite{Moore} Lemma~11.1). Whenever $\gamma(k)=1$, $\mathfrak{s}(k)=\kappa(k)$ and in particular $\mathfrak{s}$ and $\kappa$ agree on $n_{\alpha}(\RingOfIntegers)$, $\alpha^{\vee}(\RingOfIntegers^*)$, and $w_{\alpha}$.

If $\alpha=\alpha_{\glnidx+1}$ and $r>2$, restriction to $SL_2(F)$ gives a nontrivial cover (of order $r/\gcd(2,r)$) and the proceeding discussion applies (with $\gamma$ defined using $c(,)^2$ instead of $c(,)$). If $r=2$, $\mathfrak{s}|_{SL_2(F)}$ is a homomorphism ($\sigma_{G_1'}=1$, see the proof of Claim~\ref{claim:cover minimal cases}),
hence $\mathfrak{s}|_{SL_2(\RingOfIntegers)}=\kappa|_{SL_2(\RingOfIntegers)}$.

The section $\kappa$ is in particular a splitting of $K\cap T_{\glnidx+1}(F)$ and the same holds for $\mathfrak{s}$
(because $c$ is trivial on $\RingOfIntegers^*\times\RingOfIntegers^*$). Writing $t=\prod_{i=1}^{\glnidx+1}\alpha_i^{\vee}(t_i)\in K\cap T_{\glnidx+1}(F)$ ($t_i\in\RingOfIntegers^*$) and noting that
$\mathfrak{s}(\alpha_i^{\vee}(t_i))=\kappa(\alpha_i^{\vee}(t_i))$ for each $i$, since $\alpha_i^{\vee}(t_i)$ belongs to the
copy of $SL_2(\RingOfIntegers)$ along $\alpha_i\in\Delta_{G'_{\glnidx+1}}$, one deduces $\kappa(t)=\mathfrak{s}(t)$.

Regarding $\mathfrak{W}_{\glnidx}$, according to \eqref{eq:props of BLS section separate W if length is additive} it is enough to show $\kappa(w_{\alpha})=\mathfrak{s}(w_{\alpha})$ for $\alpha\in\Delta_{G_{\glnidx}}$, which holds because $w_{\alpha}$ belongs to a suitable copy of $SL_2(\RingOfIntegers)$.

It remains to consider $\mathcal{U}_{\alpha'}$ where $\alpha'\in\Sigma_{G_{\glnidx}}^+$. Let $y\in \RingOfIntegers$ and take $\alpha\in\Delta_{G'_{\glnidx+1}}$, $x\in\RingOfIntegers$ and $w\in \mathfrak{W}_{\glnidx}$ such that $w^{-1}n_{\alpha}(x)w=n_{\alpha'}(y)$. The fact that $\kappa$ is a splitting of $K$ and the assertions already proved imply
$\kappa(w^{-1}n_{\alpha}(x)w)=\mathfrak{s}(w)^{-1}\mathfrak{s}(n_{\alpha}(x))\mathfrak{s}(w)$. Then
\eqref{eq:props of BLS section image on a simple reflection}, \eqref{eq:props of BLS section separate W if length is additive} and \eqref{eq:weyl simple and unipotent positive} give $\mathfrak{s}(w)^{-1}\mathfrak{s}(n_{\alpha}(x))\mathfrak{s}(w)=n_{\alpha'}^*(y)$, whence
$\mathfrak{s}(n_{\alpha'}(y))=\kappa(n_{\alpha'}(y))$. Note that we actually apply \eqref{eq:weyl simple and unipotent positive} repeatedly, according to the decomposition of $w$ into a product of simple reflections $w_{\alpha_i}$ ($\alpha_i\in\Delta_{G_{\glnidx}}$), and after $j$ conjugations, if we have to conjugate $n_{\alpha'''}^*(x_j)$ by $w_{\alpha_i}^*$, then $\alpha'''\in\Sigma_{G_{\glnidx}}^+\setdifference\Delta_{G_{\glnidx}} $ whence $\alpha_i\ne\alpha'''$ (this is needed for \eqref{eq:weyl simple and unipotent positive}).
\end{proof}

\subsubsection{Block-compatibility}\label{subsubsection:Block-compatibility}
As mentioned in the introduction, the $2$-cocycle defined on $SL_{\glnidx+1}(F)$ was used in \cite{BLS} to define a $2$-cocycle
$\sigma_{\glnidx}$ for $\GLF{\glnidx}{F}$. Specifically, they defined
\begin{align*}
\sigma_{\glnidx}(b,b')=c(\det{b},\det{b'})^{-1}\sigma_{SL_{\glnidx+1}}(diag(b,\det{b}^{-1}),diag(b',\det{b'}^{-1}))\qquad (b,b'\in\GLF{\glnidx}{F}).
\end{align*}
If $\glnidx=1$, $\sigma_{\glnidx}=1$.
Their cocycle is block-compatible (\cite{BLS} Theorem~3.11), in the sense that for any $g_i,g_i'\in\GLF{\glnidx_i}{F}$, $1\leq i\leq l$,
\begin{align*}
\sigma_{\glnidx}(diag(g_1,\ldots,g_l),diag(g_1',\ldots,g_l'))=\prod_{i=1}^l\sigma_{\glnidx_i}(g_i,g_i')\prod_{i<j}c(\det{g_i},\det{g_j'})^{-1}.
\end{align*}

We seek similar block-compatibility. We begin with the following lemma, which encapsulates several arguments of \cite{BLS} and summarizes a list of properties which, when satisfied by a pair of subgroups, implies a certain plausible block formula.
\begin{lemma}\label{lemma:block compatibility on commuting pair}
Assume $H_1$ and $H_2$ are two subgroups of $G'_{\glnidx+1}(F)$ with the following properties.
\begin{enumerate}[leftmargin=*]
\item \label{item:Iwasawa item} For each $i$ and $h\in H_i$, there is a Bruhat decomposition $h=utwv$ with $u,v\in N_{G'_{\glnidx+1}}(F)$, $t\in T_{\glnidx+1}(F)$ and $w\in\mathfrak{W}_{\glnidx+1}$, and such that $u,t,w,v\in H_i$.
\item \label{item:commute item} The subgroups $H_1$ and $H_2$ commute. 
\item \label{item:w_1 and w_2 item} If $w_1\in H_1\cap \mathfrak{W}_{\glnidx+1}$ and $w_2\in H_2\cap \mathfrak{W}_{\glnidx+1}$, $l(w_1w_2)=l(w_1)+l(w_2)$.
\item \label{item:w_1 and t_2 item} If $i\in\{1,2\}$, $w\in H_i\cap \mathfrak{W}_{\glnidx+1}$ and $t'\in H_{3-i}\cap T_{\glnidx+1}(F)$, then $\sigma(w,t')=1$.
\item \label{item:t_1 and t_2 item} If $t_1\in H_1\cap T_{\glnidx+1}(F)$ and $t_2\in H_2\cap T_{\glnidx+1}(F)$, then $\sigma(t_2,t_1)=1$.
\end{enumerate}
Then for any $h_1,h_1'\in H_1$ and $h_2,h_2'\in H_2$,
\begin{align*}
\sigma(h_1h_2,h_1'h_2')=\sigma(h_1,h_1')\sigma(h_2,h_2')\sigma(h_1,h_2').
\end{align*}
Moreover, if $h_1=u_1t_1w_1v_1$ and $h_2'=u_2't_2'w_2'v_2'$, with $u_1,v_1,u_2',v_2'\in N_{G'_{\glnidx+1}}(F)$,
$t_1,t_2'\in T_{\glnidx+1}(F)$, $w_1,w_2'\in\mathfrak{W}_{\glnidx+1}$, $u_1,t_1,w_1,v_1\in H_1$ and $u_2',t_2',w_2',v_2'\in H_2$, then
\begin{align*}
\sigma(h_1,h_2')=\sigma(t_1,t_2').
\end{align*}
\end{lemma}
\begin{proof}
Let $i\in\{1,2\}$, $t\in H_i\cap T_{\glnidx+1}(F)$, $w\in H_i\cap \mathfrak{W}_{\glnidx+1}$,
$t'\in H_{3-i}\cap T_{\glnidx+1}(F)$ and $w'\in H_{3-i}\cap \mathfrak{W}_{\glnidx+1}$. Note that by \eqref{item:commute item} and \eqref{eq:props of BLS section separate N T and W},
\begin{align*}
\sigma(w,t')=\mathfrak{s}(w)\mathfrak{s}(t')\mathfrak{s}(wt')^{-1}
=\mathfrak{s}(w)\mathfrak{s}(t')\mathfrak{s}(t'w)^{-1}=
\mathfrak{s}(w)\mathfrak{s}(t')\mathfrak{s}(w)^{-1}\mathfrak{s}(t')^{-1},
\end{align*}
whence \eqref{item:w_1 and t_2 item} implies $\mathfrak{s}(w)\mathfrak{s}(t')\mathfrak{s}(w)^{-1}=\mathfrak{s}(t')$.
Thus
\begin{align}\label{eq:lemma block compatibility sigma last comp 1_1}
\sigma(tw,t')=&\mathfrak{s}(tw)\mathfrak{s}(t')\mathfrak{s}(twt')^{-1}=
\mathfrak{s}(t)\mathfrak{s}(w)\mathfrak{s}(t')\mathfrak{s}(tt'w)^{-1}\\\nonumber
=&\mathfrak{s}(t)\mathfrak{s}(w)\mathfrak{s}(t')\mathfrak{s}(w)^{-1}\mathfrak{s}(tt')^{-1}
=\mathfrak{s}(t)\mathfrak{s}(t')\mathfrak{s}(tt')^{-1}=\sigma(t,t'),
\end{align}
where we also used \eqref{eq:props of BLS section separate N T and W} and \eqref{item:commute item}. Similarly,
\begin{align}\label{eq:lemma block compatibility sigma last comp}
\sigma(twt',w')=&\sigma(tt'w,w')=\mathfrak{s}(tt')\mathfrak{s}(w)\mathfrak{s}(w')\mathfrak{s}(tt'ww')^{-1}\\\nonumber
=&\mathfrak{s}(tt')\sigma(w,w')\mathfrak{s}(tt')^{-1}=\mathfrak{s}(tt')\mathfrak{s}(tt')^{-1}=1.
\end{align}
Here we used \eqref{eq:props of BLS section separate W if length is additive}, \eqref{item:commute item} and \eqref{item:w_1 and w_2 item} to deduce $\sigma(w,w')=1$.

Now \eqref{eq:cocycle T and W}, \eqref{eq:2-cocycle property}, \eqref{eq:lemma block compatibility sigma last comp 1_1} and \eqref{eq:lemma block compatibility sigma last comp} imply
\begin{align*}
\sigma(tw,t'w')=\sigma(tw,t'w')\sigma(t',w')=
\sigma(tw,t')\sigma(twt',w')=\sigma(t,t').
\end{align*}
Let $h\in H_i$ and $h'\in H_{3-i}$. Write $h=utwv$ and $h'=u't'w'v'$, with $u,v\in H_{i}\cap N_{G'_{\glnidx+1}}(F)$ and $u',v'\in H_{3-i}\cap N_{G'_{\glnidx+1}}(F)$ (this is possible by \eqref{item:Iwasawa item}). Then \eqref{eq:cocycle move N around} and \eqref{item:commute item} give
\begin{align*}
\sigma(h,h')=\sigma(twv,u't'w')=\sigma(twu',vt'w')=\sigma(u'tw,t'w'v)=\sigma(tw,t'w')=\sigma(t,t').
\end{align*}
Hence for $i=1$ ($h\in H_1$), we deduce the second assertion, i.e., $\sigma(h_1,h_2')=\sigma(t_1,t_2')$. When $i=2$ ($h\in H_2$, $h'\in H_1$), $\sigma(h,h')=1$ because of \eqref{item:t_1 and t_2 item}. Then exactly as in \cite{BLS} (top of p.~157), a repeated application of \eqref{eq:2-cocycle property} implies $\sigma(h_1h_2,h_1'h_2')=\sigma(h_1,h_1')\sigma(h_2,h_2')\sigma(h_1,h_2')$.
\end{proof}

In the case of $M_k(F)$, it is a product of $H_2=\GLF{k}{F}$ and $H_1=G_{\glnidx-k}(F)$. These subgroups clearly satisfy conditions
\eqref{item:Iwasawa item}-\eqref{item:w_1 and w_2 item} of Lemma~\ref{lemma:block compatibility on commuting pair}. The following claim checks the validity of the other conditions.
\begin{claim}\label{claim:standard M_k computations of torus and Weyl elements}
Let $a=\prod_{i=1}^k\eta_i^{\vee}(a_i)\in T_{\GL{k}}(F)$ and $t=\prod_{i=1}^{\glnidx-k+1}\beta_i^{\vee}(t_i)\in T_{\glnidx-k+1}(F)$ and consider their images in $G_{\glnidx}(F)$, given by \eqref{eq:image of standard A_k in GSpin} and \eqref{eq:image of standard T_n-k+1 in GSpin}. Then
\begin{align*}
&\sigma(a,t)=1,\\
&\sigma(t,a)=\begin{dcases}c(t_1,\det{a}^{-1})&k<\glnidx,\\
c(t_1^2,\det{a}^{-1})&k=\glnidx.\end{dcases}
\end{align*}
In particular condition \eqref{item:t_1 and t_2 item} of Lemma~\ref{lemma:block compatibility on commuting pair} holds.
Further let $w_1\in H_1\cap\mathfrak{W}_{\glnidx+1}$ and $w_2\in H_2\cap\mathfrak{W}_{\glnidx+1}$. Then $\sigma(w_1,a)=1$ and if $r=2$, $\sigma(w_2,t)=1$ whence condition \eqref{item:w_1 and t_2 item} also holds.
\end{claim}
\begin{proof}
The computations of $\sigma(a,t)$ and $\sigma(t,a)$ are immediate from \eqref{eq:image of standard A_k in GSpin}, \eqref{eq:image of standard T_n-k+1 in GSpin} and Claim~\ref{claim:cocycle of spin on the torus}. The equality $\sigma(w_1,a)=1$ follows from the fact that for any $k+2\leq i\leq\glnidx+1$ and $1\leq j\leq k$, $\langle\alpha_i,\alpha_j^{\vee}\rangle=0$ and then \eqref{eq:weyl element and torus} implies that $w_{\alpha_i}^*$ commutes with ${\alpha_j^{\vee}}^*(x)$ for any $x\in F^*$.

Regarding $\sigma(w_2,t)$, we show that if $r=2$,
\begin{align}\label{eq:section equality w and t in claim M_k standard computations}
\mathfrak{s}(w_2)\mathfrak{s}(t)\mathfrak{s}(w_2)^{-1}=
\mathfrak{s}(t).
\end{align}
Then the result follows from \eqref{eq:props of BLS section separate N T and W} (note that $w_2t=tw_2$). It is enough to establish this
for $w=w_{\alpha_i}^*$ with $2\leq i\leq k$.

First assume $k<\glnidx$. Using \eqref{eq:weyl element and torus} and then \eqref{eq:torus alpha and alpha} and \eqref{eq:torus alpha and alpha'},
\begin{align*}
w{\alpha_{i+1}^{\vee}}^{*}(t_1){\alpha_{i}^{\vee}}^{*}(t_1){\alpha_{i-1}^{\vee}}^{*}(t_1)w^{-1}=&
{\alpha_{i}^{\vee}}^{*}(t_1){\alpha_{i+1}^{\vee}}^{*}(t_1){\alpha_{i}^{\vee}}^{*}(t_1^{-2}){\alpha_{i}^{\vee}}^{*}(t_1){\alpha_{i}^{\vee}}^{*}(t_1){\alpha_{i-1}^{\vee}}^{*}(t_1)\\
=&{\alpha_{i+1}^{\vee}}^{*}(t_1){\alpha_{i}^{\vee}}^{*}(t_1){\alpha_{i-1}^{\vee}}^{*}(t_1)
\iota(c_{\alpha_i}(t_1,t_1^{-1})^2c_{\alpha_i}(t_1^{-2},t_1))\\
=&{\alpha_{i+1}^{\vee}}^{*}(t_1){\alpha_{i}^{\vee}}^{*}(t_1){\alpha_{i-1}^{\vee}}^{*}(t_1).
\end{align*}
This computation implies \eqref{eq:section equality w and t in claim M_k standard computations}.
For $k=\glnidx$ the computation is similar (one distinguishes between the cases $i<\glnidx$ and $i=\glnidx$).
\end{proof}

\begin{remark}\label{remark:another embedding}
For a general $r$, the cocycle does not seem to satisfy a convenient formula on $M_k(F)\times M_k(F)$. One can remedy this by
defining another isomorphism of $\GL{k}\times G_{\glnidx-k}$ with $M_k$. Let $SL_{k}$ be generated by the simple roots
$\setof{\alpha_i}{1\leq i\leq k-1}$ and $G'_{\glnidx-k}$ be generated by $\setof{\alpha_i}{k+2\leq i\leq \glnidx+1}$.
If $\{\eta_i^{\vee}\}$ are the cocharacters of $T_{\GL{k}}$ and $\{\theta_i\}$ are the characters of $T_{\glnidx-k+1}$,
map $\eta_i^{\vee}\mapsto\epsilon_{i}^{\vee}-\epsilon_{k+1}^{\vee}$ for $1\leq i\leq k$ and
$\theta_i^{\vee}\mapsto\epsilon_{k+i}^{\vee}$ for $1\leq i\leq \glnidx-k+1$.
This is a twist of the embeddings described in Section~\ref{subsection:properties of GSpin}, by a Weyl element of $G'_{\glnidx+1}$.
\end{remark}

The evaluation of $\sigma$ on $T_{\GL{k}}(F)\times T_{\GL{k}}(F)$ and $T_{\glnidx-k+1}(F)\times T_{\glnidx-k+1}(F)$ provides evidence that the cocycle indeed satisfies certain block-compatibility properties. Claim~\ref{claim:cocycle of spin on the torus} along with \eqref{eq:image of standard A_k in GSpin}-\eqref{eq:image of standard T_n-k+1 in GSpin} show that for any $0\leq k\leq\glnidx$,
\begin{align*}
&\sigma(\prod_{i=1}^k\eta_i^{\vee}(a_i),\prod_{i=1}^k\eta_i^{\vee}(a'_i))=c(\det{a},\det{a'})\prod_{1\leq i<j\leq k}c(a_i,a'_j)^{-1},\\
&\sigma(\prod_{i=1}^{\glnidx-k+1}\beta_i^{\vee}(t_i),\prod_{i=1}^{\glnidx-k+1}\beta_i^{\vee}(t'_i))=
c(t_{\glnidx-k+1}^2,t'_{\glnidx-k+1})c(t_{\glnidx-k},{t'}_{\glnidx-k+1}^{-2})\prod_{i=1}^{\glnidx-k}c(t_i,t'_i)\prod_{i=1}^{\glnidx-k-1}c(t_i,{t'}_{i+1}^{-1}).
\end{align*}
We see that $\sigma(a,a')=c(\det{a},\det{a'})\sigma_{k}(a,a')$ for $a,a'\in T_{\GL{k}}(F)$ and for $t,t'\in T_{\glnidx-k+1}(F)$, $\sigma(t,t')=\sigma_{G'_{\glnidx-k+1}}(t,t')$.

Henceforth until the end of Section~\ref{section:local nonarchimedean theory}, $r=2$.
Lemma~\ref{lemma:block compatibility on commuting pair} and Claim~\ref{claim:standard M_k computations of torus and Weyl elements}
imply that for $b,b'\in\GLF{k}{F}$ and $h,h'\in G_{\glnidx-k}(F)$,
\begin{align*}
\sigma(bh,b'h')=\sigma(b,b')\sigma(h,h')c(\Upsilon(h),\det{b'}).
\end{align*}
The subgroup $\GLF{k}{F}$ (of $M_k(F)$) is contained in the subgroup $SL_{k+1}(F)<G'_{\glnidx+1}(F)$ generated by $\setof{\alpha_i}{1\leq i\leq k}$. If we regard $SL_{k+1}(F)$ as a group of matrices in the standard way, i.e., identify $n_{\epsilon_i-\epsilon_j}(x)$ ($1\leq i<j\leq k+1$, $x\in F$) with the matrix having $1$ on the diagonal, $x$ on the $(i,j)$-th place and $0$ elsewhere, $b$ takes the form $diag(\det{b}^{-1},b)$. Since $SL_{k+1}(F)$ is a standard subgroup (in the sense of \cite{BLS}), the restriction of $\sigma$ to
$SL_{k+1}(F)\times SL_{k+1}(F)$ is just $\sigma_{SL_{k+1}}$ of \cite{BLS} (Section~3). We define
a $2$-cocycle $\sigma_{\GL{k}}$ of $\GLF{k}{F}$ via
\begin{align*}
\sigma_{\GL{k}}(b,b')=\sigma_{SL_{k+1}}(diag(\det{b}^{-1},b),diag(\det{b'}^{-1},b')).
\end{align*}
Here with a minor abuse of notation, we regard the arguments $b$ and $b'$ on the \rhs\ as matrices. This cocycle is related to
$\sigma_{k}$ of \cite{BLS} by
\begin{align*}
\sigma_{\GL{k}}(b,b')=c(\det{b},\det{b'})\sigma_{k}(b',b).
\end{align*}
Since $\sigma(b,b')=\sigma_{SL_{k+1}}(b,b')$, we get
\begin{align*}
\sigma(b,b')=\sigma_{\GL{k}}(b,b').
\end{align*}

Regarding $G_{\glnidx-k}(F)$, as a subgroup of $M_k(F)$ its embedding in $G_{\glnidx}(F)$ is not contained in a standard subgroup of type $G_{\glnidx-k+1}'$ (in contrast with the embedding described by Remark~\ref{remark:another embedding}). However, it is still true that
\begin{align}\label{eq:sigma restricted to G(n-k) is the correct sigma}
\sigma(h,h')=\sigma_{G'_{\glnidx-k+1}}(h,h').
\end{align}
To see this, first note that $G_{\glnidx-k}'$ (as a subgroup of $G_{\glnidx-k}$) is generated by the roots $\setof{\alpha_i}{k+2\leq i\leq \glnidx+1}$, and is therefore a standard subgroup whence
\begin{align*}
\sigma|_{G_{\glnidx-k}'(F)\times G_{\glnidx-k}'(F)}=\sigma_{G_{\glnidx-k+1}'}|_{G_{\glnidx-k}'(F)\times G_{\glnidx-k}'(F)}.
\end{align*}
Moreover, according to the computation on $T_{\glnidx-k+1}(F)\times T_{\glnidx-k+1}(F)$ above,
\begin{align*}
\sigma|_{T_{\glnidx-k+1}(F)\times T_{\glnidx-k+1}(F)}=\sigma_{G'_{\glnidx-k+1}}|_{T_{\glnidx-k+1}(F)\times T_{\glnidx-k+1}(F)}.
\end{align*}
According to Sun \cite{Su} (Proposition~1), the restrictions of a cocycle (defined using a bilinear Steinberg symbol) to the derived group and to the torus determine it uniquely. Thus we conclude \eqref{eq:sigma restricted to G(n-k) is the correct sigma}.

Therefore,
\begin{align}\label{eq:block-compatibility on 2 Levi subgroups}
\sigma(bh,b'h')=\sigma_{\GL{k}}(b,b')\sigma_{G'_{\glnidx-k+1}}(h,h')c(\Upsilon(h),\det{b'}).
\end{align}

Let $Q<G_{\glnidx}$ be a standard parabolic subgroup with a Levi part $M$ isomorphic to $\GL{k_1}\times\ldots\times\GL{k_l}\times G_{\glnidx-k}$, where $k=k_1+\ldots+k_l$. Our standard embedding $M<G_{\glnidx}$ was defined in Section~\ref{subsection:properties of GSpin}. According to \eqref{eq:block-compatibility on 2 Levi subgroups}, for $(b_1,\ldots,b_l,h),(b_1',\ldots,b_l',h') \in M(F)$ ($h,h'\in G_{\glnidx-k}(F)$),
\begin{align}\label{eq:block-compatibility on Levi subgroups}
&\sigma((b_1,\ldots,b_l,h),(b_1',\ldots,b_l',h'))\\\nonumber&\qquad=(\prod_{i=1}^l\sigma_{\GL{k_i}}(b_i,b_i')c(\Upsilon(h),\det{b_i'}))(\prod_{i>j}^lc(\det{b_i},\det{b_j'}))\sigma_{G'_{\glnidx-k+1}}(h,h').
\end{align}


We mention one particularly convenient Levi subgroup, $M_{\glnidx}$. Then
$c(\Upsilon(h),\cdot)=1$ for all $h\in G_0(F)$ (see Section~\ref{subsection:properties of GSpin}). Therefore the subgroups $\CGLF{\glnidx}{F}$ and $\cover{G}_0(F)$ of $\cover{M}_{\glnidx}(F)$ are commuting and $\cover{M}_{\glnidx}(F)$ is simply their direct product with amalgamated $\mu_2$. Also note that (since $r=2$) the cover $\cover{G}_0(F)$ is abelian and splits, $\sigma|_{G_0(F)\times G_0(F)}$ is trivial (Claim~\ref{claim:cover minimal cases}). Then
\begin{align}\label{eq:cocycle convenient coordinates r=2}
\sigma(\prod_{i=1}^{\glnidx}\eta_i^{\vee}(a_i)\beta_1^{\vee}(t_1),\prod_{i=1}^{\glnidx}\eta_i^{\vee}(a_i')\beta_1^{\vee}(t_1'))=c(\det{a},\det{a'})\prod_{1\leq i<j\leq k}c(a_i,a'_j)^{-1}.
\end{align}

\subsubsection{Subgroups of the torus}\label{subsection:Subgroups of the torus}
For $\glnidx>1$, $\cover{T}_{\glnidx+1}(F)$ is not abelian. Its irreducible genuine representations are parameterized
by genuine characters of its center, as follows from an analog of the Stone-von Neumann theorem (\cite{KP,McNamara}). In this section we compute $C_{\cover{T}_{\glnidx+1}(F)}$ and several other subgroups, which will be used in Section~\ref{subsection:Representations local stuff} to study these representations.

The following subgroup of $T_{\glnidx+1}(F)$ will play an important role in constructing irreducible representations of $\cover{T}_{\glnidx+1}(F)$:
\begin{align*}
T_{\glnidx+1}(F)^2&=\setof{\prod_{i=1}^{\glnidx+1}\alpha_i^{\vee}(t_i)}{\forall1\leq i\leq\glnidx,t_i\in{\Fsquares},t_{\glnidx+1}\in F^*}.
\end{align*}
Equality~\eqref{eq:image of convenient coordinates of the torus in GSpin} implies
$T_{\glnidx+1}(F)^2=\setof{\prod_{i=1}^{\glnidx}\eta_i^{\vee}(a_i^2)\beta_1^{\vee}(t_1)}{a_i,t_1\in F^*}$.
\begin{claim}\label{claim:T^2 is split by s}
The subgroup $T_{\glnidx+1}(F)^2$
splits under $\mathfrak{s}$. In particular, any character of $T_{\glnidx+1}(F)^2$ can be extended to a genuine character of $\cover{T}_{\glnidx+1}(F)^2$ (uniquely, using the splitting $\mathfrak{s}$).
\end{claim}
\begin{proof}
This is evident from \eqref{eq:section of BLS on torus}, \eqref{eq:torus alpha and alpha} and \eqref{eq:torus alpha and alpha'}. The point is that ${\alpha_i^{\vee}}^*(x)$ and ${\alpha_j^{\vee}}^*(y)$ commute whenever both $x,y\in {\Fsquares}$ or $\glnidx+1\in\{i,j\}$.
\end{proof}

Next we describe the center of $\cover{G}_{\glnidx}(F)$.
\begin{claim}\label{claim:center of cover}
$C_{\cover{G}_{\glnidx}(F)}=\cover{C}_{G_{\glnidx}(F)}$ and splits under $\mathfrak{s}$.
\end{claim}
\begin{proof}
Let $t=\prod_{i=1}^{\glnidx}\alpha_i^{\vee}(t_1^2)\alpha_{\glnidx+1}^{\vee}(t_1)\in C_{G_{\glnidx}(F)}$ (see \eqref{eq:center of the group}).
Then $[t,g]=1$ for all $g\in G_{\glnidx}(F)$ and we must show $[t,g]_{\sigma}=1$ ($[,]_{\sigma}$ was defined in Section~\ref{subsubsection:Definition of the cover}). Since $[,]_{\sigma}$ is bi-multiplicative, it suffices to consider
$g=t',u,w$ where $t'\in T_{\glnidx+1}(F)$, $u\in N_{\glnidx}(F)$ and $w=w_{\alpha_i}$ with $2\leq i\leq\glnidx+1$. Now $[t,t']_{\sigma}=1$ because of
\eqref{eq:torus alpha and alpha'} and $[t,u]_{\sigma}=1$ by \eqref{eq:cocycle move N around}. Using
\eqref{eq:cocycle trick}, it is enough to
show $\mathfrak{s}(w)\mathfrak{s}(t)\mathfrak{s}(w)^{-1}=\mathfrak{s}(t)$. This follows from
\eqref{eq:torus alpha and alpha}-\eqref{eq:weyl element and torus} and the observation in the proof of Claim~\ref{claim:T^2 is split by s}. Now the splitting follows from Claim~\ref{claim:T^2 is split by s}.
\end{proof}

Recall that by \cite{KP} (Proposition~0.1.1),
\begin{align}\label{eq:the center of GLn}
C_{\CGLF{\glnidx}{F}}=p^{-1}(\setof{\prod_{i=1}^{\glnidx}\eta_i^{\vee}(d)}{d\in (F^*)^{2/\gcd(2,\glnidx+1)}}).
\end{align}
Also set
$T_{\GL{\glnidx}}(F)^2=\setof{\prod_{i=1}^{\glnidx}\eta_i^{\vee}(a_i^2)}{a_i\in F^*}$.
By \cite{KP} (p.~57), $C_{\cover{T}_{\GL{\glnidx}}(F)}=\cover{T}_{\GL{\glnidx}}(F)^2C_{\CGLF{\glnidx}{F}}$.

Note that $C_{\cover{G}_{\glnidx}(F)}<\cover{T}_{\glnidx+1}(F)^2$, in contrast with the case of $\CGLF{\glnidx}{F}$, where
$C_{\CGLF{\glnidx}{F}}$ is contained in $\cover{T}_{\GL{\glnidx}}(F)^2$ only when $\glnidx$ is even. This causes technical difficulties when trying to adapt definitions of metaplectic tensor product from $\GL{\glnidx}$ (\cite{Kable,Mezo,Tk2}) to $G_{\glnidx}$, see Sections~\ref{subsubsection:tensor product}-\ref{subsubsection:tensor discussion} below.

We compute the center of the torus.
\begin{claim}\label{claim:center of the torus}
We have $C_{\cover{T}_{\glnidx+1}(F)}=\cover{T}_{\glnidx+1}(F)^2C_{\CGLF{\glnidx}{F}}$. More specifically,
\begin{align}
\label{claim:constructing maximal abelian subgroup item 1} &C_{\cover{T}_{\glnidx+1}(F)}=p^{-1}(\setof{\prod_{i=1}^{\glnidx}\eta_i^{\vee}(a_i^2d)\beta_1^{\vee}(t_1)}{a_i\in F^*,d\in (F^*)^{2/\gcd(2,\glnidx+1)},t_1\in F^*}).
\end{align}
\end{claim}
\begin{proof}
In general if $b=\prod_{i=1}^{\glnidx}\eta_i^{\vee}(b_i)\beta_1^{\vee}(t_1)$ and $b'=\prod_{i=1}^{\glnidx}\eta_i^{\vee}(b_i')\beta_1^{\vee}(t_1')$, Equality~\eqref{eq:cocycle convenient coordinates r=2} implies
\begin{align}\label{eq:general commutator on torus}
[b,b']_{\sigma}=\prod_{i=1}^{\glnidx}c(b_i,b_i')c(\prod_{i=1}^{\glnidx}b_i,\prod_{i=1}^{\glnidx}b_i').
\end{align}
Set $d=\prod_{i=1}^{\glnidx}b_i^{-1}$. On the one hand, for any $1\leq j\leq\glnidx$ and $z\in F^*$,
$[b,\eta_j^{\vee}(z)]_{\sigma}=c(b_jd^{-1},z)$. Then if $b\in C_{\cover{T}_{\glnidx+1}(F)}$, $b_jd^{-1}\in{\Fsquares}$ whence $b_j=a_j^2d$ for some $a_j\in F^*$. It follows that $d^{\glnidx-1}\in{\Fsquares}$ or equivalently $d^{\glnidx+1}\in{\Fsquares}$ and $d\in(F^*)^{2/\gcd(2,\glnidx+1)}$. Hence $b$ is of the required form. On the other hand, if $d\in(F^*)^{2/\gcd(2,\glnidx+1)}$,
\begin{align*}
[\prod_{i=1}^{\glnidx}\eta_i^{\vee}(a_i^2d)\beta_1^{\vee}(t_1),\eta_j^{\vee}(z)]_{\sigma}=
c(d^{\glnidx+1},z)c(a_j^2,z)\prod_{i=1}^{\glnidx}c(a_i^2,z).
\end{align*}
Because $c(d^{\glnidx+1},z)=c(a_i^2,z)=1$, we obtain $[\prod_{i=1}^{\glnidx}\eta_i^{\vee}(a_i^2d)\beta_1^{\vee}(t_1),\eta_j^{\vee}(z)]_{\sigma}=1$ for all $1\leq j\leq\glnidx$. Since $[\prod_{i=1}^{\glnidx}\eta_i^{\vee}(a_i^2d)\beta_1^{\vee}(t_1),\eta_{\glnidx+1}^{\vee}(z)]_{\sigma}=1$ clearly holds and $[,]_{\sigma}$ is bi-multiplicative, any element on the \rhs\ of \eqref{claim:constructing maximal abelian subgroup item 1} belongs to the center of $\cover{T}_{\glnidx+1}(F)$.

Now $C_{\cover{T}_{\glnidx+1}(F)}=\cover{T}_{\glnidx+1}(F)^2C_{\CGLF{\glnidx}{F}}$ follows immediately from \eqref{eq:the center of GLn}.
\end{proof}
\begin{remark}\label{remark:center of GLn can be ignored for even n}
If $\glnidx$ is even, $C_{\CGLF{\glnidx}{F}}<\cover{T}_{\glnidx+1}(F)^2$ whence $C_{\cover{T}_{\glnidx+1}(F)}=\cover{T}_{\glnidx+1}(F)^2$.
\end{remark}
\begin{remark}\label{remark:center of the torus of SO2n+1}
In the case of the cover of $SO_{2\glnidx+1}(F)$ obtained by restricting the $4$-fold cover of $SL_{2\glnidx+1}(F)$ of Matsumoto \cite{Mats}, we have (\cite{BFG2} Section~6)
\begin{align*}
C_{\cover{T}_{SO_{2\glnidx+1}}(F)}=p^{-1}(\setof{diag(a_1^2,\ldots,a_{\glnidx}^2,1,a_{\glnidx}^{-2},\ldots,a_1^{-2})}{a_i\in F^*}).
\end{align*}
\end{remark}

We describe certain maximal abelian subgroups of $\cover{T}_{\glnidx+1}(F)$. A representation of $C_{\cover{T}_{\glnidx+1}(F)}$ can always be extended to such a subgroup, then induced to a representation of $\cover{T}_{\glnidx+1}(F)$.
\begin{claim}\label{claim:constructing maximal abelian subgroup unramified case}
Assume $|2|=1$ and $q>3$ in $F$.
The subgroup $C(\cover{T}_{\glnidx+1}(F),\cover{T}_{\glnidx+1}(F)\cap K^*)$ (i.e., the centralizer of $\cover{T}_{\glnidx+1}(F)\cap K^*$ in $\cover{T}_{\glnidx+1}(F)$, see Section~\ref{subsection:the groups})
is a maximal abelian subgroup of $\cover{T}_{\glnidx+1}(F)$ containing $\cover{T}_{\glnidx+1}(F)\cap K^*$. We have
\begin{align}\label{eq:constructing maximal abelian subgroup item 3}
C(\cover{T}_{\glnidx+1}(F),\cover{T}_{\glnidx+1}(F)\cap K^*)=C_{\cover{T}_{\glnidx+1}(F)}\cdot(\cover{T}_{\glnidx+1}(F)\cap K^*).
\end{align}
\end{claim}
\begin{proof}
Equality~\eqref{eq:constructing maximal abelian subgroup item 3} follows if we show that
$C(\cover{T}_{\glnidx+1}(F),\cover{T}_{\glnidx+1}(F)\cap K^*)$ is equal to
\begin{align}
&\label{eq:constructing maximal abelian subgroup item 2}
p^{-1}(\setof{\prod_{i=1}^{\glnidx}\eta_i^{\vee}(a_id)\beta_1^{\vee}(t_1)}{a_i\in \RingOfIntegers^*{\Fsquares},d\in \RingOfIntegers^*(F^*)^{2/\gcd(2,\glnidx+1)},t_1\in F^*}).
\end{align}
The proof of this is similar to the proof of Claim~\ref{claim:center of the torus} and omitted.
\end{proof}

It can be useful for applications to have a convenient choice of a maximal abelian subgroup, over any field. We follow the construction of Bump and Ginzburg \cite{BG} (p.~141). Let
\begin{align*}
T_{\GL{\glnidx}}(F)^{\mathm}=\setof{\prod_{i=1}^{\glnidx}\eta_i^{\vee}(a_i)}{a_1,\ldots,a_{\glnidx}\in F^*,a_{\glnidx-2i}^{-1}a_{\glnidx-2i-1}\in {\Fsquares}, 0\leq i<\lfloor\glnidx/2\rfloor}.
\end{align*}
(This is $T^{\mathe}_F$ 
of \cite{BG}.) Then $\cover{T}_{\GL{\glnidx}}(F)^{\mathm}$ is a maximal abelian subgroup. The advantage of this subgroup is that it always contains $\cover{C}_{\GLF{\glnidx}{F}}$, regardless of the parity of $\glnidx$. Similarly,
\begin{claim}\label{claim:constructing a maximal abelian subgroup of the torus in general}
Let
\begin{align*}
T_{\glnidx+1}(F)^{\mathm}=\setof{\prod_{i=1}^{\glnidx}\eta_i^{\vee}(a_i)\beta_1^{\vee}(t_1)}{
a_1,\ldots,a_{\glnidx}\in F^*,a_{\glnidx-2i}^{-1}a_{\glnidx-2i-1}\in {\Fsquares}, 0\leq i<\lfloor\glnidx/2\rfloor,t_1\in F^*}.
\end{align*}
Then $\cover{T}_{\glnidx+1}(F)^{\mathm}$ is a maximal abelian subgroup of $\cover{T}_{\glnidx+1}(F)$ containing $\cover{C}_{\GLF{\glnidx}{F}}$.
\end{claim}
\begin{proof}
If $b$ belongs to an abelian subgroup of $\cover{T}_{\glnidx+1}(F)$ containing $\cover{T}_{\glnidx+1}(F)^{\mathm}$, write $p(b)=\prod_{i=1}^{\glnidx}\eta_i^{\vee}(b_i)\beta_1^{\vee}(t_1)$, then for all $0\leq j<\lfloor\glnidx/2\rfloor$ and $x,y\in F^*$, by \eqref{eq:general commutator on torus},
\begin{align*}
[b,\eta_j^{\vee}(x)\eta_{j+1}^{\vee}(xy^2)]_{\sigma}=c(b_j,x)c(b_j,xy^2)=c(b_jb_{j+1},x).
\end{align*}
Because $c(b_jb_{j+1},x)$ must be equal to $1$, we obtain $b_j^{-1}b_{j+1}\in{\Fsquares}$, whence $b\in \cover{T}_{\glnidx+1}(F)^{\mathm}$.

To see that $\cover{T}_{\glnidx+1}(F)^{\mathm}$ is an abelian subgroup let $t=\prod_{i=1}^{\glnidx}\eta_i^{\vee}(a_i)\beta_1^{\vee}(t_1)$,
$t'=\prod_{i=1}^{\glnidx}\eta_i^{\vee}(a'_i)\beta_1^{\vee}(t'_1)$ belong to $T_{\glnidx+1}(F)^{\mathm}$. If $\glnidx$ is even, $\prod_{i=1}^{\glnidx}a_i\in {\Fsquares}$ and \eqref{eq:general commutator on torus} implies $[t,t']_{\sigma}=\prod_{i=1}^{\glnidx}c(a_i,a_i')$. This equals $1$ because if $a_{j+1}^{-1}a_j,{a'_{j+1}}^{-1}a'_j\in{\Fsquares}$.
\begin{align*}
c(a_j,a'_j)c(a_{j+1},a'_{j+1})=c(a_{j+1}a_{j+1}^{-1}a_j,a'_j)c(a_{j+1},a'_{j+1})=c(a_{j+1},a'_ja'_{j+1})=1.
\end{align*}
If $\glnidx$ is odd, $\prod_{i=2}^{\glnidx}a_i\in {\Fsquares}$ and we proceed similarly.
\end{proof}

\subsection{Principal series representations}\label{subsection:Representations local stuff}

\subsubsection{Representations of the torus}\label{subsubsection:representations of the torus}
We describe the representations of $\cover{T}_{\glnidx+1}(F)$.
Since $r=2$, $\cover{T}_{\glnidx+1}(F)$ is abelian for $\glnidx\leq1$,
as may be seen either by a direct verification using \eqref{eq:torus alpha and alpha'} or from
Claim~\ref{claim:cover minimal cases}. For $\glnidx>1$ it is a $2$-step nilpotent group.

We follow the exposition of
McNamara \cite{McNamara} (Sections~13.5-13.7, we only use arguments which do not impose restrictions on the field). See also Kazhdan and Patterson \cite{KP} (Sections~0.3,I.1-I.2) and Bump, Friedberg and Ginzburg \cite{BFG} (Section~2).

Assume $\glnidx>1$. Since $\cover{T}_{\glnidx+1}(F)$ is a $2$-step nilpotent group, the genuine irreducible representations of $\cover{T}_{\glnidx+1}(F)$ are parameterized by the genuine characters of $C_{\cover{T}_{\glnidx+1}(F)}$. Let $\chi$ be a genuine character of
$C_{\cover{T}_{\glnidx+1}(F)}$ and choose a maximal abelian subgroup $X<\cover{T}_{\glnidx+1}(F)$. Of course, $C_{\cover{T}_{\glnidx+1}(F)}<X$. We can extend $\chi$ to a character of $X$, then induce it to $\cover{T}_{\glnidx+1}(F)$. Denote $\rho(\chi)=\cinduced{X}{\cover{T}_{\glnidx+1}(F)}{\chi}$. This is a genuine irreducible representation, which is independent of the actual choices of $X$ and the extension (\cite{McNamara}, Theorem~3).
If $\glnidx\leq1$, we start with a genuine character $\chi$ of $C_{\cover{T}_{\glnidx+1}(F)}=\cover{T}_{\glnidx+1}(F)$ and put $\rho(\chi)=\chi$.

Extend $\rho(\chi)$ to a representation of $\cover{B}_{\glnidx}(F)$, trivially on $\mathfrak{s}(N_{\glnidx}(F))$, 
and induce to a representation $\induced{\cover{B}_{\glnidx}(F)}{\cover{G}_{\glnidx}(F)}{\rho(\chi)}$, whose space
we denote by $V(\chi)$. The elements $f\in V(\chi)$ are smooth complex-valued functions $f$ on $\cover{G}_{\glnidx}(F)\times \cover{T}_{\glnidx+1}(F)$ such that $f(tug,1)=\delta_{B_{\glnidx}(F)}^{1/2}(t)f(g,t)$ and $t\mapsto f(g,t)$ belongs to the space of $\rho(\chi)$ ($t\in \cover{T}_{\glnidx+1}(F)$, $u\in \mathfrak{s}(N_{\glnidx}(F))$ and $g\in\cover{G}_{\glnidx}(F)$). For brevity we will write $f(g)=f(g,1)$.


The Weyl group $W_{\glnidx}$ acts on the representations $\chi$ and $\rho(\chi)$.
If $\mathbf{w}\in W_{\glnidx}$ and $\xi$ is a representation of $C_{\cover{T}_{\glnidx+1}(F)}$ or $\cover{T}_{\glnidx+1}(F)$,
$\rconj{\mathbf{w}}\xi$ denotes the representation on the space of $\xi$ given by $\rconj{\mathbf{w}}\xi(t)=\xi(\rconj{\mathfrak{s}(w)^{-1}}t)$, where
$w\in\mathfrak{W}_{\glnidx}$ is the representative of $\mathbf{w}$. We have
$\rconj{\mathbf{w}}\rho(\chi)=\rho(\rconj{\mathbf{w}}\chi)$.


Let $\chi$ be a genuine character of $C_{\cover{T}_{\glnidx+1}(F)}$. Then ${\rho(\chi)}^{\wedge}=\rho(\chi^{\wedge})=\rho(\chi^{-1})$ (recall from Section~\ref{subsection:representations} that ${\rho(\chi)}^{\wedge}$ is the contragradient representation). 

A genuine character $\chi$ of $C_{\cover{T}_{\glnidx+1}(F)}$ is called regular if $\rconj{\mathbf{w}}\chi\ne\chi$ for all $1\ne \mathbf{w}\in W_{\glnidx}$. For such characters,
$j_{N_{\glnidx}}(\rho(\chi))$ is semi-simple (\cite{McNamara} Proposition~5) and is the direct sum of $\rho(\rconj{\mathbf{w}}\chi)$ where $\mathbf{w}$ varies over $W_{\glnidx}$. If $\chi'$ is another genuine character, the dimension of
\begin{align*}
Hom_{\cover{G}_{\glnidx}}(\induced{\cover{B}_{\glnidx}(F)}{\cover{G}_{\glnidx}(F)}{\rho(\chi)},\induced{\cover{B}_{\glnidx}(F)}{\cover{G}_{\glnidx}(F)}{\rho(\chi')}
\end{align*}
is zero unless $\chi'=\rconj{\mathbf{w}}\chi$ for some $\mathbf{w}$, in which case the dimension is $1$. These statements follow from Bernstein and Zelevinsky \cite{BZ2} (see \cite{McNamara}).

Let $\chi$ be a genuine character of $C_{\cover{T}_{\glnidx+1}(F)}$. There are unique $m_1,\ldots,m_{\glnidx+1}\in\R$ such that for all $t\in C_{\cover{T}_{\glnidx+1}(F)}$, if $p(t)=\prod_{i=1}^{\glnidx+1}{\alpha_i^{\vee}}(t_i)$,
$|\chi(t)|=\prod_{i=1}^{\glnidx+1}|t_i|^{m_i}$. Define $\Repart\chi=\sum_{i=2}^{\glnidx+1}m_i\alpha_i$. We say that $\chi$ belongs to the positive Weyl chamber if $\langle\Repart\chi,\alpha^{\vee}\rangle>0$ for all $\alpha\in\Delta_{G_{\glnidx}}$. 

For a genuine character $\chi$ of $C_{\cover{T}_{\glnidx+1}(F)}$, $\chi|_{\cover{T}_{\glnidx+1}(F)^2}$ is a genuine character and by
Claim~\ref{claim:T^2 is split by s}, it is obtained as the unique extension (via $\mathfrak{s}$) to $\cover{T}_{\glnidx+1}(F)^2$ of a character $\chi_0$ of
$T_{\glnidx+1}(F)^2$. According to Claim~\ref{claim:center of the torus}, one can further extend to $C_{\cover{T}_{\glnidx+1}(F)}$ and obtain $\chi$, according to $\chi|_{C_{\CGLF{\glnidx}{F}}}$. We will see that the ``important" properties of $\chi$ are shared by $\chi_0$. Motivated by this observation, let
$\Pi(\chi)$ be the set of genuine characters $\chi'$ of $C_{\cover{T}_{\glnidx+1}(F)}$ which agree with $\chi$ on
$\cover{T}_{\glnidx+1}(F)^2$. By Claim~\ref{claim:center of the torus}, $|\Pi(\chi)|=1$ or $[F^*:\Fsquares]$ depending on the parity of $\glnidx$.



\subsubsection{Unramified representations}\label{subsubsection:Unramified representations}
In this section assume $|2|=1$ and $q>3$ in $F$. Then $K$ 
is split under $\kappa$ and $K^*=\kappa(K)$ (Section~\ref{subsubsection:A splitting of the hyperspecial subgroup}). An irreducible genuine representation $\pi$ of $\cover{G}_{\glnidx}(F)$ is called unramified if it has a nonzero vector fixed by $K^*$.
Since Claim~\ref{claim:relations section s and section kappa} implies that for any $w\in \mathfrak{W}_{\glnidx}\subset K$,
$\mathfrak{s}(w)=\kappa(w)\in K^*$, $\rconj{\mathbf{w}}\pi$ is also unramified for all $\mathbf{w}\in W_{\glnidx}$.

Let $\chi$ be a genuine character of $C_{\cover{T}_{\glnidx+1}(F)}$. The subgroup $\cover{T}_{\glnidx+1}(F)\cap K^*$ is abelian (Claim~\ref{claim:constructing maximal abelian subgroup unramified case}). Let $X$ be a maximal abelian subgroup of $\cover{T}_{\glnidx+1}(F)$ which contains $\cover{T}_{\glnidx+1}(F)\cap K^*$. Assume that $\chi$ can be extended to a character of $X$ such that this extension is trivial on $\cover{T}_{\glnidx+1}(F)\cap K^*$.

According to \cite{McNamara} (Lemma~2) the subspace of $V(\chi)$ fixed by $K^*$ is one-dimensional. In particular
$\induced{\cover{B}_{\glnidx}(F)}{\cover{G}_{\glnidx}(F)}{\rho(\chi)}$ is 
unramified 
and we also call $\chi$ an unramified character. A function $f\in V(\chi)$ is unramified if it is fixed by $K^*$, and normalized if in addition $f(1)=1$. We will use the following observation from \cite{KP} (Lemma~I.1.3), \cite{McNamara} (Lemma~2): let $f$ be unramified.
If $t\in\cover{T}_{\glnidx+1}(F)$ does not belong to $X$,
$f(t)=0$. This is because for $k\in \cover{T}_{\glnidx+1}(F)\cap K^*$, $tk=[t,k]kt$, $[k,t]\ne1$ 
by the maximality of $X$ and $\chi(k)=1$, therefore
\begin{align}\label{eq:f normalized unramified vanishes}
f(t)&=f(tk,1)=[k,t]\delta_{B_{\glnidx}(F)}^{1/2}(t)\chi(k)f(1,t)=
[k,t]f(t).
\end{align}
We 
always choose $X=C(\cover{T}_{\glnidx+1}(F),\cover{T}_{\glnidx+1}(F)\cap K^*)$ (see 
Claim~\ref{claim:constructing maximal abelian subgroup unramified case}). According to \eqref{eq:constructing maximal abelian subgroup item 3} any genuine character of
$C_{\cover{T}_{\glnidx+1}(F)}$, which is trivial on $C_{\cover{T}_{\glnidx+1}(F)}\cap K^*$, can be extended uniquely
to a character of $X$ which is trivial on $\cover{T}_{\glnidx+1}\cap K^*$. Therefore, any such character is unramified.


Let $\alpha\in\Sigma_{G_{\glnidx}}^+$. If $\alpha$ is a long root and $\glnidx>1$, put $\mathfrak{l}(\alpha)=2$, otherwise $\mathfrak{l}(\alpha)=1$. Set
$a_{\alpha}={\alpha^{\vee}}^*(\varpi^{\mathfrak{l}(\alpha)})$. For any character $\chi$ of $C_{\cover{T}_{\glnidx+1}(F)}$ (not necessarily unramified), $\chi(a_{\alpha})$ is defined.
%

\subsubsection{Intertwining operators}\label{subsubsection:Intertwining operators}
Let $\chi$ be a genuine character of $C_{\cover{T}_{\glnidx+1}(F)}$ and $\rho(\chi)$ be the corresponding representation of $\cover{T}_{\glnidx+1}(F)$ (see Section~\ref{subsubsection:representations of the torus}). Let $\mathbf{w}\in W_{\glnidx}$. 
Put $N_{\glnidx}^w=\rconj{w}N_{\glnidx}
\cap N_{\glnidx}$.
Let $M(w,\chi):V(\chi)\rightarrow V(\rconj{\mathbf{w}}\chi)$ be the intertwining operator defined by
\begin{align*}
M(w,\chi)f(g)=\int_{\lmodulo{N_{\glnidx}^w(F)}{N_{\glnidx}(F)}}f(\mathfrak{s}(w)^{-1}\mathfrak{s}(u)g)du\qquad (g\in\cover{G}_{\glnidx}(F)).
\end{align*}
If $|2|=1$ and $q>3$ in $F$, $\mathfrak{s}(w)=\kappa(w)\in K^*$ 
(Claim~\ref{claim:relations section s and section kappa}) and we follow Casselman \cite{CS1} in the normalization of the measure $du$.
The following claim adapted from \cite{KP} (Sections~I.2 and I.6) provides the basic properties of the intertwining operators.
\begin{claim}\label{claim:props of intertwining and GK formula}
The integral defining $M(w,\chi)$ is absolutely convergent if $\langle\Repart\chi,\alpha^{\vee}\rangle>0$ for all $\alpha\in\Sigma_{G_{\glnidx}}^+$ such that $\mathbf{w}\alpha<0$.
The integral has a meromorphic continuation by which it is defined for all $\chi$. If $\mathbf{w},\mathbf{w'}\in W_{\glnidx}$, 
$ww'\in \mathfrak{W}_{\glnidx}$ and $l(ww')=l(w)+l(w')$, then
$M(ww',\chi)=M(w,\rconj{\mathbf{w'}}\chi)M(w',\chi)$. Finally, if $\chi$ is unramified 
and $f\in V(\chi)$ is the unramified normalized element, $M(w,\chi)f=c(\mathbf{w},\chi)f$, where
\begin{align*}
c(\mathbf{w},\chi)=\prod_{\setof{\alpha\in\Sigma_{G_{\glnidx}}^+}{\mathbf{w}\alpha<0}}\frac{1-q^{-1}\chi(a_{\alpha})}{1-\chi(a_{\alpha})}.
\end{align*}
\end{claim}
\begin{remark}\label{remark:why we need to prove GK formula}
The constant $c(\mathbf{w},\chi)$ is given by the Gindikin-Karpelevich formula (\cite{CS1} Section~3). This formula was proved in the context of $\cover{GL}_{\glnidx}$ by \cite{KP} (Section~I.2) and by McNamara \cite{McNamara2,McNamara} for any split reductive algebraic group over a $p$-adic field, as long as $|\mu_{2r}|=2r$ and $q$ is coprime to $2r$. 
\end{remark}

\begin{proof}
We assume $\glnidx>1$, since otherwise the cover splits and the statement is well known.
These results were proved by Kazhdan and Patterson \cite{KP} (Sections~I.2, I.6) for $\CGLF{\glnidx}{F}$, over any local field. We provide the calculation of
the constant $c(\mathbf{w},\chi)$. Over a $p$-adic field, the meromorphic continuation can also be proved using the continuation principle of Bernstein, see Banks \cite{Banks} (see also \cite{KP} p.~67 and \cite{Mu}).

Assume that $\chi$ is unramified and take $f$ as in the claim. It is enough to consider $\mathbf{w}_{\alpha}$ for $\alpha\in\Delta_{G_{\glnidx}}$. Let $2\leq l\leq\glnidx+1$ and put $\alpha=\alpha_l$. The volume of $\RingOfIntegers$ with respect to the additive measure of $F$ is $1$. Hence
\begin{align*}
c(\mathbf{w_{\alpha}},\chi)=M(w_{\alpha},\chi)f(1)=1+\int_{\mathcal{U}_{\alpha}\setdifference K}f({w_{\alpha}^*}^{-1}\mathfrak{s}(u))du.
\end{align*}
Write $u=n_{\alpha}(x)$ for some $x\in F$ with $x\notin \RingOfIntegers$, then $\mathfrak{s}(u)=n_{\alpha}^*(x)$.
Using \eqref{eq:torus alpha and alpha'}-\eqref{eq:torus and unipotent} we see that
\begin{align*}
{w_{\alpha}^*}^{-1}n_{\alpha}^*(x)&={{\alpha^{\vee}}^*}(x^{-1})n_{\alpha}^*(-x){{\alpha^{\vee}}^*}(-1)w_{\alpha}^*n_{\alpha}^*(-x^{-1}){w_{\alpha}^*}^{-1}.
\end{align*}
By virtue of Claim~\ref{claim:relations section s and section kappa} we have $w_{\alpha}^*=\mathfrak{s}(w_{\alpha})\in K^*$, $n_{\alpha}^*(-x^{-1})=\mathfrak{s}(n_{\alpha}(-x^{-1}))\in K^*$ ($|x|>1$) and ${\alpha^{\vee}}^*(-1)=\mathfrak{s}(\alpha^{\vee}(-1))\in K^*$. Therefore
\begin{align*}
f({w_{\alpha}^*}^{-1}\mathfrak{s}(u))=f({\alpha^{\vee}}^*(x^{-1}))=
\delta_{B_{\glnidx}(F)}^{1/2}(\alpha^{\vee}(x^{-1}))f(1,{\alpha^{\vee}}^*(x^{-1}))
=|x|f(1,{\alpha^{\vee}}^*(x^{-1})).
\end{align*}

Assume $l\leq\glnidx$. If $f(1,{\alpha^{\vee}}^*(x^{-1}))\ne0$, by the observation in Section~\ref{subsubsection:Unramified representations} (see \eqref{eq:f normalized unramified vanishes}) we must have ${\alpha^{\vee}}^*(x^{-1})\in C(\cover{T}_{\glnidx+1}(F),\cover{T}_{\glnidx+1}(F)\cap K^*)$. Assume this holds and set $l'=l+1$ if $l\leq\glnidx-1$, otherwise $l'=\glnidx-1$ (recall that $\glnidx>1$). Then Equality~\eqref{eq:torus alpha and alpha'} applied to ${\alpha^{\vee}}^*(x^{-1})$ and ${\alpha_{l'}^{\vee}}^*(z)$ implies $c(x^{-1},z)=1$ for all $z\in\RingOfIntegers^*$, whence $x\in\RingOfIntegers^*{\Fsquares}$. We have shown that the integrand vanishes unless $x\in\RingOfIntegers^*{\Fsquares}$, in which case $f(1,{\alpha^{\vee}}^*(x^{-1}))=\chi({\alpha^{\vee}}^*(x^{-1}))$. When $l=\glnidx+1$, this last equality holds for all $x\in F^*$.

Thus for all $2\leq l\leq\glnidx+1$,
\begin{align*}
\int_{\mathcal{U}_{\alpha}(F)\setdifference K}f({w_{\alpha}^*}^{-1}\mathfrak{s}(u))du=\sum_{v=1}^{\infty}\chi(a_{\alpha})^v=(1-q^{-1})\frac{\chi(a_{\alpha})}{1-\chi(a_{\alpha})}.
\end{align*}
(The change $du\mapsto d^*x$ contributed a factor canceled by $\delta_{B_{\glnidx}(F)}^{1/2}$.)
Therefore
$c(\mathbf{w_{\alpha}},\chi)=(1-q^{-1}\chi(a_{\alpha}))/(1-\chi(a_{\alpha}))$,
completing the proof.
\end{proof}

\subsubsection{Structure of certain induced representations}\label{subsubsection:tensor product}
In this section $F$ is $p$-adic. For any $H<G_{\glnidx}$, set $H(F)^{\star}=\setof{h\in H(F)}{\Upsilon(h)\in {\Fsquares}}$. In particular if $H=\GL{k}$,
$H(F)^{\star}=\setof{h\in \GLF{k}{F}}{\det{h}\in {\Fsquares}}$, and $G_0(F)^{\star}=G_0(F)$. The subgroup $H(F)^{\star}$ is normal in $H(F)$ and $\lmodulo{H(F)^{\star}}{H(F)}$ is abelian. The index of $H(F)^{\star}$ in $H(F)$ is either $1$ or $[F^*:{\Fsquares}]$. These properties hold also for $\cover{H}(F)^{\star}(=p^{-1}(H(F)^{\star}))$ and $\cover{H}(F)$.

The subgroup $H(F)^{\star}$ is open in $H(F)$, and $\cover{H}(F)^{\star}$ is open in $\cover{H}(F)$ %
(\cite{Moore} p.~54-56, see also \cite{KP} Proposition~0.1.2). 

If $\xi$ is a genuine irreducible representation of $\cover{H}(F)$, let $\xi^{\star}=\xi|_{\cover{H}(F)^{\star}}$. The representation $\xi^{\star}$ is a finite sum of (at most $[F^*:{\Fsquares}]$) genuine irreducible representations (\cite{BZ1} 2.9).

We will need the following results of Kable \cite{Kable} (in Propositions~3.1 and 3.2).
\begin{proposition}\label{proposition:Kable summary of restriction results}
Let $H$ be a Levi subgroup of $\GL{\glnidx}$ and let $\tau$ be a genuine irreducible representation of $\cover{H}(F)$. If $\glnidx$ is odd, $\tau^{\star}$ is irreducible and $\cinduced{\cover{H}(F)^{\star}}{\cover{H}(F)}{\tau^{\star}}=\oplus_{\omega}\omega\cdot\tau$, where $\omega$ ranges over the (finite set of) characters of $\lmodulo{\cover{H}(F)^{\star}}{\cover{H}(F)}$. Furthermore,
$Hom_{\cover{H}(F)}(\omega\tau,\tau)=0$ unless $\omega=1$. If $\glnidx$ is even, $\tau^{\star}=\oplus_{a}\rconj{a}\sigma$ where $\sigma$ is any irreducible summand of $\tau^{\star}$ and $a$ ranges over $\lmodulo{\cover{H}(F)^{\star}}{\cover{H}(F)}$. Furthermore, $\rconj{a}\sigma\not\isomorphic\sigma$ for all $a\notin \cover{H}(F)^{\star}$.
\end{proposition}

Let $H_1$ be a Levi subgroup of $\GL{k}$ and let $H_2$ be a Levi subgroup of $G_{\glnidx-k}$. Then $H=H_1\times H_2$ is a Levi subgroup of $G_{\glnidx}$. We assume that the cocycle defined on $G_{\glnidx}(F)$ satisfies the block-compatibility criterion
\eqref{eq:block-compatibility on Levi subgroups}.

Let $\tau$ be a genuine irreducible representation of $\cover{H}_1(F)$ and let $\pi$ be a genuine irreducible representation of $\cover{H}_2(F)$. The following discussion describes a replacement for the usual tensor ``$\tau\otimes\pi$", which is not defined when the groups are not commuting. 

According to \eqref{eq:block-compatibility on Levi subgroups}, the subgroups $\cover{H}_1(F)^{\star}$ and $\cover{H}_2(F)^{\star}$ are commuting in $\cover{G}_{\glnidx}(F)$, hence 
\begin{align*}
p^{-1}(H_1(F)^{\star}\times H_2(F)^{\star})\isomorphic\lmodulo{\setof{(\zeta,\zeta)}{\zeta\in\mu_2}}
{(\cover{H}_1(F)^{\star}\times\cover{H}_2(F)^{\star})}.
\end{align*}
Here on the \rhs\ we have an outer direct product. In other words, $p^{-1}(H_1(F)^{\star}\times H_2(F)^{\star})$ is the direct product of $\cover{H}_1(F)^{\star}$ and $\cover{H}_2(F)^{\star}$ with amalgamated $\mu_2$.
The representation $\tau^{\star}\otimes\pi^{\star}$ is defined as the usual tensor product and can be regarded as a genuine irreducible
representation of $p^{-1}(H_1(F)^{\star}\times H_2(F)^{\star})$. The representations $\tau^{\star}\otimes\pi$ and
$\tau\otimes\pi^{\star}$ are defined similarly, because the same arguments apply to the pairs $(\cover{H}_1(F)^{\star},\cover{H}_2(F))$ and
$(\cover{H}_1(F),\cover{H}_2(F)^{\star})$. We have the following semisimple representation
\begin{align*}
I(\tau,\pi)^{\star}=\cinduced{p^{-1}(H_1(F)^{\star}\times H_2(F)^{\star})}{\cover{H}(F)}{\tau^{\star}\otimes\pi^{\star}}.
\end{align*}
In the case of $H_2=G_0$, $H_2(F)^{\star}=H_2(F)$ hence $(\cover{H}_1(F),\cover{H}_2(F))$ are commuting and
\begin{align*}
I(\tau,\pi)^{\star}=\cinduced{p^{-1}(H_1(F)^{\star}\times H_2(F))}{\cover{H}(F)}{\tau^{\star}\otimes\pi}
=\cinduced{\cover{H}_1(F)^{\star}}{\cover{H}_1(F)}{\tau^{\star}}\otimes\pi.
\end{align*}
\begin{remark}\label{remark:the Kable induced representation is not needed when k=n}
In this case the tensor $\tau\otimes\pi$ is defined and of course will be preferred over $I(\tau,\pi)^{\star}$. See Claim~\ref{claim:special case k=n} in Section~\ref{subsubsection:defs and basic properties} below.
\end{remark}
\begin{lemma}\label{lemma:induced representation composition factors}
Assume $0<k<\glnidx$. The representation $I(\tau,\pi)^{\star}$ is a direct sum of $[F^*:{\Fsquares}]$ copies of
\begin{align*}
\cinduced{p^{-1}(H_1(F)^{\star}\times H_2(F))}{\cover{H}(F)}{\tau^{\star}\otimes\pi}.
\end{align*}
Similarly, it is a direct sum of $[F^*:{\Fsquares}]$ copies of
\begin{align*}
\cinduced{p^{-1}(H_1(F)\times H_2(F)^{\star})}{\cover{H}(F)}{\tau\otimes\pi^{\star}}.
\end{align*}
\end{lemma}
\begin{proof}
The arguments are similar to those of Kable \cite{Kable} (Theorem~3.1). Since $p^{-1}(H_1(F)^{\star}\times H_2(F)^{\star})$ is a normal subgroup of $\cover{H}(F)$ with a finite index and the quotient of $p^{-1}(H_1(F)^{\star}\times H_2(F))$ by  $p^{-1}(H_1(F)^{\star}\times H_2(F)^{\star})$ is abelian,
\begin{align}\label{eq:claim tensor following Kable 1}
\cinduced{p^{-1}(H_1(F)^{\star}\times H_2(F)^{\star})}{\cover{H}(F)}{\tau^{\star}\otimes\pi^{\star}}
=\bigoplus_{\omega}\cinduced{p^{-1}(H_1(F)^{\star}\times H_2(F))}{\cover{H}(F)}{\tau^{\star}\otimes\omega\pi},
\end{align}
where the summation ranges over the characters of $\lmodulo{\cover{H}_2(F)^{\star}}{\cover{H}_2(F)}$ (these are nongenuine characters). Any such character $\omega$ takes the form $\omega(h)=\omega_a(h)=c(\Upsilon(h),\det{a})$ for some $a\in\cover{H}_1(F)$. By the block-compatibility of the cocycle (and \eqref{eq:cocycle trick}), $ha=c(\Upsilon(h),\det{a})ah$ ($a\in \cover{H}_1(F)$, $h\in \cover{H}_2(F)$) whence $\tau^{\star}\otimes\omega_a\pi=\rconj{a}(\rconj{a^{-1}}(\tau^{\star})\otimes\pi)$. Thus the \rhs\ of \eqref{eq:claim tensor following Kable 1} is equal to a direct sum of $[\cover{H}_1(F):\cover{H}_1(F)^{\star}]=[F^*:{\Fsquares}]$ copies of
\begin{align}\label{eq:claim tensor following Kable 2}
\cinduced{p^{-1}(H_1(F)^{\star}\times H_2(F))}{\cover{H}(F)}{\tau^{\star}\otimes\pi}.
\end{align}
One can similarly consider
$p^{-1}(H_1(F)\times H_2(F)^{\star})$ and repeating the steps above obtain
\begin{align*}
\cinduced{p^{-1}(H_1(F)^{\star}\times H_2(F)^{\star})}{\cover{H}(F)}{\tau^{\star}\otimes\pi^{\star}}
&=\bigoplus_{h\in\lmodulo{\cover{H}_2(F)^{\star}}{\cover{H}_2(F)}}\cinduced{p^{-1}(H_1(F)\times H_2(F)^{\star})}{\cover{H}(F)}{\omega_h\tau\otimes\pi^{\star}}\\
&=[F^*:{\Fsquares}]\cinduced{p^{-1}(H_1(F)\times H_2(F)^{\star})}{\cover{H}(F)}{\tau\otimes\pi^{\star}}.
\end{align*}
Here $\omega_h(a)=c(\Upsilon(h),\det{a})$.
\end{proof}
The aforementioned results of Kable \cite{Kable} (Propositions~3.1, 3.2) do not apply to the group $G_{\glnidx}(F)$, mainly because
its center does not play a role similar to that of the center of $\GLF{\glnidx}{F}$. Namely, $C_{G_{\glnidx}(F)}<G_{\glnidx}(F)^{\star}$ for all $\glnidx$ (see Section~\ref{subsection:properties of GSpin}). However, if $H_2$ is a Levi subgroup of a proper parabolic subgroup of $G_{\glnidx}$, it is possible to extend Proposition~\ref{proposition:Kable summary of restriction results} to $H_2$.
We will only need the following result on $T_{\glnidx+1}$.
\begin{claim}\label{claim:restriction results analogous to Kable}
Let $\pi$ be a genuine irreducible representation of $\cover{T}_{\glnidx+1}(F)$ and assume that $\glnidx$ is even (including zero). Then $\pi^{\star}=\oplus_{h}\rconj{h}\rho$ where $\rho$ is any irreducible summand of $\pi^{\star}$ and $h$ ranges over $\lmodulo{\cover{T}_{\glnidx+1}(F)^{\star}}{\cover{T}_{\glnidx+1}(F)}$.
\end{claim}
\begin{proof}
Equality~\eqref{eq:general commutator on torus} implies that $\cover{C}_{\GLF{\glnidx}{F}}$ is contained in the center of $\cover{T}_{\glnidx+1}(F)^{\star}$ and furthermore, if $z\in\cover{C}_{\GLF{\glnidx}{F}}$ with $p(z)=\prod_{i=1}^{\glnidx}\eta_i^{\vee}(d)$ and $t\in\cover{T}_{\glnidx+1}(F)$ with $p(t)=\prod_{i=1}^{\glnidx}\eta_i^{\vee}(t_i)\beta_1^{\vee}(t_{\glnidx+1})$, $tzt^{-1}=c(\prod_{i=1}^{\glnidx}t_i,d)z$ (see \eqref{eq:cocycle trick} and \eqref{eq:general commutator on torus}). Given these observations, the arguments of
Kable \cite{Kable} (Proposition~3.2) readily apply to our case.
\end{proof}

Recall that for a genuine irreducible representation $\xi$ of $\cover{T}_{\glnidx+1}(F)$ we have a corresponding genuine character $\chi_{\xi}$ of $C_{\cover{T}_{\glnidx+1}(F)}$ and $\xi=\rho(\chi_{\xi})$. The same applies to representations of $\cover{T}_{\GL{\glnidx}}(F)$ and we use the same notation $\rho(\cdots)$.

Let $\chi_1'$ and $\chi_2'$ be genuine characters of $\cover{T}_{\GL{k}}(F)^2$ and $\cover{T}_{\glnidx-k+1}(F)^2$. Since
\begin{align}\label{iso:torus splits into GL part and G part}
\cover{T}_{\glnidx+1}(F)^2\isomorphic\lmodulo{\setof{(\zeta,\zeta)}{\zeta\in\mu_2}}
{(\cover{T}_{\GL{k}}(F)^2\times\cover{T}_{\glnidx-k+1}(F)^2)},
\end{align}
the tensor representation $\chi_1'\otimes\chi_2'$ of $\cover{T}_{\glnidx+1}(F)^2$ is defined, it is a genuine character. For
genuine characters $\chi_1$ and $\chi_2$ of $C_{\cover{T}_{\GL{k}}(F)}$ and $C_{\cover{T}_{\glnidx-k+1}(F)}$, put
\begin{align*}
\chi_1\odot\chi_2=\chi_1|_{\cover{T}_{\GL{k}}(F)^2}\otimes\chi_2|_{\cover{T}_{\glnidx-k+1}(F)^2}.
\end{align*}

The last claim enables us to deduce the following result.
\begin{lemma}\label{lemma:induced representation composition factors for torus}
Assume $0<k\leq\glnidx$, $\tau$ and $\pi$ are irreducible genuine representations of $\cover{T}_{\GL{k}}(F)$ and $\cover{T}_{\glnidx-k+1}(F)$ (resp.). Write $\tau=\rho(\chi_{\tau})$ and $\pi=\rho(\chi_{\pi})$, where
$\chi_{\tau}$ and $\chi_{\pi}$ are genuine characters of $C_{\cover{T}_{\GL{k}}(F)}$ and
$C_{\cover{T}_{\glnidx-k+1}(F)}$ (resp.). Also let $\chi$ be a genuine character of $C_{\cover{T}_{\glnidx+1}(F)}$ which agrees with $\chi_{\tau}\odot\chi_{\pi}$ on $\cover{T}_{\glnidx+1}(F)^2$. Then
\begin{align*}
I(\tau,\pi)^{\star}=[F^*:{\Fsquares}]^r\bigoplus_{\chi'\in\Pi(\chi)}\rho(\chi'),
\end{align*}
where $r=2$ if $\glnidx$ and $k$ are even and $k<\glnidx$, $r=0$ if $\glnidx$ is odd and $k=\glnidx$, otherwise $r=0$;
the finite set $\Pi(\chi)$ was defined in Section~\ref{subsubsection:representations of the torus}.
\end{lemma}
\begin{proof}
We follow the arguments of Kable \cite{Kable} (Theorem~3.1). Applying
Lemma~\ref{lemma:induced representation composition factors} to $I(\tau,\pi)^{\star}$ we write it as a direct sum of
$[F^*:{\Fsquares}]$ copies of
\begin{align}\label{eq:claim tensor following Kable 3}
\cinduced{p^{-1}(T_{\GL{k}}(F)^{\star}\times T_{\glnidx-k+1}(F))}{\cover{T}_{\glnidx+1}(F)}{\tau^{\star}\otimes\pi}.
\end{align}
Now we consider the different cases.
\begin{enumerate}[leftmargin=*]
\item\label{item:k odd, n even}$\glnidx$ is even, $k$ is odd: by Proposition~\ref{proposition:Kable summary of restriction results} the representation $\tau^{\star}$ is irreducible. Hence $\tau^{\star}\otimes\pi$ is also irreducible. As explained in the proof of Lemma~\ref{lemma:induced representation composition factors}, for any $a\in \cover{T}_{\GL{k}}(F)$, $\rconj{a}(\tau^{\star}\otimes\pi)=\tau^{\star}\otimes\omega_a\pi$. Because $\pi$ and $\omega_a\pi$ are both irreducible genuine representations of $\cover{T}_{\glnidx-k+1}(F)$, they are isomorphic if and only if their restrictions to
$C_{\cover{T}_{\glnidx-k+1}(F)}$ are identical. Since $\glnidx-k$ is odd, $C_{\CGLF{\glnidx-k}{F}}<C_{\cover{T}_{\glnidx-k+1}(F)}$. The restriction $\omega_a|_{C_{\CGLF{\glnidx-k}{F}}}$ is trivial if and only if $a\in\cover{T}_{\GL{k}}(F)^{\star}$, whence
\begin{align*}
Hom_{\cover{T}_{\glnidx-k+1}(F)}(\omega_a\pi,\pi)=0,\qquad\forall a\notin\cover{T}_{\GL{k}}(F)^{\star}.
\end{align*}
This means that $\tau^{\star}\otimes\pi$ satisfies Mackey's criterion and therefore by Mackey's theory \eqref{eq:claim tensor following Kable 3} is an irreducible representation of $\cover{T}_{\glnidx+1}(F)$. We claim that it is isomorphic to $\rho(\chi)$, which implies
$I(\tau,\pi)^{\star}=[F^*:{\Fsquares}]\rho(\chi)$ (since $\glnidx$ is even, $\Pi(\chi)=\{\chi\}$).

Indeed, $\glnidx$ is even
hence $C_{\cover{T}_{\glnidx+1}(F)}=\cover{T}_{\glnidx+1}(F)^2$ and it suffices to show both representations agree on $\cover{T}_{\glnidx+1}(F)^2$. This holds because of \eqref{iso:torus splits into GL part and G part} and $\cover{T}_{\GL{k}}(F)^2<\cover{T}_{\GL{k}}(F)^{\star}$.

\item\label{item:k and n even}$\glnidx$ and $k$ are even: write $\tau^{\star}=\oplus_{a}\rconj{a}\sigma$ as in Proposition~\ref{proposition:Kable summary of restriction results}. Then
    \eqref{eq:claim tensor following Kable 3} is the direct sum
\begin{align}\label{eq:claim tensor following Kable 4}
\bigoplus_{a\in\lmodulo{\cover{T}_{\GL{k}}(F)^{\star}}{\cover{T}_{\GL{k}}(F)}}\cinduced{
p^{-1}(T_{\GL{k}}(F)^{\star}\times T_{\glnidx-k+1}(F))}{\cover{T}_{\glnidx+1}(F)}{\sigma\otimes\omega_a\pi}.
\end{align}
If $\rconj{b}(\sigma\otimes\omega_a\pi)\isomorphic\sigma\otimes\omega_a\pi$ for some $b\in \cover{T}_{\GL{k}}(F)$, then
$\rconj{b}\sigma\isomorphic \sigma$, which by Proposition~\ref{proposition:Kable summary of restriction results} only happens for
$b\in\cover{T}_{\GL{k}}(F)^{\star}$. Thus $\sigma\otimes\omega_a\pi$ satisfies Mackey's criterion and each summand is irreducible.
Since $\glnidx-k$ is even, $\omega_a|_{C_{\cover{T}_{\glnidx-k+1}(F)}}=1$ whence $\omega_a\pi\isomorphic\pi$ and all summands are equal. Thus if $k<\glnidx$ we get $I(\tau,\pi)^{\star}=[F^*:{\Fsquares}]^2\rho(\chi)$. If $k=\glnidx$, $I(\tau,\pi)^{\star}$ is equal to \eqref{eq:claim tensor following Kable 3} hence $I(\tau,\pi)^{\star}=[F^*:{\Fsquares}]\rho(\chi)$.

\item\label{item:odd n, k even}$\glnidx$ is odd, $k$ is even: as in case~\eqref{item:k and n even} consider \eqref{eq:claim tensor following Kable 4}. Because $k$ is even, as above each summand of \eqref{eq:claim tensor following Kable 4} is irreducible. Since both $\glnidx$ and $\glnidx-k$ are odd, any $z\in C_{\CGLF{\glnidx}{F}}$ can be written as $z=z_1z_2$ with $z_1\in\cover{C}_{\GLF{k}{F}}$ and $z_2\in  C_{\CGLF{\glnidx-k}{F}}$, and furthermore, $z\in C_{\cover{T}_{\glnidx+1}(F)}$ and $z_2\in C_{\cover{T}_{\glnidx-k+1}(F)}$. Additionally, because $k$ is even, $\cover{C}_{\GLF{k}{F}}$ is contained in the center of $\cover{T}_{\GL{k}}(F)^{\star}$, 
    hence $\sigma|_{\cover{C}_{\GLF{k}{F}}}$ is a character. Using these observations we see that the restrictions of
$\sigma\otimes\omega_a\pi$ to $C_{\CGLF{\glnidx}{F}}$ as $a$ varies are different. Therefore the summands in \eqref{eq:claim tensor following Kable 4} are inequivalent.

Each summand takes the form $\rho(\chi')$, where $\chi'$ is some genuine character of $C_{\cover{T}_{\glnidx+1}(F)}$, but $\chi'$ must agree with $\chi$ on $\cover{T}_{\glnidx+1}(F)^2$ because $\sigma|_{\cover{T}_{\GL{k}}(F)^2}=\tau|_{\cover{T}_{\GL{k}}(F)^2}$ ($\cover{T}_{\GL{k}}(F)^2<C_{\cover{T}_{\GL{k}}(F)}$) and $\omega_a|_{\cover{T}_{\glnidx-k+1}(F)^2}=1$. There are $[F^*:\Fsquares]$ such characters $\chi'$ and because there are exactly $[F^*:\Fsquares]$ non isomorphic summands in \eqref{eq:claim tensor following Kable 4}, all possible $\chi'$ appear. It follows that $I(\tau,\pi)^{\star}=[F^*:{\Fsquares}]\oplus_{\chi'\in\Pi(\chi)}\rho(\chi')$.

\item\label{item:n and k odd}$\glnidx$ and $k$ are odd: assume $k<\glnidx$. Using Lemma~\ref{lemma:induced representation composition factors},
\begin{align*}
I(\tau,\pi)^{\star}&=[F^*:{\Fsquares}]\cinduced{p^{-1}(T_{\GL{k}}(F)\times T_{\glnidx-k+1}(F)^{\star})}{\cover{T}_{\glnidx+1}(F)}{\tau\otimes\pi^{\star}}\\\nonumber
&=[F^*:{\Fsquares}]\bigoplus_{h\in\lmodulo{\cover{T}_{\glnidx-k+1}(F)^{\star}}{\cover{T}_{\glnidx-k+1}(F)}}\cinduced{p^{-1}(T_{\GL{k}}(F)\times T_{\glnidx-k+1}(F)^{\star})}{\cover{T}_{\glnidx+1}(F)}{\omega_h\tau\otimes\rho}.
\end{align*}
Here $\omega_h(a)=c(\Upsilon(h),\det{a})$ and $\pi^{\star}=\oplus_{h}\rconj{h}\rho$ as in Claim~\ref{claim:restriction results analogous to Kable}. Since $k$ is odd, by Proposition~\ref{proposition:Kable summary of restriction results} we have $Hom_{\cover{T}_{\GL{k}}(F)}(\omega_{h'}\omega_h\tau,\omega_h\tau)=0$ as long as $h'\notin\cover{T}_{\glnidx-k+1}(F)^{\star}$ and it follows that each of the last summands is irreducible. For $h\ne h'$, the summands are non isomorphic: this follows as in case~\eqref{item:odd n, k even} (the opposite case with respect to $k$ and $\glnidx-k$), because $\cover{C}_{\GLF{\glnidx-k}{F}}$ is contained in the center of $\cover{T}_{\glnidx-k+1}(F)^{\star}$ (see \eqref{eq:general commutator on torus}) and for $z\in C_{\CGLF{\glnidx}{F}}$, $z=z_1z_2$ where $z_1\in C_{\CGLF{k}{F}}$ and $z_2\in \cover{C}_{\GLF{\glnidx-k}{F}}$. Thus $I(\tau,\pi)^{\star}=[F^*:{\Fsquares}]\oplus_{\chi'\in\Pi(\chi)}\rho(\chi')$.
Finally if $k=\glnidx$, by Proposition~\ref{proposition:Kable summary of restriction results},
\[
I(\tau,\pi)^{\star}=\cinduced{\cover{T}_{\GL{\glnidx}}(F)^{\star}}{\cover{T}_{\GL{\glnidx}}(F)}{\tau^{\star}}\otimes\pi=\oplus_{\omega}(\omega\tau\otimes\pi)=\oplus_{\chi'\in\Pi(\chi)}\rho(\chi').\qedhere
\]
\end{enumerate}
\end{proof}
\begin{remark}
One could define
$I(\tau,\pi)^{\star}=\cinduced{p^{-1}(H_1(F)\times H_2(F)^{\star})}{\cover{H}(F)}{\tau\otimes\pi^{\star}}$.
With this definition the case $k=\glnidx$ is simpler, $I(\tau,\pi)^{\star}=\tau\otimes\pi$.
Lemma~\ref{lemma:induced representation composition factors} implies that for $k<\glnidx$,
\begin{align*}
\cinduced{p^{-1}(H_1(F)\times H_2(F)^{\star})}{\cover{H}(F)}{\tau\otimes\pi^{\star}}\isomorphic\cinduced{p^{-1}(H_1(F)^{\star}\times H_2(F))}{\cover{H}(F)}{\tau^{\star}\otimes\pi}.
\end{align*}
Then in Lemma~\ref{lemma:induced representation composition factors for torus} we obtain smaller multiplicities. However, the presentation in Sections~\ref{subsubsection:defs and basic properties}-\ref{subsubsection:vanishing and results} below seems to be simpler with the definition above.
\end{remark}
\subsubsection{Discussion on the metaplectic tensor product}\label{subsubsection:tensor discussion}
Irreducible representations of Levi subgroups of classical groups are usually described in terms of the tensor product.
For metaplectic groups the direct factors of Levi subgroups do not necessarily commute, hence the tensor construction cannot be extended in a straightforward manner.

The metaplectic tensor product in the context of $\GL{\glnidx}$ has been studied by various authors \cite{FK,Su2,Kable,Mezo,Tk2}. Kable \cite{Kable} used a representation similar to $I(\tau,\pi)^{\star}$ (see Section~\ref{subsubsection:tensor product}) to define a metaplectic tensor product, between genuine indecomposable representations of Levi subgroups of the double cover of $\GLF{\glnidx}{F}$. He studied the induced space $I(\tau_1,\tau_2)^{\star}$, where $\tau_i$ are representations of $\cover{L}_i(F)$ and $L_i<\GL{\glnidx_i}$ are Levi subgroups. To any genuine character $\omega$ of $C_{\CGLF{\glnidx}{F}}$, Kable defined $\tau_1\otimes_{\omega}\tau_2$ as an indecomposable summand of $I(\tau_1,\tau_2)^{\star}$, on which $C_{\CGLF{\glnidx}{F}}$ acts by $\omega$. He proved that this construction satisfies many of the useful properties of the tensor product, e.g., associativity, compatibility with taking contragradients and certain compatibility with Jacquet functors.

Essential for his approach was the Mackey theory for the restriction $\tau_i^{\star}=\tau_i|_{\cover{L}_i(F)^{\star}}$, which he developed for the double cover (\cite{Kable} Propositions~3.1,3.2). The main obstacle in trying to extend his construction to our setting, is the lack of an analog of these results 
for $G_{\glnidx}(F)$.

Mezo \cite{Mezo} considered $r$-fold covers. Denote $\GLF{m}{F}^r=\setof{b\in\GLF{m}{F}}{\det{b}\in {F^{*r}}}$. Mezo defined a tensor representation for $L=\GL{\glnidx_1}\times\ldots\times\GL{\glnidx_k}$ using restrictions from $\GLF{\glnidx_i}{F}$ to $\GLF{\glnidx_i}{F}^r$ and a process resembling the construction in Section~\ref{subsubsection:representations of the torus}. Takeda \cite{Tk2} constructed a global metaplectic tensor product, whose local components agree with those of Mezo \cite{Mezo}. He proved several properties for the global (and local) tensor, e.g., compatibility with induction and automorphicity.

For our purposes it will be sufficient to consider the space $I(\tau,\pi)^{\star}$, on which we can compute certain Jacquet functors simply enough. The lack of a universal property for the metaplectic tensor product places a heavy burden on details. At present, we do not attempt to extend the results of \cite{Mezo,Tk2} to our context.
\subsection{Exceptional (small) representations}\label{subsection:small representations}
\subsubsection{Definition and basic properties}\label{subsubsection:defs and basic properties}
We construct exceptional representations, adapting the results of Bump, Friedberg and Ginzburg \cite{BFG} in the context of $SO_{2\glnidx+1}$. These representations were called ``small" in \cite{BFG}. 

Let $\chi$ be a genuine character of $C_{\cover{T}_{\glnidx+1}(F)}$. We say that $\chi$ is exceptional if for all $\alpha\in\Delta_{G_{\glnidx}}$, $\chi({\alpha^{\vee}}^*(x^{\mathfrak{l}(\alpha)}))=|x|$ for all $x\in F^*$. In particular, such a character is regular.

Let $\chi_0$ be a character of $T_{\glnidx+1}(F)^2$ such that $\chi$ is obtained from $\chi_0$ by extension, as explained in Section~\ref{subsubsection:representations of the torus}. Since $\mathfrak{s}(\alpha^{\vee}(x^{\mathfrak{l}(\alpha)}))={\alpha^{\vee}}^*(x^{\mathfrak{l}(\alpha)})$, $\chi_0$ satisfies $\chi_0(\alpha^{\vee}(x^{\mathfrak{l}(\alpha)}))=|x|$ for all $\alpha\in\Delta_{G_{\glnidx}}$ and $x\in F^*$. We see that $\chi$ is exceptional if and only if $\chi_0$ satisfies this property. Refer to Section~\ref{subsubsection:Explicit construction of an exceptional representation} below for a detailed construction of an exceptional character.

In this section, except in Proposition~\ref{proposition:periodicity theorem}, the field is $p$-adic. Recall that $\mathbf{w_0}$ denotes the longest element of $W_{\glnidx}$ and $w_0\in\mathfrak{W}_{\glnidx}$ is the representative.
Additionally, for brevity, when we refer to a Jacquet functor applied to a genuine representation along some unipotent subgroup $U$, we drop $\mathfrak{s}$ from the notation, e.g., write $j_U(\cdots)$ instead of $j_{\mathfrak{s}(U)}(\cdots)$.

As in \cite{BFG} (Theorem~2.2) we have the ``periodicity theorem" (see also \cite{KP} Theorem~I.2.9):
\begin{proposition}\label{proposition:periodicity theorem}
Let $\chi$ be an exceptional character. The representation $\Theta_{G_{\glnidx},\chi}$ on the space $M(w_0,\chi)V(\chi)$ is irreducible, it is the unique irreducible subrepresentation of $V(\rconj{\mathbf{w_0}}\chi)$ and the unique irreducible quotient of $V(\chi)$. Furthermore if $F$ is $p$-adic, $j_{N_{\glnidx}}(\Theta_{G_{\glnidx},\chi})=\rho(\rconj{\mathbf{w_0}}\chi)$.
\end{proposition}
\begin{proof}
Consider the $p$-adic case first. The facts that $V(\rconj{\mathbf{w_0}}\chi)$ has a unique irreducible subrepresentation and $V(\chi)$ has a unique irreducible quotient, follow from the Langlands Quotient Theorem proved for metaplectic groups by Ban and Jantzen \cite{BJ}. This is because $\chi$ belongs to the positive Weyl chamber. Now it is left to compute the Jacquet module and establish the irreducibility of $\Theta_{G_{\glnidx},\chi}$.

One argues exactly as in \cite{BFG}, starting with computation of $j_{N_{\glnidx}}(\Theta_{G_{\glnidx},\chi})$, which implies that $\Theta_{G_{\glnidx},\chi}$ is irreducible. 
The calculation follows from the properties of the Jacquet functor described in Section~\ref{subsubsection:representations of the torus}, Claim~\ref{claim:props of intertwining and GK formula}, the observation that when $|2|=1$ and $q>3$, $1-q^{-1}(\rconj{\mathbf{w_{\alpha}}}\chi(a_{\alpha}))=0$ for all $\alpha\in\Delta_{G_{\glnidx}}$, and a computation on $G_1(F)$ and $\cover{SL}_2(F)$ in the remaining cases 
($\cover{G}_1(F)$ is split).

If $F$ is \archimedean, the proposition follows from the Langlands Quotient Theorem \cite{La3}, whose proof by Borel and Wallach \cite{BW} is applicable to cover groups (see \cite{BJ}). Note that the results of \cite{BJ} do not include the characterization of the Langlands quotient in terms of the intertwining operators.
\end{proof}
The representation $\Theta_{G_{\glnidx},\chi}$ is the exceptional representation corresponding to the exceptional character $\chi$.

We describe the contragradient exceptional representation.
\begin{claim}\label{claim:small rep is not self dual}
We have $\Theta_{G_{\glnidx},\chi}^{\wedge}=\Theta_{G_{\glnidx},\rconj{\mathbf{w_0}}\chi^{\wedge}}$. Furthermore,
the restriction of $\rconj{\mathbf{w_0}}\chi^{\wedge}$ to $\cover{T}_{\glnidx+1}(F)^2$ agrees with the restriction of $\chi$ multiplied by some nongenuine character $\lambda$ 
of $\cover{T}_{\glnidx+1}(F)^2$.
\end{claim}
\begin{proof}
Put $\chi'=\rconj{\mathbf{w_0}}\chi^{\wedge}$. Proposition~\ref{proposition:periodicity theorem} implies that $\Theta_{G_{\glnidx},\chi}^{\wedge}$ is an irreducible quotient of
$V(\chi')$. Let $\chi_0$ (resp. $\chi'_0$) be a character of $T_{\glnidx+1}(F)^2$ such that $\chi|_{\cover{T}_{\glnidx+1}(F)^2}$ (resp. $\chi'|_{\cover{T}_{\glnidx+1}(F)^2}$) is the extension of $\chi_0$ (resp. $\chi'_0$). If $t\in T_{\glnidx+1}(F)^2$, Equality~\eqref{eq:formula for conjugation by w_0 with convenient coordinates} implies
$\chi'_0=\chi_0(t\cdot \beta_1^{\vee}(\Upsilon(t)))$. Denote $\lambda(t)=\chi_0(\beta_1^{\vee}(\Upsilon(t)))$. Since $\beta_1^{\vee}(x)\in T_{\glnidx+1}(F)^2$ for all $x\in F^*$, $\lambda$ is a character of $T_{\glnidx+1}(F)^2$ and
$\chi'_0=\lambda\cdot\chi_0$. It follows that $\chi'|_{\cover{T}_{\glnidx+1}(F)^2}=\lambda\cdot(\chi|_{\cover{T}_{\glnidx+1}(F)^2})$. In particular,
$\chi'$ is an exceptional character. Then by Proposition~\ref{proposition:periodicity theorem}, $\Theta_{G_{\glnidx},\chi'}$ is the unique irreducible quotient of $V(\chi')$. Hence
$\Theta_{G_{\glnidx},\chi}^{\wedge}=\Theta_{G_{\glnidx},\chi'}$.
\end{proof}
\begin{remark}\label{remark:small rep is not self dual}
In contrast with $\Theta_{SO_{2\glnidx+1}}$ of \cite{BFG}, the representation $\Theta_{G_{\glnidx},\chi}$ is not self-dual.
\end{remark}

The next proposition describes an exceptional representation in terms of exceptional representations of $\GL{k}$ and $G_{\glnidx-k}$.
We briefly recall the construction of exceptional representations of Kazhdan and Patterson \cite{KP} (Section~I.1 and Theorem~I.2.9). Let $\chi_1$ be a genuine character of $C_{\cover{T}_{\GL{k}}(F)}=\cover{T}_{\GL{k}}(F)^2C_{\CGLF{k}{F}}$. Call $\chi_1$ exceptional if $\chi_1(\mathfrak{s}(\eta_i^{\vee}(x^2)\eta_{i+1}^{\vee}(x^{-2})))=|x|$ for all $1\leq i<k$ and $x\in F^*$. In this case
the exceptional representation $\Theta_{\GL{k},\chi_1}$ is the unique irreducible quotient of
$\induced{\cover{B}_{\GL{k}}(F)}{\CGLF{k}{F}}{\rho(\chi_1)}$. It is also the unique irreducible subrepresentation of $\induced{\cover{B}_{\GL{k}}(F)}{\CGLF{k}{F}}{\rho(\rconj{\mathbf{w_0'}}\chi_1)}$,
where $\mathbf{w_0'}$ is the longest element of the Weyl group of $\GL{k}$. Note that for $1\leq i<\glnidx$,
$\mathfrak{s}(\eta_i^{\vee}(x^2)\eta_{i+1}^{\vee}(x^{-2}))={\alpha_{i+1}^{\vee}}^*(x^{\mathfrak{l}(\alpha_{i+1})})$.

Let $\mathbf{w_0''}$ be the longest element of $W_{\glnidx-k}$ (the Weyl group of $G_{\glnidx-k}$).
Also recall the representation $\chi_1\odot\chi_2$ defined in Section~\ref{subsubsection:tensor product}.

The following claim relates an exceptional character of $C_{\cover{T}_{\glnidx+1}(F)}$ to a pair of such characters of
$C_{\cover{T}_{\GL{k}}(F)}$ and $C_{\cover{T}_{\glnidx-k+1}(F)}$.
\begin{claim}\label{claim:relating a character to two characters}
Let $\chi$ be an exceptional character and let $0\leq k\leq\glnidx$. There are exceptional characters
$\chi_1$ and $\chi_2$ of $C_{\cover{T}_{\GL{k}}(F)}$ and $C_{\cover{T}_{\glnidx-k+1}(F)}$ such that
\begin{align}\label{eq:good condition for relating an exceptional character to two such}
\rconj{\mathbf{w_0}}\chi|_{\cover{T}_{\glnidx+1}(F)^2}=\rconj{\mathbf{w_0'}}\chi_1\odot\rconj{\mathbf{w_0''}}\chi_2.
\end{align}
\end{claim}
\begin{proof}
Since $\cover{T}_{\glnidx+1}(F)^2$ is split under $\mathfrak{s}$, one can regard $\chi$ as a character of
$T_{\glnidx+1}(F)^2$ satisfying $\chi(\alpha^{\vee}(x^{\mathfrak{l}(\alpha)}))=|x|$ for all $\alpha\in\Delta_{G_{\glnidx}}$ and $x\in F^*$. It is enough to show there are characters $\chi_1$ and $\chi_2$ of
$T_{\GL{k}}(F)^2$ and $T_{\glnidx-k+1}(F)^2$ with the following properties.
\begin{enumerate}[leftmargin=*]
\item \label{item:claim relating item 1}$\chi_1(\eta_i^{\vee}(x^2)\eta_{i+1}^{\vee}(x^{-2}))=|x|$ for all $1\leq i<k$ and $x\in F^*$,
\item \label{item:claim relating item 2}$\chi_2(\alpha_i^{\vee}(x^{\mathfrak{l}(\alpha_i)}))=|x|$ for all $k+2\leq i\leq\glnidx+1$ and $x\in F^*$,
\item \label{item:claim relating item 3}
$\rconj{\mathbf{w_0}}\chi|_{T_{\glnidx+1}(F)^2}=\rconj{\mathbf{w_0'}}\chi_1\otimes\rconj{\mathbf{w_0''}}\chi_2$.
\end{enumerate}
Indeed, given such characters, one can extend them to genuine characters of $\cover{T}_{\GL{k}}(F)^2$ and $\cover{T}_{\glnidx-k+1}(F)^2$, then extend them again, not necessarily uniquely (depending on the parity of $k$ and $\glnidx$), to exceptional characters of the corresponding centers.

If $a=\prod_{i=1}^k{\eta_i^{\vee}}(a_i^2)$ and $t=\prod_{i=1}^{\glnidx-k}\beta_i^{\vee}(t_i^2)\beta_{\glnidx-k+1}^{\vee}(t_{\glnidx-k+1})$, using
\eqref{eq:conjugation of torus by w_0}-\eqref{eq:image of standard T_n-k+1 in GSpin} and 
\eqref{eq:formula for conjugation by w_0 with convenient coordinates} we get
\begin{align*}
\rconj{\mathbf{w_0}}(at)=\begin{dcases}
\prod_{i=1}^{k}\eta_i^{\vee}(a_i^{-2})\prod_{i=1}^{\glnidx-k}\beta_i^{\vee}(t_i^{-2}t_1^4\det{a}^{-2})\beta_{\glnidx-k+1}^{\vee}(t_{\glnidx-k+1}^{-1}t_1^2\det{a}^{-1})&k<\glnidx,\\
\prod_{i=1}^{\glnidx}\eta_i^{\vee}(a_i^{-2})\beta_1^{\vee}(\det{a}^{-1}t_1)&k=\glnidx.
\end{dcases}
\end{align*}
Set $\eta(a)=\chi(\prod_{i=1}^{\glnidx-k}\beta_i^{\vee}(\det{a}^{-2})\beta_{\glnidx-k+1}^{\vee}(\det{a}^{-1}))$,
$\chi_1=\eta\cdot\rconj{\mathbf{w_0'}}(\chi|_{T_{\GL{k}}(F)^2})^{-1}$ and $\chi_2=\chi|_{T_{\glnidx-k+1}(F)^2}$.
Clearly these characters satisfy the above properties.
\end{proof}
In general if $0<k\leq\glnidx$ and $\chi_1$ and $\chi_2$ are exceptional characters of $C_{\cover{T}_{\GL{k}}(F)}$ and $C_{\cover{T}_{\glnidx-k+1}(F)}$, 
we have the representation $I(\Theta_{\GL{k},\chi_1},\Theta_{G_{\glnidx-k},\chi_2})^{\star}$ defined in Section~\ref{subsubsection:tensor product},
\begin{align*}
I(\Theta_{\GL{k},\chi_1},\Theta_{G_{\glnidx-k},\chi_2})^{\star}=
\cinduced{p^{-1}(\GLF{k}{F}^{\star}\times G_{\glnidx-k}(F)^{\star})}{\cover{M}_k(F)}
{\Theta_{\GL{k},\chi_1}^{\star}\otimes \Theta_{G_{\glnidx-k},\chi_2}^{\star}}.
\end{align*}

The following two results are analogs of \cite{BFG} (Theorem~2.3 and Proposition~2.4).
\begin{proposition}\label{proposition:Jacquet along standard nonminimal unipotent}
Let $\chi$ be an exceptional character and let $0<k\leq\glnidx$. Then for any exceptional characters $\chi_1$ and $\chi_2$ of $C_{\cover{T}_{\GL{k}}(F)}$ and $C_{\cover{T}_{\glnidx-k+1}(F)}$ satisfying \eqref{eq:good condition for relating an exceptional character to two such},
\begin{align*}
j_{U_k}(\Theta_{G_{\glnidx},\chi})\subset I(\Theta_{\GL{k},\chi_1},\Theta_{G_{\glnidx-k},\chi_2})^{\star}.
\end{align*}
\end{proposition}
\begin{proof}
We start with computing the Jacquet module along $N_{\GL{k}}\times N_{\glnidx-k}<M_k$ of both representations. The double coset space
\begin{align*}
\rmodulo{\lmodulo{(\GLF{k}{F}^{\star} G_{\glnidx-k}(F)^{\star})}{M_k(F)}}{(T_{\glnidx+1}(F)N_{\GL{k}}(F)N_{\glnidx-k}(F))}
\end{align*}
has one element.
According to the Geometric Lemma of Bernstein and Zelevinsky \cite{BZ2} (Theorem~5.2),
\begin{align}\label{eq:jacquet module of rhs}
&j_{N_{\GL{k}}N_{\glnidx-k}}(I(\Theta_{\GL{k},\chi_1},\Theta_{G_{\glnidx-k},\chi_2})^{\star})\\&=\nonumber\cinduced{p^{-1}
(T_{\GL{k}}(F)^{\star}\times T_{\glnidx-k+1}(F)^{\star})}{\cover{T}_{\glnidx+1}(F)}
{j_{N_{\GL{k}}}^{\star}(\Theta_{\GL{k},\chi_1}^{\star})\otimes j_{N_{\glnidx-k}}^{\star}(\Theta_{G_{\glnidx-k},\chi_2}^{\star})}.
\end{align}
Here $j_{N_{\GL{k}}}^{\star}$ is the Jacquet functor taking representations of $\CGLF{k}{F}^{\star}$ to representations of
$\cover{T}_{\GL{k}}(F)^{\star}$ and $j_{N_{\glnidx-k}}^{\star}$ is defined similarly. We have
\begin{align*}
j_{N_{\GL{k}}}^{\star}(\Theta_{\GL{k},\chi_1}^{\star})=(j_{N_{\GL{k}}}(\Theta_{\GL{k},\chi_1}))^{\star}=\rho(\rconj{\mathbf{w_0'}}\chi_1)^{\star},
\end{align*}
where for the first equality we used the fact that $N_{\GL{k}}(F)<\CGLF{k}{F}^{\star}$, the second equality follows from \cite{KP} (Theorem~I.2.9).
Similarly, using Proposition~\ref{proposition:periodicity theorem} we get
$j_{N_{\glnidx-k}}^{\star}(\Theta_{G_{\glnidx-k},\chi_2}^{\star})=\rho(\rconj{\mathbf{w_0''}}\chi_2)^{\star}$. 
Applying Lemma~\ref{lemma:induced representation composition factors for torus} to the \rhs\ of \eqref{eq:jacquet module of rhs} implies
\begin{align*}
&j_{N_{\GL{k}}N_{\glnidx-k}}(I(\Theta_{\GL{k},\chi_1},\Theta_{G_{\glnidx-k},\chi_2})^{\star})
=m\bigoplus_{\chi'}\rho(\chi'),
\end{align*}
where $m>0$ is an integer (depending only on $k$ and $\glnidx$), and the summation is over all genuine characters $\chi'$ of $C_{\cover{T}_{\glnidx+1}(F)}$ such that
\begin{align*}
\chi'|_{\cover{T}_{\glnidx+1}(F)^2}=\rconj{\mathbf{w_0'}}\chi_1\odot\rconj{\mathbf{w_0''}}\chi_2=\rconj{\mathbf{w_0}}\chi|_{\cover{T}_{\glnidx+1}(F)^2}
\end{align*}
(the second equality is \eqref{eq:good condition for relating an exceptional character to two such}). Thus $\rconj{\mathbf{w_0}}\chi$ appears in the summation at least once.

Since $j_{N_{\GL{k}}N_{\glnidx-k}}(I(\Theta_{\GL{k},\chi_1},\Theta_{G_{\glnidx-k},\chi_2})^{\star})$ is semisimple and the Jacquet functor is exact, $j_{N_{\GL{k}}N_{\glnidx-k}}(\mathcal{V}_0)$ is semisimple for any $\mathcal{V}_0\subset I(\Theta_{\GL{k},\chi_1},\Theta_{G_{\glnidx-k},\chi_2})^{\star}$. Thus there is an irreducible $\mathcal{V}\subset I(\Theta_{\GL{k},\chi_1},\Theta_{G_{\glnidx-k},\chi_2})^{\star}$ such that
$\rho(\rconj{\mathbf{w_0}}\chi)$ is a quotient of $j_{N_{\GL{k}}N_{\glnidx-k}}(\mathcal{V})$. Set
\begin{align*}
I^{\cover{M}_k}(\rho(\rconj{\mathbf{w_0}}\chi))=\induced{p^{-1}(T_{\glnidx+1}(F)N_{\GL{k}}(F)N_{\glnidx-k}(F))}{\cover{M}_k(F)}{\rho(\rconj{\mathbf{w_0}}\chi)}.
\end{align*}
Then Frobenius reciprocity shows
\begin{align}\label{eq:Frobenius on Jacquet U_k claim}
Hom_{\cover{M}_{k}(F)}(\mathcal{V},I^{\cover{M}_k}(\rho(\rconj{\mathbf{w_0}}\chi))) =Hom_{\cover{T}_{\glnidx+1}(F)}(j_{N_{\GL{k}}N_{\glnidx-k}}(\mathcal{V}),\rho(\rconj{\mathbf{w_0}}\chi))\ne0
\end{align}
whence $\mathcal{V}\subset I^{\cover{M}_k}(\rho(\rconj{\mathbf{w_0}}\chi))$.
According to the Langlands Quotient Theorem \cite{BJ} the representation
$I^{\cover{M}_k}(\rho(\rconj{\mathbf{w_0}}\chi))$ has a unique irreducible subrepresentation, which is $\mathcal{V}$.

Let us turn to $j_{U_k}(\Theta_{G_{\glnidx},\chi})$. Proposition~\ref{proposition:periodicity theorem} implies
\begin{align*}
j_{N_{\GL{k}}N_{\glnidx-k}}(j_{U_{k}}(\Theta_{G_{\glnidx},\chi}))=
j_{N_{\glnidx}}(\Theta_{G_{\glnidx},\chi})=\rho(\rconj{\mathbf{w_0}}\chi).
\end{align*}
The representation $j_{U_{k}}(\Theta_{G_{\glnidx},\chi})$, as a Jacquet module of a quotient of $V(\chi)$ with respect to a non-minimal unipotent radical, does not have a cuspidal constituent. Hence, since
$j_{N_{\GL{k}}N_{\glnidx-k}}(j_{U_{k}}(\Theta_{G_{\glnidx},\chi}))$ is irreducible, so is $j_{U_{k}}(\Theta_{G_{\glnidx},\chi})$. Now \eqref{eq:Frobenius on Jacquet U_k claim} with $j_{U_k}(\Theta_{G_{\glnidx},\chi})$ instead of $\mathcal{V}$ shows $j_{U_k}(\Theta_{G_{\glnidx},\chi})\subset I^{\cover{M}_k}(\rho(\rconj{\mathbf{w_0}}\chi))$. Thus $j_{U_k}(\Theta_{G_{\glnidx},\chi})=\mathcal{V}\subset I(\Theta_{\GL{k},\chi_1},\Theta_{G_{\glnidx-k},\chi_2})^{\star}$.
\end{proof}

\begin{corollary}\label{corollary:Jacquet along standard nonminimal unipotent}
Let $\chi$ be an exceptional character and let $0<k\leq\glnidx$. For any exceptional characters $\chi_1,\chi_3$ of $C_{\cover{T}_{\GL{k}}(F)}$ and $\chi_2,\chi_4$ of $C_{\cover{T}_{\glnidx-k+1}(F)}$ such that $(\chi_1,\chi_2)$ satisfy \eqref{eq:good condition for relating an exceptional character to two such} with respect to $\chi$, and $(\rconj{\mathbf{w_0'}}\chi_3^{\wedge},\rconj{\mathbf{w_0''}}\chi_4^{\wedge})$ satisfy \eqref{eq:good condition for relating an exceptional character to two such} with respect to $\rconj{\mathbf{w_0}}\chi^{\wedge}$,
\begin{align*}
&Hom_{\cover{G}_{\glnidx}(F)}(\Theta_{G_{\glnidx},\chi},\induced{\cover{Q}_k(F)}{\cover{G}_{\glnidx}(F)}{I(\Theta_{\GL{k},\chi_1},\Theta_{G_{\glnidx-k},\chi_2})^{\star}})\ne0,\\
&Hom_{\cover{G}_{\glnidx}(F)}(\induced{\cover{Q}_k(F)}{\cover{G}_{\glnidx}(F)}{I(\Theta_{\GL{k},\chi_3},\Theta_{G_{\glnidx-k},\chi_4})^{\star}},\Theta_{G_{\glnidx},\chi})\ne0.
\end{align*}
\end{corollary}
\begin{proof}
The first assertion follows from Proposition~\ref{proposition:Jacquet along standard nonminimal unipotent} using the Frobenius reciprocity. Dualizing,
\begin{align*}
&Hom_{\cover{G}_{\glnidx}(F)}(\induced{\cover{Q}_k(F)}{\cover{G}_{\glnidx}(F)}{(I(\Theta_{\GL{k},\chi_1},\Theta_{G_{\glnidx-k},\chi_2})^{\star})^{\wedge}},\Theta_{G_{\glnidx},\chi}^{\wedge})\ne0.
\end{align*}
In general if $\xi$ is a genuine representation of $\cover{H}(F)$ where $H<G_{\glnidx}$, $(\xi^{\star})^{\wedge}=(\xi^{\wedge})^{\star}$ (because $\cover{H}(F)^{\star}$ contains an open neighborhood of the identity of $\cover{H}(F)$). Applying Claim~\ref{claim:small rep is not self dual} and its analog for $\GL{\glnidx}$ we get
\begin{align*}
(I(\Theta_{\GL{k},\chi_1},\Theta_{G_{\glnidx-k},\chi_2})^{\star})^{\wedge}=
I(\Theta_{\GL{k},\rconj{\mathbf{w_0'}}\chi_1^{\wedge}},\Theta_{G_{\glnidx-k},\rconj{\mathbf{w_0''}}\chi_2^{\wedge}})^{\star}.
\end{align*}
Now replace $\chi$ with $\rconj{\mathbf{w_0}}\chi^{\wedge}$ and set $\chi_1=\rconj{\mathbf{w_0'}}\chi_3^{\wedge}$ and $\chi_2=\rconj{\mathbf{w_0''}}\chi_4^{\wedge}$. We see that 
\[
Hom_{\cover{G}_{\glnidx}(F)}(\induced{\cover{Q}_k(F)}{\cover{G}_{\glnidx}(F)}{I(\Theta_{\GL{k},\chi_3},\Theta_{G_{\glnidx-k},\chi_4})^{\star}},\Theta_{G_{\glnidx},\chi})\ne0.\qedhere
\]
\end{proof}
In the case $k=\glnidx$, we can strengthen the results of Proposition~\ref{proposition:Jacquet along standard nonminimal unipotent} and Corollary~\ref{corollary:Jacquet along standard nonminimal unipotent} and obtain a result more similar to that of \cite{BFG}.
\begin{claim}\label{claim:special case k=n}
There are unique exceptional characters $\chi_1,\chi_2,\chi_3,\chi_4$ such that
\begin{align*}
&j_{U_{\glnidx}}(\Theta_{G_{\glnidx},\chi})=\Theta_{\GL{\glnidx},\chi_1}\otimes\Theta_{G_{0},\chi_2},\\
&Hom_{\cover{G}_{\glnidx}(F)}(\Theta_{G_{\glnidx},\chi},\induced{\cover{Q}_{\glnidx}(F)}{\cover{G}_{\glnidx}(F)}{\Theta_{\GL{\glnidx},\chi_1}\otimes\Theta_{G_{0},\chi_2}})\ne0,\\
&Hom_{\cover{G}_{\glnidx}(F)}(\induced{\cover{Q}_{\glnidx}(F)}{\cover{G}_{\glnidx}(F)}{\Theta_{\GL{\glnidx},\chi_3}\otimes\Theta_{G_{0},\chi_4}},\Theta_{G_{\glnidx},\chi})\ne0.
\end{align*}
\end{claim}
\begin{proof}
In this case $C_{\cover{T}_{\glnidx+1}(F)}$ is the product of $C_{\cover{T}_{\GL{\glnidx}}(F)}$ and $\cover{T}_1(F)$ with amalgamated $\mu_2$. Therefore, there are unique characters $\chi_1$ and $\chi_2$ such that $\rconj{\mathbf{w_0}}\chi=\rconj{\mathbf{w_0'}}\chi_1\otimes\chi_2$. Then the first assertion follows as in \cite{BFG} (Theorem~2.3), by calculating the Jacquet modules and using the fact that the representation induced from $\rho(\rconj{\mathbf{w_0}}\chi)$ to $\cover{M}_{\glnidx}(F)$ has a unique irreducible subrepresentation. The other assertions follow from the Frobenius reciprocity. 
\end{proof}

Due to global reasons (see Proposition~\ref{proposition:global small rep} below), it will be necessary to compute the constant $c(\mathbf{w_0},\chi)$ defined in Claim~\ref{claim:props of intertwining and GK formula}. We have the following claim. 
\begin{claim}\label{claim:values of exceptional unramified character on a alpha}
Let $\chi$ be an exceptional character, which is unramified. Then
\begin{align*}
c(\mathbf{w_0},\chi)=\prod_{2\leq i<j\leq\glnidx+1}\frac{(1-q^{-1-j+i})(1-q^{-1+j+i-2(\glnidx+2)})}
{(1-q^{-j+i})(1-q^{j+i-2(\glnidx+2)})}\prod_{2\leq i\leq\glnidx+1}\frac{(1-q^{-1-\glnidx-2+i})}{(1-q^{-\glnidx-2+i})}.
\end{align*}
\end{claim}
\begin{proof}
We compute $\chi(a_{\alpha})$ for an arbitrary $\alpha\in\Sigma_{G_{\glnidx}}^+$. We will show
\begin{align*}
\chi(a_{\alpha})=\begin{cases}
q^{-j+i}&\alpha=\epsilon_i-\epsilon_j,\quad 2\leq i<j\leq\glnidx+1,\\
q^{j+i-2(\glnidx+2)}&\alpha=\epsilon_i+\epsilon_j,\quad 2\leq i<j\leq\glnidx+1,\\
q^{-\glnidx-2+i}&\alpha=\epsilon_i,\quad 2\leq i\leq\glnidx+1.
\end{cases}
\end{align*}
The formula for $c(\mathbf{w_0},\chi)$ clearly follows from this.

In general if $\alpha^{\vee}(x)=\prod_{l=1}^m\alpha_{i_l}^{\vee}(x_l)$ where $\alpha_{i_l}\in\Delta_{G'_{\glnidx+1}}$, i.e.,
$1\leq i_l\leq\glnidx+1$ for all $1\leq l\leq m$, and for each $l$ either $x_l\in{\Fsquares}$ or $i_l=\glnidx+1$, then $\prod_{l=1}^m\alpha_{i_l}^{\vee}(x_l)\in T_{\glnidx+1}(F)^2$. Then by applying $\mathfrak{s}$ we obtain
${\alpha^{\vee}}^*(x)=\prod_{l=1}^m{\alpha_{i_l}^{\vee}}^*(x_l)$, whence
$\chi({\alpha^{\vee}}^*(x))=\prod_{l=1}^m\chi({\alpha_{i_l}^{\vee}}^*(x_l))$.

Start with the computation of $\chi(a_{\alpha})$ for $\alpha=\epsilon_i-\epsilon_j$, $2\leq i<j\leq\glnidx+1$.
Since $\alpha^{\vee}=\prod_{l=i}^{j-1}\alpha_l^{\vee}$ and $\mathfrak{l}(\alpha)=2$, ${\alpha^{\vee}}(\varpi^2)=\prod_{l=i}^{j-1}{\alpha_l^{\vee}}(\varpi^2)$. The definition of $\chi$ implies $\chi({{\alpha_l}^{\vee}}^*(\varpi^2))=q^{-1}$, giving the result.

Next consider $\alpha=\epsilon_{i}$, $2\leq i\leq\glnidx+1$. We have $\mathfrak{l}(\alpha)=1$ and $\alpha^{\vee}(\varpi)=\prod_{l=i}^{\glnidx}\alpha_{l}^{\vee}(\varpi^2)\alpha_{\glnidx+1}^{\vee}(\varpi)$ (if $\glnidx=1$ there is no product), whence $\chi(a_{\alpha})=q^{-\glnidx-2+i}$.

For $\alpha=\epsilon_{i}+\epsilon_{j}$ ($2\leq i<j\leq\glnidx+1$),
$\alpha^{\vee}(\varpi^2)=(\epsilon_i-\epsilon_j)^{\vee}(\varpi^2)\epsilon_j^{\vee}(\varpi^2)$. Applying $\mathfrak{s}$ gives
${\alpha^{\vee}}^*(\varpi^2)={(\epsilon_i-\epsilon_j)^{\vee}}^*(\varpi^2){\epsilon_j^{\vee}}^*(\varpi^2)$. Now the result follows from the previous calculations. 
\end{proof}

\subsubsection{Vanishing results}\label{subsubsection:vanishing and results}
As described in the introduction, the main motivation for studying exceptional representations is their applications.
The remarkable (local) property of these representations is the vanishing of many twisted Jacquet modules.

In this section $F$ is a \nonarchimedean\ field of odd residual characteristic. Fix a nontrivial additive character $\psi$ of $F$. For a column $b \in F^l$ define the ``length" of $b$, $\ell(b)$, with respect to the symmetric bilinear form corresponding to $J_{l}$, $\ell(b)=\transpose{b}J_{l}b$. If $U<N_{\glnidx}$, $u\in U(F)$ and $\alpha\in\Sigma_{G_{\glnidx}}^+$, denote by $u_{\alpha}$ the projection of $u$ on $\mathcal{U}_{\alpha}$.

We recall the notion of unipotent classes and their corresponding unipotent subgroups and characters. Since the unipotent subgroups
of $G_{\glnidx}$ are in bijection with those of $SO_{2\glnidx+1}$, we can use the description of Bump, Friedberg and Ginzburg \cite{BFG} (Section~4). For a general reference see \cite{CM,Cr}. See also Ginzburg \cite{G2}.

A partition of $k$ is an $m$-tuple $r_1\geq\ldots\geq r_m>0$ such that
$r_1+\ldots+r_m=k$. There is a natural partial order on the set of partitions of $k$, $(r_1,\ldots,r_m)\geq (r'_1,\ldots,r'_{m'})$ if
$\sum_{i=1}^lr_i\geq\sum_{i=1}^lr'_i$ for all $1\leq l\leq\min(m,m')$. We denote $(r_1,\ldots,r_m)\gtorncw(r'_1,\ldots,r'_{m'})$ if
$(r_1,\ldots,r_m)$ is greater than or non comparable with $(r'_1,\ldots,r'_{m'})$.

A unipotent class $\UnipotentOrbit=(r_1,\ldots,r_m)$ of $G_{\glnidx}$ corresponds to a partition of $2\glnidx+1$ in which any even number appears with even multiplicity. 
Write the integers occurring in the multiset $\bigcup_{i=1}^m\{r_i-2j+1\}_{j=1}^{r_i}$ in a decreasing order and let
$l_1\geq\ldots\geq l_{\glnidx}\geq0$ be the first $\glnidx$ numbers (these are the $\glnidx$ largest ones). Define the one-parameter subgroup $h_{\UnipotentOrbit}$ by the image of
\begin{align*}
h_{\UnipotentOrbit}=\prod_{i=1}^{\glnidx}({\eta_i^{\vee}})^{l_i}.
\end{align*}
The group $h_{\UnipotentOrbit}$ acts on $N_{\glnidx}$ by conjugation. For any $\alpha\in\Sigma_{G_{\glnidx}}^+$ there exists $j_{\alpha}\geq0$ such that $h_{\UnipotentOrbit}(t)n_{\alpha}(x)h_{\UnipotentOrbit}(t)^{-1}=n_{\alpha}(t^{j_{\alpha}}x)$ for all $x$ and $t$. Let $V_{\UnipotentOrbit}<N_{\glnidx}$ be the unipotent subgroup generated by those $n_{\alpha}$ for which $j_{\alpha}\geq2$ (i.e., $V_{\UnipotentOrbit}(F)$ is generated by $\setof{n_{\alpha}(x)}{j_{\alpha}\geq2,x\in F}$). Let $V_{\UnipotentOrbit}^{\matha}$ be the quotient of $V_{\UnipotentOrbit}$ by its derived group. Recall (Section~\ref{subsection:the groups}) that $C(G_{\glnidx},h_{\UnipotentOrbit})$ denotes the centralizer of
$h_{\UnipotentOrbit}$ in $G_{\glnidx}$. The centralizer $C(G_{\glnidx},h_{\UnipotentOrbit})$ acts on $V_{\UnipotentOrbit}^{\matha}$ by conjugation. The stabilizer of $x\in V_{\UnipotentOrbit}^{\matha}$ under this action is denoted $St_x$ and its connected component by $St_x^0$.

Any character of $V_{\UnipotentOrbit}(F)$ is the pull back of a character of $V_{\UnipotentOrbit}^{\matha}(F)$, and because $V_{\UnipotentOrbit}^{\matha}(F)$ is abelian, such a character can be identified with a point in $V_{\UnipotentOrbit}^{\matha}(F)$. If $x\in V_{\UnipotentOrbit}^{\matha}(F)$, let $\psi_x$ be the corresponding character.

Let $\overline{F}$ be the algebraic closure of $F$. The action of $C(G_{\glnidx}(\overline{F}),h_{\UnipotentOrbit}(\overline{F}^*))$ on $V_{\UnipotentOrbit}^{\matha}(\overline{F})$ has an open orbit. For $\varepsilon\in V_{\UnipotentOrbit}^{\matha}(\overline{F})$ in this orbit,
$St_{\varepsilon}(\overline{F})^{0}$ is a reductive group.
Let $b\in V_{\UnipotentOrbit}^{\matha}(F)$. The character $\psi_b$ is called generic if $b$ belongs to the open orbit.

\begin{remark}
In \cite{BFG} the notion of a generic character was restricted to allow only $F$-split stabilizers. 
For example, if $St_{\varepsilon}(\overline{F})^{0}=SO_{2l}(\overline{F})$, then $\psi_b$ is generic if $St_{b}(F)^{0}$ is the $F$-split group $SO_{2l}(F)$. A result of \cite{BFG2} (Proposition~3) indicates this notion can be relaxed by considering \quasisplit\ stabilizers. 
\end{remark}

Recall that $\UnipotentOrbit_{0}=(2^{\glnidx}1)$ if $\glnidx$ is even, otherwise $\UnipotentOrbit_{0}=(2^{\glnidx-1}1^3)$. Let $\chi$ be an exceptional character of $C_{\cover{T}_{\glnidx+1}(F)}$.
In this section we prove Theorem~\ref{theorem:local vanishing result}. Namely, for any $\UnipotentOrbit\gtorncw\UnipotentOrbit_0$ and generic character $\psi_b$ of $V_{\UnipotentOrbit}(F)$,
\begin{align}\label{eq:short statement of local vanishing thm}
j_{V_{\UnipotentOrbit},\psi_b}(\Theta_{G_{\glnidx},\chi})=0.
\end{align}
(As in Section~\ref{subsubsection:defs and basic properties} we omitted $\mathfrak{s}$.)
Let $\UnipotentOrbit_1=(31^{2\glnidx-2})$. If a unipotent class $\UnipotentOrbit$ satisfies $\UnipotentOrbit\gtorncw\UnipotentOrbit_0$, then
$\UnipotentOrbit\geq\UnipotentOrbit_1$. This was used by Bump, Friedberg and Ginzburg \cite{BFG} (Theorem~4.2) to reduce the (global) vanishing results to a statement on $\UnipotentOrbit_1$ (\cite{BFG} Proposition~4.3). We follow (a local version of) their arguments (see \cite{BFG2}, proof of Proposition~6). In Lemmas~\ref{lemma:thm minimal case vanishing n=1}-\ref{lemma:twisted Jacquet modules vanish on small representations} we prove $j_{V_{\UnipotentOrbit_1},\psi_b}(\Theta_{G_{\glnidx},\chi})=0$. The proof of \eqref{eq:short statement of local vanishing thm} will follow from this using properties of Jacquet modules.

In the case of $\UnipotentOrbit_1$, the corresponding unipotent subgroup $V_{\UnipotentOrbit_1}$ is $U_1$.
If $u\in U_1(F)$, let $r(u)\in F^{2\glnidx-1}$ be the row vector given by
\begin{align*}
r(u)=(u_{\epsilon_2-\epsilon_3},\ldots,u_{\epsilon_2-\epsilon_{\glnidx+1}},u_{\epsilon_{2}},u_{\epsilon_2+\epsilon_{\glnidx+1}},\ldots,u_{\epsilon_2+\epsilon_3}).
\end{align*}
Since $U_1$ is abelian, for any column $b\in F^{2\glnidx-1}$ we have a character of $U_1(F)$ given by $\psi_b(u)=\psi(r(u)b)$. Also $C(G_{\glnidx},h_{\UnipotentOrbit_1})=M_1$.
The points belonging to the open orbit are those $b\in U_1(\overline{F})$ with $\ell(b)\ne0$. Indeed, there are only three orbits, $\setof{b}{\ell(b)\ne0}$, $\setof{b}{\ell(b)=0}$ and $\{b=0\}$. if $\ell(b)=0$, $St_b(\overline{F})^{0}$ is not reductive, because it contains a unipotent radical of $G_{\glnidx-1}$, which is normal in $St_b(\overline{F})^{0}$.

Let $b\in U_1(F)$. Then $\psi_b$ is generic if and only if $\ell(b)\ne0$. For a generic $\psi_b$,
$St_{b}(F)^0\isomorphic GSpin_{2(\glnidx-1)}(F)$, a \quasisplit\ group (in \cite{BFG,BFG2} this was $SO_{2(\glnidx-1)}(F)$), where $GSpin_{2(\glnidx-1)}<G_{\glnidx-1}<M_1$. To see this start with the root subgroups contained in the stabilizer, these are found using the identification of $U_1$ with a unipotent radical $U_1'$ of $SO_{2\glnidx+1}$ (see Section~\ref{subsection:properties of GSpin}). To find the coroots write $t\in T_{\glnidx+1}(F)$ in the form \eqref{eq:convenient coordinates of torus}, then consider the action of $\prod_{i=1}^{\glnidx}\eta_{i}^{\vee}(a_i)$ on $U_1'(F)$ ($\beta_1^{\vee}(t_1)\in C_{G_{\glnidx}(F)}$). 
\begin{lemma}\label{lemma:thm minimal case vanishing n=1}
Assume $\glnidx=1$ and let $\varphi$ be a linear functional on the space of $\Theta_{G_{1},\chi}$. If $\varphi(n_{\alpha_2}^*(x)v)=\psi(x)\varphi(v)$ for all $x\in F$, then $\varphi=0$.
\end{lemma}
\begin{proof}
The subgroup generated by $n_{\alpha_2}(x)$ is $U_1(F)=N_{1}(F)$. The functional $\varphi$ is the usual Whittaker functional with respect to $N_1$ and $\psi$. We will show that $V(\rconj{\mathbf{w_0}}\chi)$ contains an irreducible genuine representation, which does not afford such a functional. By Proposition~\ref{proposition:periodicity theorem} this representation must be $\Theta_{G_1,\chi}$, which gives the result.

In this case $\chi$ is a character of $\cover{T}_{2}(F)$ (see Section~\ref{subsubsection:representations of the torus}). By Claim~\ref{claim:cover minimal cases}
the cover splits, hence $\chi$ can be considered as the extension of a character $\chi'$ of $T_{2}(F)$. The character $\chi'$ must satisfy
$\chi'({\alpha_2^{\vee}}(t_2))=|t_2|$ for all $t_2\in F^*$.
Since $w_0=w_{\alpha_2}$, Equalities~\eqref{eq:image of convenient coordinates of the torus in GSpin}-\eqref{eq:formula for delta Borel with convenient coordinates} imply 
\begin{align*}
\delta_{B_1(F)}^{1/2}\rconj{\mathbf{w_0}}\chi'({\alpha_1^{\vee}}(t_1){\alpha_2^{\vee}}(t_2))=
|t_1|^{1/2}\chi'({\alpha_1^{\vee}}(t_1))=\eta(\Upsilon({\alpha_1^{\vee}}(t_1){\alpha_2^{\vee}}(t_2))),
\end{align*}
where $\eta$ is some character of $F^*$.
Hence the nongenuine function $f(g)=\eta(\Upsilon(g))$ belongs to $V(\rconj{\mathbf{w_0}}\chi)$ and by its
definition, spans a one-dimensional subspace which is also a $G_1(F)$-module. Because the cover splits, $f$ can be extended to a genuine function hence this is also a genuine $\cover{G}_1(F)$-module.
The action of $n_{\alpha_2}^*(x)$ on this subspace is trivial, completing the proof.
\end{proof}

In the case $\glnidx=2$, $U_1(F)$ is generated by $n_{\alpha_2}(x)$, $n_{\alpha_2+\alpha_3}(y)$ and $n_{\alpha_2+2\alpha_3}(z)$.
\begin{lemma}\label{lemma:minimal cases vanishing}
Assume $\glnidx=2$ and let $\varphi$ be a linear functional on the space of $\Theta_{G_{2},\chi}$. If $b_1,b_2,b_3\in F$ satisfy $\ell(\transpose{(b_1,b_2,b_3)})\ne0$ and for all $x,y,z\in F$,
\begin{align*}
\varphi(n_{\alpha_2}^*(x)n_{\alpha_2+\alpha_3}^*(y)n_{\alpha_2+2\alpha_3}^*(z)v)=\psi(b_1x+b_2y+b_3z)\varphi(v),
\end{align*}
then $\varphi=0$.
\end{lemma}
\begin{proof}
By Claim~\ref{claim:center of the torus}, $C_{\cover{T}_3(F)}=\cover{T}_3(F)^2$, hence we can regard $\chi$ as the extension to $\cover{T}_3(F)^2$ of a character $\chi_0$ of $T_3(F)^2$. If
$t={\eta_1^{\vee}}(a_1^2){\eta_2^{\vee}}(a_2^2){\beta_1^{\vee}}(t_1)$,
\begin{align*}
\delta_{B_2(F)}^{1/2}\rconj{\mathbf{w_0}}\chi_0(t)=|a_1|^{-1}|a_2|^{-2}|t_1|^2\chi_0(\alpha_1^{\vee}(a_1^{-2}a_2^{-2}t_1^{2}))
=|a_1|\eta^{-1}(a_1^{-2}a_2^{-2}t_1^{2})=|a_1|\eta(\Upsilon(t)),
\end{align*}
for a suitable character $\eta$ of $F^*$. 

Let $M_1(F)^{\boxplus}=\GLF{1}{F}^{\star}\times G_{1}(F)^{\star}$
and $Q_1(F)^{\boxplus}=M_1(F)^{\boxplus}\ltimes U_1(F)$. According to \eqref{eq:block-compatibility on Levi subgroups}, and because
the covers of $\GLF{1}{F}$ and $G_1(F)$ are split, the cover of $M_1(F)^{\boxplus}$ is split. Therefore $\delta_{Q_1(F)}^{1/6}\eta\Upsilon$ can be pulled back to a genuine representation of $\cover{Q}_1(F)^{\boxplus}$ whence $\induced{\cover{Q_1}(F)^{\boxplus}}{\cover{G}_2(F)}{\delta_{Q_1(F)}^{1/6}\eta\Upsilon}$ is a genuine representation (this induction is not normalized). Because $[\cover{Q_1}(F):\cover{Q_1}(F)^{\boxplus}]$ is finite 
and for $t$ as above $\delta_{Q_1(F)}^{1/6}(t)=|a_1|$, we get
\begin{align*}
\induced{\cover{Q_1}(F)^{\boxplus}}{\cover{G}_2(F)}{\delta_{Q_1(F)}^{1/6}\eta\Upsilon}\subset \induced{\cover{B}_2(F)}{\cover{G}_2(F)}{\rho(\rconj{\mathbf{w_0}}\chi)}.
\end{align*}

This inclusion and Proposition~\ref{proposition:periodicity theorem} imply that $\induced{\cover{Q_1}(F)^{\boxplus}}{\cover{G}_2(F)}{\delta_{Q_1(F)}^{1/6}\eta\Upsilon}$ has a unique irreducible genuine subrepresentation, which must be $\Theta_{G_2,\chi}$. We proceed to show that $[F^*:{\Fsquares}]$ non-isomorphic irreducible
representations can be embedded in $\Theta_{G_2,\chi}$ over $\cover{G}_2(F)^{\star}$.
These will be certain twists of the Weil representation. According to Bernstein and Zelevinsky \cite{BZ1} (2.9),
$\Theta_{G_2,\chi}^{\star}$ equals the sum of these representations
($[\cover{G}_2(F):\cover{G}_2(F)^{\star}]=[F^*:{\Fsquares}]$). Thus it will suffice to show that the prescribed functional must vanish on each of them.

We use the exceptional isomorphism $G_2'\isomorphic Sp_2$. Let $Sp_2$ be the subgroup of $\GLF{4}{F}$ preserving the antisymmetric form on $F^4$ given by $(x,y)\mapsto\transpose{x}\left(\begin{smallmatrix}&J_{2}\\-J_{2}\end{smallmatrix}\right)y$ (the matrix $J_2$ was defined in Section~\ref{subsection:the groups}). Fix the Borel subgroup $B_{Sp_2}=T_{Sp_2}\ltimes N_{Sp_2}$ of upper triangular matrices in $Sp_2$. Let $\gamma_1,\gamma_2$ be the coordinate functions, i.e., if $y=diag(y_1,y_2,y_2^{-1},y_1^{-1})\in T_{Sp_2}(F)$, $\gamma_i(y)=y_i$. Denote by $\varrho_1,\varrho_2$ the simple roots of $Sp_2$. The mapping from $G_2'$ to $Sp_2$ is defined by $\alpha_2\mapsto\varrho_2$ and $\alpha_3\mapsto \varrho_1$.

Let $\omega_{\psi}$ be the Weil representation of the metaplectic double cover $\cover{Sp}_2(F)$ of $Sp_2(F)$, realized on the space $\mathcal{S}(F^2)$ of Schwartz-Bruhat functions on the row space $F^2$. It satisfies the following formulas (see \cite{P}):
\begin{align}\nonumber
&\omega_{\psi}((diag(a,a^*),\zeta))\phi(\xi)=\zeta\gamma_{\psi}(\det{a})\absdet{a}^{\half}\phi(\xi a),\\\label{eq:weil rep Siegel unipotent action}
&\omega_{\psi}((\left(\begin{smallmatrix}I_{2}&u\\&I_{2}\end{smallmatrix}\right),\zeta))\phi(\xi)=\zeta\psi(\half\xi J_{2}\transpose{u}\transpose{\xi})\phi(\xi).
\end{align}
Here $a\in\GLF{2}{F}$, $\zeta\in\mu_2$, $\phi\in\mathcal{S}(F^{2})$, $\xi\in F^2$ and $\gamma_{\psi}$ is the normalized Weil factor associated to $\psi$ (see Sections~\ref{subsection:the groups} and \ref{subsection:The Weil symbol}).
Note that $\gamma_{\psi}|_{{\Fsquares}}=1$. We have $\omega_{\psi}=\omega_{\psi}^{even}\oplus\omega_{\psi}^{odd}$, a decomposition
into even and odd functions, each space is irreducible. For $c\in F^*$, denote $\psi_c(x)=\psi(cx)$. If $c\ne d$ modulo
${\Fsquares}$, $\omega_{\psi_c}^{even}$ and $\omega_{\psi_d}^{even}$ are non-isomorphic.

The cover $\cover{G}_2'(F)$ is nontrivial and because $\cover{Sp}_2(F)$ is unique, the isomorphism $G_2'\isomorphic Sp_2$
extends to the cover groups and we can regard $\omega_{\psi}$ as a representation of $\cover{G}_2'(F)$.
We have $C_{\cover{G}_2(F)^{\star}}=C_{\cover{G}_2(F)}$,
$C_{\cover{G}_2(F)}\cap \cover{G}_2'(F)=p^{-1}(\alpha_3^{\vee}(\mp1))$ and
$\cover{G}_2(F)^{\star}=C_{\cover{G}_2(F)^{\star}}\cover{G}_2'(F)$ (see Claim~\ref{claim:center of cover}). Since $\alpha_3^{\vee}(-1)$ is mapped to the matrix $-I_4\in Sp_2(F)$ which acts by $-1$ on $\omega_{\psi}^{odd}$ and by $1$ on $\omega_{\psi}^{even}$, and $\Upsilon(\alpha_3^{\vee}(-1))=1$, we can extend $\omega_{\psi}^{even}$ to an irreducible genuine representation
$\eta\Upsilon\cdot\omega_{\psi}^{even}$ of $\cover{G}_2(F)^{\star}$.

For an even function $\phi\in\mathcal{S}(F^2)$ , define
\begin{align*}
E(\phi)(g)=\begin{cases}\eta\Upsilon\cdot\omega_{\psi}^{even}(g)\phi(0)&g\in\cover{G}_2(F)^{\star},\\
0&\text{otherwise.}\end{cases}
\end{align*}
We claim that $\phi\mapsto E(\phi)$ defines an embedding of $\eta\Upsilon\cdot\omega_{\psi}^{even}$ in $\Theta_{G_2,\chi}^{\star}$. 
It is enough to show $E(\phi)\in\induced{\cover{Q_1}(F)^{\boxplus}}{\cover{G}_2(F)}{\delta_{Q_1(F)}^{1/6}\eta\Upsilon}$. Indeed, because $\induced{\cover{Q_1}(F)^{\boxplus}}{\cover{G}_2(F)}{\delta_{Q_1(F)}^{1/6}\eta\Upsilon}$ is of finite length and $\lmodulo{\cover{G}_2(F)^{\star}}{\cover{G}_2(F)}$ is a finite group, any $\cover{G}_2(F)^{\star}$-submodule of
$\induced{\cover{Q_1}(F)^{\boxplus}}{\cover{G}_2(F)}{\delta_{Q_1(F)}^{1/6}\eta\Upsilon}$ is contained in $\Theta_{G_2,\chi}^{\star}$.


Clearly $E(\phi)$ is a smooth genuine function on $\cover{G}_2(F)$ with support in $\cover{G}_2(F)^{\star}$. 

For $m\in M_1(F)^{\boxplus}$ we can write $m=cm_1h'$ with
\begin{align*}
c=\alpha_1^{\vee}(a_1^{-2}t_1^2)\alpha_2^{\vee}(a_1^{-2}t_1^2)\alpha_3^{\vee}(a_1^{-1}t_1)\in C_{G_2(F)},\qquad m_1=\alpha_2^{\vee}(a_1^2)\alpha_3^{\vee}(a_1t_1^{-1}t_2),\qquad h'\in G_1'(F).
\end{align*}
We must show
\begin{align*}
E(\phi)(mg)=|a_1|\eta\Upsilon(a_1^{-2}t_1^2)E(\phi)(g),\qquad\forall g\in \cover{G}_2(F).
\end{align*}
Note that we can regard $m$ as an element of $\cover{M}_1(F)^{\boxplus}$ because the cover splits. Since $M_1(F)^{\boxplus}<G_2(F)^{\star}$, it is enough to assume $g\in\cover{G}_2(F)^{\star}$.

Since $h'$ is generated by the elements $n_{\alpha_3}(x)$ ($x\in F$) and $w_{\alpha_3}$, whose images in $Sp_2(F)$ belong to $\setof{diag(a,a^*)}{a\in SL_2(F)}$, $\omega_{\psi}^{even}((h',1))\phi(0)=\phi(0)$. Also the image of $m_1$ in $Sp_2(F)$ is
\begin{align*}
\varrho_1^{\vee}(a_1t_1^{-1}t_2)\varrho_2^{\vee}(a_1^2)=diag(a_1t_1^{-1}t_2,a_1t_1t_2^{-1},(a_1t_1t_2^{-1})^{-1},(a_1t_1^{-1}t_2)^{-1}).
\end{align*}
Thus $\eta\Upsilon\cdot\omega_{\psi}^{even}((m,1))\phi(0)=|a_1|\eta\Upsilon(a_1^{-2}t_1^2)\phi(0)$ as required.

Regarding invariancy under $\mathfrak{s}(U_1(F))$, let $V(F)<Sp_2(F)$ be the unipotent subgroup
\begin{align*}
\left\{v(x,y,z)=\left(\begin{array}{cccc}1&0&y&z\\&1&x&y\\&&1&0\\&&&1\end{array}\right):x,y,z\in F\right\}.
\end{align*}
The image of $n_{\alpha_2}(x)n_{\alpha_2+\alpha_3}(y)n_{\alpha_2+2\alpha_3}(z)\in U_1(F)$ in $Sp_2(F)$ is
$v(x,y,z)$, which acts trivially at $0$ by \eqref{eq:weil rep Siegel unipotent action}.

We conclude that $E$ is an embedding. Of course, this applies to any $\psi_c$.
It remains to show that $\eta\Upsilon\cdot\omega_{\psi}^{even}$ does not afford the functional appearing in the proposition. Indeed, the subgroup $U_1(F)$ is in bijection with $V(F)$ and exactly as in 
\cite{BFG} (Proposition~2.5), as long as $2b_1b_3+b_2^2\ne0$, any such functional vanishes on the space of $\omega_{\psi}^{even}$.
\end{proof}
\begin{remark}
The proof follows the arguments of \cite{BFG} (Proposition~2.5). There, the analog of $G_2(F)^{\star}$ is the kernel of the spinor norm map, $SO_5(F)'$. 
\end{remark}

The following lemma combines \cite{BFG} (Theorem~2.6) and \cite{BFG2} (Proposition~3).
\begin{lemma}\label{lemma:twisted Jacquet modules vanish on small representations}
Assume $\glnidx\geq1$. For any generic character $\psi_b$ of $V_{\UnipotentOrbit_1}(F)$, $j_{V_{\UnipotentOrbit_1},\psi_b}(\Theta_{G_{\glnidx},\chi})=0$.
\end{lemma}
\begin{remark}
If $\ell(b)\in\Fsquares$, it is possible to conjugate $\psi_b$ into the character $u\mapsto\psi(u_{\epsilon_2})$. Then the statement of Lemma~\ref{lemma:twisted Jacquet modules vanish on small representations} is similar to that of Theorem~2.6 of \cite{BFG}. The general case was stated in \cite{BFG2} (Proposition~3).
\end{remark}
\begin{proof}
Exactly as in \cite{BFG}, the proof is derived using a sequence of claims, all of which follow from
the results already proved in Sections~\ref{subsubsection:defs and basic properties}-\ref{subsubsection:vanishing and results}. The arguments of \cite{BFG} referring to unipotent subgroups can be repeated without a change. The difference is that here instead of using a tensor product representation to describe $j_{U_k}(\Theta_{G_{\glnidx},\chi})$, we use the induced representation of Proposition~\ref{proposition:Jacquet along standard nonminimal unipotent} (or Claim~\ref{claim:special case k=n} when $k=\glnidx$). We briefly present the arguments of \cite{BFG} and focus on the necessary changes.
The restriction of the cover $\cover{G}_{\glnidx}(F)$ to $GSpin_{2(\glnidx-1)}(F)$ gives a cover of the latter. The key observation is that $j_{U_1,\psi_b}(\Theta_{G_{\glnidx},\chi})$ is either a supercuspidal representation of $\cover{GSpin}_{2(\glnidx-1)}(F)$, or zero.


\begin{claim}\label{claim:twisted Jacquet modules vanish on small representations claim 1}
Assume $2\leq k\leq\glnidx$. Let $c_1,\ldots,c_{k-1}\in F$ and $d_1,d_2,d_3\in F$ be given with $\ell(\transpose{(d_1,d_2,d_3)})\ne0$. Define a character $\psi$ of $N_{\GL{k}}(F)U_k(F)$ by
\begin{align*}
\psi(u)=\begin{cases}\psi(\sum_{i=1}^{k-1}c_iu_{\alpha_{i+1}}+
d_1u_{\epsilon_{k+1}-\epsilon_{\glnidx+1}}+
d_2u_{\epsilon_{k+1}}+
d_3u_{\epsilon_{k+1}+\epsilon_{\glnidx+1}})&k<\glnidx,\\
\psi(\sum_{i=1}^{k-1}c_iu_{\alpha_{i+1}}+u_{\epsilon_{\glnidx+1}})&k=\glnidx.
\end{cases}
\end{align*}
Then $j_{N_{\GL{k}}U_k,\psi}(\Theta_{G_{\glnidx},\chi})=0$. Here $N_{\GL{k}}$ is embedded in the $\GL{k}$ part of $M_k$.
\end{claim}
\begin{proof}
We use induction on $\glnidx$, the base case is $\glnidx=2$. For any $1\leq l<k$, let $G_{\glnidx-l}$ be embedded in $M_{l}$ and $Q_{k-l}'=M_{k-l}'\ltimes U_{k-l}'<G_{\glnidx-l}$ be the standard parabolic subgroup with a Levi part $M_{k-l}'$ isomorphic to $\GL{k-l}\times G_{\glnidx-k}$. Let $N_{\GL{k-l}}'$ be the embedding of $N_{\GL{k-l}}$ in the $\GL{k-l}$ part of $M_{k-l}'$.
As a unipotent subgroup of $SO_{2\glnidx+1}$, $N_{\GL{k}}U_k$ takes the form \begin{align*}
\left(\begin{array}{ccccc}
z_1&u_1&u_2&u_3&u_4\\
&z_2&v_1&v_2&*\\
&&I_{2(\glnidx-k)+1}&*&*\\
&&&*&*\\
&&&&*\end{array}\right),\quad
\end{align*}
where $z_1\in N_{\GL{l}}$, $z_2\in N_{\GL{k-l}}'$, $u_1,\ldots,u_4$ are the coordinates of $U_l$ and $v_1,v_2$
are the coordinates of $U_{k-l}'$.

Assume $c_l=0$ for some $1\leq l<k$. Then
\begin{align*}
j_{N_{\GL{k}}U_k,\psi}=j_{N_{\GL{l}},\psi_1}j_{N_{\GL{k-l}}'U_{k-l}',\psi_2}j_{U_l},
\end{align*}
where $\psi_1$ and $\psi_2$ are obtained from $\psi$ by restriction.

By the induction hypothesis, or by Lemma~\ref{lemma:thm minimal case vanishing n=1} if $\glnidx-l=1$ (and then $k=\glnidx$),
\begin{align}\label{eq:vanishing claim 1 induction hypothesis}
j_{N_{\GL{k-l}}'U_{k-l}',\psi_2}(\Theta_{G_{\glnidx-l},\chi_2})=0
\end{align}
 for any exceptional character $\chi_2$ of $C_{\cover{T}_{\glnidx-l+1}(F)}$.

Write as in Proposition~\ref{proposition:Jacquet along standard nonminimal unipotent},
$j_{U_l}(\Theta_{G_{\glnidx},\chi})\subset I(\Theta_{\GL{l},\chi_1},\Theta_{G_{\glnidx-l},\chi_2})^{\star}$.
It is enough to prove
\begin{align*}
j_{N_{\GL{k-l}}'U_{k-l}',\psi_2}(I(\Theta_{\GL{l},\chi_1},\Theta_{G_{\glnidx-l},\chi_2})^{\star})=0.
\end{align*}
Moreover, according to Lemma~\ref{lemma:induced representation composition factors} it suffices to prove that $j_{N_{\GL{k-l}}'U_{k-l}',\psi_2}$ vanishes on
\begin{align}\label{eq:vanishing twisted Jacquet claim 1 first space to consider}
\cinduced{p^{-1}(\GLF{l}{F}^{\star}\times G_{\glnidx-l}(F))}{\cover{M}_l(F)}
{\Theta_{\GL{l},\chi_1}^{\star}\otimes \Theta_{G_{\glnidx-l},\chi_2}}.
\end{align}
Let $St_{\psi_2}(F)$ be the normalizer of $N_{\GL{k-l}}'(F)U_{k-l}'(F)$ and stabilizer of $\psi_2$ in $G_{\glnidx-l}(F)$.
The double coset space
\begin{align*}
\rmodulo{\lmodulo{(\GLF{l}{F}^{\star}G_{\glnidx-l}(F))}{M_l(F)}}{(\GLF{l}{F}St_{\psi_2}(F))}
\end{align*}
has one element. Using the Geometric Lemma of Bernstein and Zelevinsky \cite{BZ2} (Theorem~5.2), the application of $j_{N_{\GL{k-l}}'U_{k-l}',\psi_2}$ to \eqref{eq:vanishing twisted Jacquet claim 1 first space to consider} is seen to be equal to
\begin{align*}
\cinduced{p^{-1}(\GLF{l}{F}^{\star}\times St_{\psi_2}(F))}{p^{-1}(\GLF{l}{F}\times St_{\psi_2}(F))}
{\Theta_{\GL{l},\chi_1}^{\star}\otimes j_{N_{\GL{k-l}}'U_{k-l}',\psi_2}(\Theta_{G_{\glnidx-l},\chi_2})},
\end{align*}
which vanishes by \eqref{eq:vanishing claim 1 induction hypothesis}.

Thus we may assume $c_1=\ldots=c_{k-1}=1$. By virtue of Corollary~\ref{corollary:Jacquet along standard nonminimal unipotent} and Lemma~\ref{lemma:induced representation composition factors} it is enough to show
\begin{align*}
j_{N_{\GL{k}}U_k,\psi}(\induced{\cover{Q}_1(F)}{\cover{G}_{\glnidx}(F)}{\cinduced{p^{-1}(\GLF{1}{F}^{\star}\times G_{\glnidx-1}(F))}{\cover{M}_1(F)}
{\Theta_{\GL{1},\chi_3}^{\star}\otimes \Theta_{G_{\glnidx-1},\chi_4}}})=0.
\end{align*}

Let $\chi'$ be an exceptional character of $C_{\cover{T}_{\glnidx}(F)}$. If $\glnidx=2$, let $r=1$, otherwise take
$2\leq r\leq\glnidx-1$. Let $\psi'$ be a nontrivial character of $N_{\GL{r}}(F)U_{r}(F)$, where $N_{\GL{r}}U_{r}$ is a subgroup of $G_{\glnidx-1}$ embedded in the standard maximal parabolic subgroup with a Levi part $\GL{r}\times G_{\glnidx-1-r}$. If $\glnidx>2$, assume that $\psi'$ is defined
as in the statement of the lemma, with respect to $c'_1,\ldots,c'_{r-1},d'_1,d'_2,d'_3\in F$, where $\ell(\transpose{(d'_1,d'_2,d'_3)})\ne0$. According to the induction hypothesis and Lemma~\ref{lemma:thm minimal case vanishing n=1},
\begin{align*}
j_{N_{\GL{r}}U_{r},\psi'}(\Theta_{G_{\glnidx-1},\chi'})=0.
\end{align*}
As above \cite{BZ2} (Theorem~5.2) implies
\begin{align}\label{eq:vanishing twisted Jacquet claim 1 second induction}
j_{N_{\GL{r}}U_r,\psi'}(\cinduced{p^{-1}(\GLF{1}{F}^{\star}\times G_{\glnidx-1}(F))}{\cover{M}_1(F)}
{\Theta_{\GL{1},\chi_3}^{\star}\otimes \Theta_{G_{\glnidx-1},\chi_4}})=0.
\end{align}

Now the filtration argument of \cite{BFG} applies to our case as well. Specifically, the stabilizer in $M_{k}(F)$ of $\psi$
contains $N_{\GL{k}}(F)GSpin_{2(\glnidx-k)}(F)$, where $GSpin_{2(\glnidx-k)}(F)$ is the stabilizer of $\psi$ in $G_{\glnidx-k}(F)$ ($\GL{k}$ and $G_{\glnidx-k}$ are regarded as subgroups of $M_k$). The space of the representation induced from $\cover{Q}_1(F)$ to $\cover{G}_{\glnidx}(F)$ has a filtration according to the double cosets
$\rmodulo{\lmodulo{Q_1(F)}{G_{\glnidx}(F)}}{(N_{\GL{k}}(F)GSpin_{2(\glnidx-k)}(F))}$. Let $w$ be a representative of some double coset.
Then either $w$ conjugates one of the root subgroups of $N_{\GL{k}}U_k$ on which $\psi$ is nontrivial,
into $U_1$, in which case $j_{N_{\GL{k}}U_k,\psi}$ clearly vanishes, or $\psi|_{\rconj{w}(N_{\GL{k}}U_k)\cap (N_{\GL{k}}U_k)}=\psi'$ with $\psi'$ as above and then
the Jacquet module vanishes by \eqref{eq:vanishing twisted Jacquet claim 1 second induction}.
\end{proof}

For $\glnidx=1,2$ the lemma follows immediately from Lemmas~\ref{lemma:thm minimal case vanishing n=1} and \ref{lemma:minimal cases vanishing}. Henceforth assume $\glnidx>2$.
\begin{claim}\label{claim:twisted Jacquet modules is supercuspidal}
The representation $j_{U_1,\psi_b}(\Theta_{G_{\glnidx},\chi})$, if nonzero, is a supercuspidal representation of $\cover{GSpin}_{2(\glnidx-1)}(F)$.
\end{claim}
\begin{proof}
Assume $\ell(b)\in\Fsquares$, then we may assume $b=e_{\glnidx}(\in F^{2\glnidx-1})$, i.e., $\psi_b(u)=\psi(u_{\epsilon_2})$. Exactly as in \cite{BFG} (Lemma~2.8), using an ``exchange of roots" the claim can be reduced to showing the vanishing of
$\tau=j_{U_{l+1},\psi_2}(\Theta_{G_{\glnidx},\chi})$, where $0<l<\glnidx$ and $\psi_2$ is the character of $U_{l+1}(F)$
given by $\psi_2(u)=\psi(u_{\epsilon_{l+1}})$. Let $Y_{l+1}<\GL{l+1}$ be the mirabolic subgroup - the subgroup of matrices with the last row equal
to $(0,\ldots,0,1)$. Then $\tau$ is a genuine representation of $\cover{Y}_{l+1}(F)$ and we must show $\tau=0$. Suppose this does not hold. According to
Bernstein and Zelevinsky \cite{BZ1} (5.14), there are $c_1,\ldots,c_l\in F$ such that the character $\psi_1(u)=\psi(\sum_{i=1}^{l}c_iu_{\alpha_{i+1}})$ of
$N_{\GL{l+1}}(F)$ satisfies $j_{N_{\GL{l+1}},\psi_1}(\tau)\ne0$. But this means
$j_{N_{\GL{l+1}}U_{l+1},\psi}(\Theta_{G_{\glnidx},\chi})\ne0$,
where $\psi=\psi_1\cdot\psi_2$, contradicting Claim~\ref{claim:twisted Jacquet modules vanish on small representations claim 1}. If $\ell(b)\in F^*\setdifference\Fsquares$ the argument is similar, with a minor change to $\psi_2$.
\end{proof}

\begin{claim}\label{claim:supercuspidal in homspace}
For any genuine supercuspidal representation $\tau$ of $\cover{GSpin}_{2(\glnidx-1)}(F)$ and exceptional characters $\chi_1$ and $\chi_2$ of $C_{\cover{T}_{\GL{\glnidx}}}(F)$ and $\cover{T}_1(F)$,
\begin{align}\label{eq:homspace in supercuspidal in homspace}
Hom_{\cover{GSpin}_{2(\glnidx-1)}(F)}(j_{U_1,\psi_b}(\induced{\cover{Q}_{\glnidx}(F)}{\cover{G}_{\glnidx}(F)}{\Theta_{\GL{\glnidx},\chi_1}\otimes\Theta_{G_0,\chi_2}}),\tau)=0.
\end{align}
\end{claim}
\begin{proof}
Since in this particular case the tensor representation is defined, we are in the same situation as in \cite{BFG} (Lemma~2.9), except that
we have the quasi-split $GSpin_{2(\glnidx-1)}(F)$ instead of split $SO_{2(\glnidx-1)}(F)$. The space \eqref{eq:homspace in supercuspidal in homspace} can be analyzed using the Bruhat theory of
Silberger \cite{Silb} (Theorems~1.9.4 and 1.9.5). One considers the double coset space
$\rmodulo{\lmodulo{Q_{\glnidx}(F)}{G_{\glnidx}(F)}}{(U_1(F)GSpin_{2(\glnidx-1)}(F))}$, then uses the theory of derivations of Bernstein and Zelevinsky \cite{BZ1}. Since $\tau$ is supercuspidal, this space is of dimension at most $1$. Because $\glnidx\geq3$ and we are assuming that the residue characteristic of $F$ is odd,
$\Theta_{\GL{\glnidx},\chi_1}$ does not have a Whittaker model (\cite{KP} Section~I.3, see also \cite{BG} p.~145).
In this case, the dimension of \eqref{eq:homspace in supercuspidal in homspace} is zero. The arguments follow closely those of \cite{me5} (Section~4) for \quasisplit\ $SO_{2(\glnidx-1)}(F)$, Soudry \cite{Soudry} (Section~8) and Ginzburg, Rallis and Soudry \cite{GRS5} (Theorem~6.2(c)).
\end{proof}
As in \cite{BFG}, the proposition follows: set $\tau=j_{U_1,\psi_b}(\Theta_{G_{\glnidx},\chi})$, then
Claims~\ref{claim:special case k=n}, \ref{claim:twisted Jacquet modules is supercuspidal} and \ref{claim:supercuspidal in homspace} applied to
$\tau$ imply $Hom_{\cover{GSpin}_{2(\glnidx-1)}(F)}(\tau,\tau)=0$, contradiction.
\end{proof}

Now we are ready to prove the theorem.
\begin{proof}[Proof of Theorem~\ref{theorem:local vanishing result}]
Let $\Theta=\Theta_{G_{\glnidx},\chi}$, $\UnipotentOrbit=(r_1,\ldots,r_m)$ be such that
$\UnipotentOrbit\gtorncw\UnipotentOrbit_0$ and $\psi_b$ be a generic character of $V_{\UnipotentOrbit}(F)$.
The proof is a local analog of \cite{BFG} (Theorem~4.2 (i)). We briefly reproduce the arguments. Let $e$ (resp. $o$) be the number of even (resp. odd) numbers $r_i$ greater than $1$. By the definition of a unipotent class, $e$ is even, whence $m-e$ is odd.

We only show the case $o>0$. Put $r=\glnidx-o-(m+e-1)/2$. The unipotent subgroup $V_{\UnipotentOrbit}(F)$ corresponds to the following unipotent subgroup of $SO_{2\glnidx+1}(F)$ (see Section~\ref{subsection:properties of GSpin}),
\begin{align*}
\left\{\left(\begin{array}{ccccccc}\ddots&&&&&\alpha&\udots\\&I_o&0&x&y&z&\alpha'\\&&I_e&0&v&y'\\&&&I_{m-e}&0&x'\\&&&&I_e&0\\&&&&&I_o\\&&&&&&\ddots\end{array}\right)\right\}.
\end{align*}
Here $\alpha$ is an $r\times o$ matrix and if $h$ is an $i\times j$ matrix, $h'=-J_j\transpose{h}J_i$.
The middle $(2o+e+m)\times(2o+e+m)$ block, regarded as a unipotent subgroup of $G_{\glnidx-r}<M_r$, corresponds to the unipotent class $(3^o2^e1^{m-o-e})$ and we denote it by $V_{(3^o2^e1^{m-o-e})}$. We have
\begin{align*}
C(G_{\glnidx-r}(F),h_{(3^o2^e1^{m-o-e})}(F))=\GLF{o}{F}\times\GLF{e}{F}\times G_{(m-e-1)/2}(F).
\end{align*}
According to the description of Carter \cite{Cr} (p.~398), if $\varepsilon$ belongs to the open orbit (see the discussion in the beginning of this section), $St_{\varepsilon}(\overline{F})^0$ is a reductive group of Lie type
\begin{align*}
\begin{cases}
B_{(o-1)/2}\times C_{e/2}\times D_{(m-o-e)/2}&\text{$o$ is odd,}\\
D_{o/2}\times C_{e/2}\times D_{(m-o-e-1)/2}&\text{$o$ is even.}
\end{cases}
\end{align*}

Since the coordinates of $\alpha$ and $z$ vanish in $V_{\UnipotentOrbit}^{\matha}$, $\psi_b$ is trivial on the coordinates of $\alpha'$ and $z$. According to the definition, the restriction of $\psi_b$ to $V_{(3^o2^e1^{m-o-e})}(F)$ is a generic character. The restriction of $\psi_b$ to the coordinates of $x$ (resp. $v$) is characterized by an $o\times(m-e)$ (resp. $e\times e$) matrix $\mathbf{x}$ (resp. $\mathbf{v}$). We may assume that, if $(\mathbf{x},\mathbf{v})\in V_{\UnipotentOrbit}^{\matha}(F)$ is the point defined by $\mathbf{x}$ and $\mathbf{v}$, $St_{(\mathbf{x},\mathbf{v})}(\overline{F})^0$ is already of the prescribed Lie type. It follows that $\psi_b$ is trivial on $y$ and using a conjugation by a suitable $g\in\GLF{k}{F}$ (a product of a unipotent matrix and a Weyl element), we may also assume $\ell(\mathbf{x}_1)\ne0$.

%

Using another conjugation, we can move the $(r+1)$-th row (which contains $\mathbf{x}_1$) to the first row. Again by the transitivity of the Jacquet functor, it is enough to prove $j_{U,\psi_b'}(\Theta)=0$, for the subgroup
$U(F)<U_1(F)$ of matrices $u$ whose first row is
\begin{align*}
\left(\begin{array}{ccccccc}1&0_{r+o-1+e}&x_1&y_1&z_{1,1}&\alpha'_1&z_{1,2}\end{array}\right),
\end{align*}
where for a matrix $h$, $h_i$ is the $i$-th row of $h$; $(h_{i,1},h_{i,2})=h_i$; $h_{i,2}\in F$; and $\psi_b'(u)=\psi(\mathbf{x}_1(\transpose{x_1}))$.

If $j_{U,\psi_b'}(\Theta)\ne0$, 
there exists some character $\psi_c$ of $U_1(F)$ which agrees with $\psi_b'$ on $U(F)$, such that $j_{U_1,\psi_c}(\Theta)\ne0$. However, 
any such character takes the form $\psi_c(u)=\psi(r(u)c)$ where $\transpose{c}$ is the row $(c_1,\ldots,c_{r+o-1+e},\mathbf{x},0_{r+o-1+e})\in F^{2\glnidx-1}$. Since
$\ell(c)\ne0$, this contradicts Lemma~\ref{lemma:twisted Jacquet modules vanish on small representations}.
\end{proof}
We also have the following vanishing result.
\begin{proposition}\label{proposition:vanishing along GLn of the Whittaker Jacquet for small representations}
For all $3\leq k\leq\glnidx$, $j_{N_{\GL{k}}U_k,\psi}(\Theta_{G_{\glnidx},\chi})=0$, where
$\psi(u)=\psi(\sum_{i=1}^{k-1}u_{\alpha_{i+1}})$.
\end{proposition}
\begin{proof}
Since $j_{N_{\GL{k}}U_k,\psi}=j_{N_{\GL{k}},\psi}j_{U_k}$ and using
Proposition~\ref{proposition:Jacquet along standard nonminimal unipotent} and
Lemma~\ref{lemma:induced representation composition factors}, we see that it is enough to prove
\begin{align}\label{eq:vanishing twisted Jacquet on GLn space to consider}
j_{N_{\GL{k}},\psi}(\cinduced{p^{-1}(\GLF{k}{F}\times G_{\glnidx-k}(F)^{\star})}{\cover{M}_k(F)}
{\Theta_{\GL{k},\chi_1}\otimes\Theta_{G_{\glnidx-k},\chi_2}^{\star}})=0.
\end{align}
As above we use \cite{BZ2} (Theorem~5.2).
Let $St_{\psi}(F)$ be the normalizer of $N_{\GL{k}}(F)$ and stabilizer of $\psi$ in $\GLF{k}{F}$. The space
\begin{align*}
\rmodulo{\lmodulo{(\GLF{k}{F}G_{\glnidx-k}(F)^{\star})}{M_k(F)}}{(St_{\psi}(F)G_{\glnidx-k}(F))}
\end{align*}
has one element. Since $\psi|_{N_{\GL{k}}}$ is the standard Whittaker character and
$\Theta_{\GL{k},\chi_1}$ is not generic (\cite{KP} Section~I.3), $j_{N_{\GL{k}},\psi}(\Theta_{\GL{k},\chi_1})=0$. Hence \eqref{eq:vanishing twisted Jacquet on GLn space to consider} follows.
\end{proof}

\subsubsection{Explicit construction of exceptional characters}\label{subsubsection:Explicit construction of an exceptional representation}
Let $\eta$ be a character of $F^*$. Start with defining a character $\chi_0$ of $T_{\glnidx+1}(F)^2$, depending on $\eta$. For $t=\prod_{i=1}^{\glnidx}\eta_i^{\vee}(a_i^2)\beta_1^{\vee}(t_1)$, let
\begin{align*}
\chi_0(t)=\prod_{i=1}^{\glnidx}|a_i|^{\glnidx-i+1}\eta(\Upsilon(t)).
\end{align*}
Using \eqref{eq:image of convenient coordinates of the torus in GSpin} 
we see that $\chi_0(\alpha^{\vee}(x^{\mathfrak{l}(\alpha)}))=|x|$ for
all $\alpha\in\Delta_{G_{\glnidx}}$ and $x\in F^*$. Extend $\chi_0$ to a genuine character $\chi$ of $\cover{T}_{\glnidx+1}(F)^2$ using $\mathfrak{s}$: $\chi(\zeta\mathfrak{s}(t))=\zeta\chi_0(t)$ ($\zeta\in\mu_2$). If $\glnidx$ is even,
$\chi$ is exceptional.

Let $\psi$ be a nontrivial additive character of $F$. If $z=\prod_{i=1}^{\glnidx}\eta_i^{\vee}(d)$ and $z'=\prod_{i=1}^{\glnidx}\eta_i^{\vee}(d')$, where $d,d'\in(F^*)^{2/\gcd(2,\glnidx+1)}$, using
\eqref{eq:section of BLS on torus}, \eqref{eq:torus alpha and alpha} and \eqref{eq:torus alpha and alpha'} we get
$\sigma(z,z')=c(d,d')^{\lceil\glnidx/2\rceil}$. Therefore, $\zeta\mathfrak{s}(z)\mapsto\zeta\gamma_{\psi}^{\lceil\glnidx/2\rceil}(d)$ is a genuine character of $C_{\CGLF{\glnidx}{F}}$
(if $\glnidx$ is even, $\gamma_{\psi}(d)=1$).

Since $[t,z]_{\sigma}=1$, we can define the following exceptional character of $C_{\cover{T}_{\glnidx+1}(F)}$:
\begin{align}\label{eq:explicit exceptional character}
\chi(\zeta\mathfrak{s}(t)\mathfrak{s}(z))=\zeta\chi_0(t)|d|^{\glnidx(\glnidx+1)/4}\eta(d)^{\glnidx}\gamma_{\psi}^{\lceil\glnidx/2\rceil}(d).
\end{align}
In fact, every exceptional character can be written in this form, for suitable $\psi$ and $\eta$.

We mention that in the context of exceptional representations of $\GL{\glnidx}$, the character $\eta$ is the ``determinantal character" of Bump and Ginzburg (\cite{BG} p.~143).

Let $\cover{T}_{\glnidx+1}(F)^{\mathm}$ be the maximal abelian subgroup given in Claim~\ref{claim:constructing a maximal abelian subgroup of the torus in general}. The following formula defines an extension $\chi'$ of $\chi$ to $\cover{T}_{\glnidx+1}(F)^{\mathm}$. For $t=\prod_{i=1}^{\glnidx}\eta_i^{\vee}(a_i)\beta_1^{\vee}(t_1)\in T_{\glnidx+1}(F)^{\mathm}$, let
\begin{align*}
\chi'(\zeta\mathfrak{s}(t))=\zeta\prod_{i=1}^{\glnidx}|a_i|^{(\glnidx-i+1)/2}\eta(\Upsilon(t))\prod_{i=0}^{\lceil\glnidx/2\rceil-1}\gamma_{\psi}(a_{\glnidx-2i}).
\end{align*}

Now consider the unramified case. Assume all data are unramified. This means that $|2|=1$, $q>3$ in $F$; $\psi$ is unramified ($\psi|_{\RingOfIntegers}=1$, $\psi|_{\mathcal{P}^{-1}}\ne1$); and $\eta$ is unramified ($\eta|_{\RingOfIntegers^*}=1$).
Then $\gamma_{\psi}|_{\RingOfIntegers^*}=1$ (see e.g. \cite{Dani} Lemma~3.4). Therefore $\chi$ is trivial on
$C_{\cover{T}_{\glnidx+1}(F)}\cap K^*$. In this case $\chi$ is an unramified character (see Section~\ref{subsubsection:Unramified representations}).

\section{Global theory}\label{section:global theory}

\subsection{The $r$-fold cover of $G_{\glnidx}(\Adele)$}\label{subsection:the double cover of G_n Adele}
Let $F$ be a number field containing all $r$ $r$-th roots of unity, with a ring of ad\`{e}les $\Adele$. The global $r$-fold cover $\cover{G}_{\glnidx}(\Adele)$ of
$G_{\glnidx}(\Adele)$ is defined similarly to the definition of $\CGLF{\glnidx}{\Adele}$ of \cite{KP} (Section~0.2, see also \cite{FK,Tk2,Tk}). At each place $\nu$ we have the exact sequence of Section~\ref{subsubsection:Definition of the cover},
\begin{align*}
1\rightarrow{\mu_r(F_{\nu})}\xrightarrow{\iota_{\nu}} \cover{G}_{\glnidx}(F_{\nu})\xrightarrow{p_{\nu}} G_{\glnidx}(F_{\nu})\rightarrow 1.
\end{align*}
Denote by $c$ the inverse of the global Hilbert symbol, $c$ is the
product of the local functions $c_{\nu}$.
At each $\nu$, let $K_{\nu}<G_{\glnidx}(F_{\nu})$ be either $G_{\glnidx}(\RingOfIntegers_{\nu})$ if $F_{\nu}$ is $p$-adic,
or a maximal compact subgroup in the \archimedean\ case, and put $K=\prod_{\nu}K_{\nu}$.
For simplicity, fix an identification of $\mu_r(F_{\nu})$ with $\mu_r=\mu_r(F)$ (which is usually identified with a subgroup of $\C^*$). Let $S_{\infty}$ be the set of \archimedean\ places of $F$. Fix a finite set $S_0$ containing $S_{\infty}$ and the places $\nu$ such that $|r|_{\nu}<1$ or $q_{\nu}\leq3$.
For any finite set $S\supset S_0$ denote
\begin{align*}
G_{\glnidx}(\Adele)_S=\prod_{\nu\in S}G_{\glnidx}(F_{\nu})\prod_{\nu\notin S}K_{\nu},\qquad
G_{\glnidx}(\Adele)_S^*=\prod_{\nu\in S}\cover{G}_{\glnidx}(F_{\nu})\prod_{\nu\notin S}K_{\nu}^{*}.
\end{align*}
For $\nu\in S$, $\cover{G}_{\glnidx}(F_{\nu})<G_{\glnidx}(\Adele)_S^*$ (in particular $\iota_{\nu}(\mu_r)<G_{\glnidx}(\Adele)_S^*$). Note that $K_{\nu}^{*}$ was defined given the assumption on $S$ (see Section~\ref{subsubsection:A splitting of the hyperspecial subgroup}). Let $\widehat{\mu_r}$ be the subgroup of $G_{\glnidx}(\Adele)_S^*$ generated by the elements $\iota_{\nu}(\zeta)\iota_{\nu'}(\zeta)^{-1}$ where $\zeta\in\mu_r$ and $\nu,\nu'\in S$. Define $\widehat{G}_{\glnidx}(\Adele)_S=\lmodulo{\widehat{\mu_r}}{G_{\glnidx}(\Adele)_S^*}$. Then
\begin{align*}
&G_{\glnidx}(\Adele)=\lim_{S\longrightarrow}G_{\glnidx}(\Adele)_S,\qquad
\cover{G}_{\glnidx}(\Adele)=\lim_{S\longrightarrow}\widehat{G}_{\glnidx}(\Adele)_S.
\end{align*}
We have an exact sequence
\begin{align*}
1\rightarrow{\mu_r}\xrightarrow{\iota}\cover{G}_{\glnidx}(\Adele)\xrightarrow{p} G_{\glnidx}(\Adele)\rightarrow 1.
\end{align*}
Regarding the topologies on these groups refer to \cite{KP} (Section~0.2).

Let $\prod'_{\nu}G_{\glnidx}(F_{\nu})$ be the restricted direct product with respect
to the subgroups $\{K_{\nu}\}_{\nu}$. Then $G_{\glnidx}(\Adele)=\prod'_{\nu}G_{\glnidx}(F_{\nu})$.
Similarly, define $\prod'_{\nu}\cover{G}_{\glnidx}(F_{\nu})$ with respect to $\{K^*_{\nu}\}_{\nu}$. 
Set $\mu_r^{\times}=\setof{(i_{\nu}(\zeta_{\nu}))_{\nu}\in \prod'_{\nu}\cover{G}_{\glnidx}(F_{\nu})}{\prod_{\nu}{\zeta_{\nu}=1}}$. 
Then $\cover{G}_{\glnidx}(\Adele)=\lmodulo{\mu_r^{\times}}{\prod'_{\nu}\cover{G}_{\glnidx}(F_{\nu})}$.

In Section~\ref{subsubsection:A splitting of the hyperspecial subgroup} we defined $\kappa_{\nu}$ for $\nu\notin S_0$, extend $\kappa_{\nu}$ to a smooth section of $G_{\glnidx}(F_{\nu})$. Define sections $\kappa_{\nu}$ also for $\nu\in S_0$ arbitrarily. Set $\kappa=\prod_{\nu}\kappa_{\nu}$. This is a section of $G_{\glnidx}(\Adele)$, well defined because $\kappa_{\nu}(K_{\nu})=K^*_{\nu}$ for $\nu\notin S_0$. The corresponding $2$-cocycle is defined by
$\gamma(g,g')=\kappa(g)\kappa(g')\kappa(gg')^{-1}$, well defined since for $\nu\notin S_0$, $\kappa_{\nu}$ is a splitting of $K_{\nu}$.

Consider $\mathfrak{s}=\prod_{\nu}\mathfrak{s}_{\nu}$, where $\mathfrak{s}_{\nu}$ is the section given in Section~\ref{subsubsection:Definition of the cover}. This is not a section of $G_{\glnidx}(\Adele)$, but can be used with the following subgroups.
\begin{claim}\label{claim:good ol section splits G(F) and N(A)}
The section $\mathfrak{s}$ is a splitting of $G_{\glnidx}(F)$ and $N_{\glnidx}(\Adele)$. It is well defined on $T_{\glnidx+1}(\Adele)$.
\end{claim}
\begin{proof}
First observe that $\mathfrak{s}$ is indeed well defined on these subgroups. Indeed, let $g\in G_{\glnidx}(F)$ and write $g=utwu'$ as in \eqref{eq:props of BLS section separate N T and W}. Then for any place $\nu$, $\mathfrak{s}_{\nu}(g)=\mathfrak{s}_{\nu}(u)\mathfrak{s}_{\nu}(t)\mathfrak{s}_{\nu}(w)\mathfrak{s}_{\nu}(u')$. For almost all $\nu$,
the elements $u,t,w$ and $u'$ lie in $K_{\nu}$, then Claim~\ref{claim:relations section s and section kappa} implies
$\mathfrak{s}_{\nu}(g)\in K^*_{\nu}$. The same claim implies that if $u\in N_{\glnidx}(\Adele)$ and $t\in T_{\glnidx+1}(\Adele)$,
$\mathfrak{s}_{\nu}(u_{\nu})=\kappa_{\nu}(u_{\nu})\in K^*_{\nu}$ and
$\mathfrak{s}_{\nu}(t_{\nu})=\kappa_{\nu}(t_{\nu})\in K^*_{\nu}$ for almost all $\nu$. 

Since for all $\nu$, $\mathfrak{s}_{\nu}$ is a splitting of $N_{\glnidx}(F_{\nu})$, $\mathfrak{s}$ is a splitting of $N_{\glnidx}(\Adele)$. To show that $\mathfrak{s}$ is a splitting of $G_{\glnidx}(F)$ one uses \eqref{eq:props of BLS section separate N T and W}, the relations in \cite{BLS} (Section~1) defining the multiplication in $G_{\glnidx}(F_{\nu})$ and
$\cover{G}_{\glnidx}(F_{\nu})$, and the fact that $c$ is trivial on $F^*\times F^*$.
\end{proof}
\begin{remark}\label{remark:defining sigma globally whenever possible}
Let $H$ be a subgroup of $G_{\glnidx}(\Adele)$ such that for any $h\in H$, for almost all $\nu$, $\mathfrak{s}(h_{\nu})=\kappa(h_{\nu})$. Then if $h,h'\in H$,
we have $\mathfrak{s}_{\nu}(h_{\nu})\mathfrak{s}_{\nu}(h'_{\nu})=\mathfrak{s}_{\nu}(h_{\nu}h'_{\nu})$ almost everywhere. Hence we can define
the $2$-cocycle $\sigma(h,h')=\mathfrak{s}(h)\mathfrak{s}(h')\mathfrak{s}(hh')^{-1}$ on $H$ and then
$\sigma=\prod_{\nu}\sigma_{\nu}$ where $\sigma_{\nu}$ is the block-compatible cocycle of
Section~\ref{subsubsection:Definition of the cover}. According to Claim~\ref{claim:good ol section splits G(F) and N(A)},
$\sigma$ is well defined on $T_{\glnidx}(\Adele)$, and trivial on $G_{\glnidx}(F)$ and $N_{\glnidx}(\Adele)$.
\end{remark}

Henceforth $r=2$.

\subsection{Representations, intertwining operators and Eisenstein series}\label{subsubsection:global reps}
We study irreducible genuine representations of $\cover{T}_{\glnidx+1}(\Adele)$ formed by extending a
genuine character of $C_{\cover{T}_{\glnidx+1}(\Adele)}=
\lmodulo{\mu_2^{\times}}{\prod'_{\nu}C_{\cover{T}_{\glnidx+1}(F_{\nu})}}$ to a maximal abelian subgroup, then
inducing to $\cover{T}_{\glnidx+1}(\Adele)$. The resulting representation is independent of the choices of abelian subgroup and extension (\cite{KP} Section~0.3).

Let $\chi$ be a genuine character of $C_{\cover{T}_{\glnidx+1}(\Adele)}$, which is trivial on $C_{\cover{T}_{\glnidx+1}(\Adele)}\cap\mathfrak{s}(T_{\glnidx+1}(F))$. Then $\chi=\otimes'_{\nu}\chi_{\nu}$, where for almost all $\nu$, $\chi_{\nu}$ is unramified, because it is trivial on
$C_{\cover{T}_{\glnidx+1}(F_{\nu})}\cap K_{\nu}^*$ (see Section~\ref{subsubsection:Unramified representations}).

Following Kazhdan and Patterson \cite{KP} (Section~II.1) we will use two maximal subgroups, the resulting representation will be the same. One choice is $X=\lmodulo{\mu_2^{\times}}{\prod'_{\nu}X_{\nu}}$, where $C_{\cover{T}_{\glnidx+1}(F_{\nu})}<X_{\nu}<\cover{T}_{\glnidx+1}(F_{\nu})$
is a local maximal abelian subgroup such that for almost all $\nu$, $X_{\nu}=C(\cover{T}_{\glnidx+1}(F_{\nu}),\cover{T}_{\glnidx+1}(F_{\nu})\cap K^*_{\nu})$
(see 
Claim~\ref{claim:constructing maximal abelian subgroup unramified case}). The subgroup $X$ is suitable for relating a global representation to its local pieces. Globally it is simpler to work with a subgroup containing $\mathfrak{s}(T_{\glnidx+1}(F))$. As in \cite{KP} (Lemma~II.1.1), we introduce the following alternative.
\begin{claim}\label{claim:convenient global subgroup}
The subgroup $C(\cover{T}_{\glnidx+1}(\Adele),\mathfrak{s}(T_{\glnidx+1}(F)))$ is a maximal abelian subgroup of
$\cover{T}_{\glnidx+1}(\Adele)$, which contains $\mathfrak{s}(T_{\glnidx+1}(F))$. We have
$C(\cover{T}_{\glnidx+1}(\Adele),\mathfrak{s}(T_{\glnidx+1}(F)))=\mathfrak{s}(T_{\glnidx+1}(F))C_{\cover{T}_{\glnidx+1}(\Adele)}$.
\end{claim}
\begin{proof}
The containment from right to left is clear, because by Claim~\ref{claim:good ol section splits G(F) and N(A)} the section $\mathfrak{s}$
splits $G_{\glnidx}(F)$, hence the elements of $\mathfrak{s}(T_{\glnidx+1}(F))$ commute. Now we proceed as in
the proof of Claim~\ref{claim:constructing maximal abelian subgroup unramified case} and note that by
Weil \cite{We2} (Section~XIII.5 Propositions~7 and 8), $c(x,y)=1$ for all $y\in \mathcal{\Adele}^*$ if and only if $x\in(\Adele^*)^2$,
and $c(x,y)=1$ for all $y\in F^*$ if and only if $x\in F^*(\Adele^*)^2$
(compare to Section~\ref{subsection:The Hilbert symbol}). Observe that the proof of Claim~\ref{claim:constructing maximal abelian subgroup unramified case} uses
the local cocycle $\sigma_{\nu}$ defined using $\mathfrak{s}_{\nu}$. As explained in Remark~\ref{remark:defining sigma globally whenever possible},
the argument extends to the global setting.
\end{proof}

For $\underline{s}=(s_1,\ldots,s_{\glnidx+1})\in\C^{\glnidx+1}$ define a nongenuine character ${\mathe}_{\underline{s}}$ of $\cover{T}_{\glnidx+1}(\Adele)$
by requiring ${\mathe}_{\underline{s}}(\alpha^{\vee}(x))=|x|^{(\alpha,\underline{s})/\mathfrak{l}(\alpha)}$ for all $\alpha\in\Delta_{G'_{\glnidx+1}}$ and $x\in\Adele^*$. E.g., ${\mathe}_{\underline{s}}(\alpha_1^{\vee}(x))=|x|^{(s_1-s_2)/2}$, ${\mathe}_{\underline{s}}(\alpha_{\glnidx+1}^{\vee}(x))=|x|^{s_{\glnidx+1}}$. The character ${\mathe}_{\underline{s}}$ is trivial on $\mathfrak{s}(T_{\glnidx+1}(F))$.

If $d\in\C$, denote $\underline{d}=(d,\ldots,d)$ (the length of $\underline{d}$ will
be clear from the context).

Let $\chi$ be a genuine character of $C_{\cover{T}_{\glnidx+1}(\Adele)}$, which is trivial on $C_{\cover{T}_{\glnidx+1}(\Adele)}\cap\mathfrak{s}(T_{\glnidx+1}(F))$. Denote $\chi_{\underline{s}}={\mathe}_{\underline{s}}\cdot\chi$.
Assume we are given an abelian subgroup $C_{\cover{T}_{\glnidx+1}(\Adele)}<H<\cover{T}_{\glnidx+1}(\Adele)$ and an extension $\chi'$ of $\chi$ to $H$.
We will always assume $(\chi')_{\underline{s}}=(\chi_{\underline{s}})'$. Therefore we continue to denote the extension by $\chi$ (instead of $\chi'$) and
the notation $\chi_{\underline{s}}$ is not ambiguous.

Let $V_H(\chi_{\underline{s}})$ be the space of the representation $\Induced{HN_{\glnidx}(\Adele)}{\cover{G}_{\glnidx}(\Adele)}{\chi_{\underline{s}}}$. Specifically, this is the space of functions $f:\cover{G}_{\glnidx}(\Adele)\rightarrow\C$ which
are $\cover{K}$-smooth on the right and $f(h\mathfrak{s}(u)g)=\delta_{B_{\glnidx}(\Adele)}^{1/2}(h)\chi_{\underline{s}}(h)f(g)$ for all $h\in H$, $u\in N_{\glnidx}(\Adele)$ and
$g\in\cover{G}_{\glnidx}(\Adele)$. Set $V_H(\chi)=V_H(\chi_{\underline{0}})$. A standard section $f_{\underline{s}}$ is a function such that $f_{\underline{s}}|_{\cover{K}}$ is independent of $\underline{s}$ and
for all $\underline{s}$, $f_{\underline{s}}\in V_H(\chi_{\underline{s}})$. For any $f\in V_H(\chi)$ there is a standard section
$f_{\underline{s}}$ with $f_{\underline{0}}=f$.


We extend $\chi$ to $C(\cover{T}_{\glnidx+1}(\Adele),\mathfrak{s}(T_{\glnidx+1}(F)))$ and to $X$. The extension to
$C(\cover{T}_{\glnidx+1}(\Adele),\mathfrak{s}(T_{\glnidx+1}(F)))$ is obtained
by acting trivially on $\mathfrak{s}(T_{\glnidx+1}(F))$. The extension to $X$ is arbitrary, as long as both extensions agree on $C(\cover{T}_{\glnidx+1}(\Adele),\mathfrak{s}(T_{\glnidx+1}(F)))\cap X$ (a suitable extension to $X$ with this property always exists,
see \cite{KP} p.~108).

For brevity, put $\Xi=C(\cover{T}_{\glnidx+1}(\Adele),\mathfrak{s}(T_{\glnidx+1}(F)))$.

The space $V_X(\chi_{\underline{s}})$ decomposes as the restricted tensor product $\otimes'_{\nu}V(\chi_{\underline{s},\nu})$ with respect
to $\{V(\chi_{\underline{s},\nu})^{K^*_{\nu}}\}_{\nu\notin S_0}$, where $V(\chi_{\underline{s},\nu})$ is the space of
$\induced{\cover{B}_{\glnidx}(F_{\nu})}{\cover{G}_{\glnidx}(F_{\nu})}{\rho_{\nu}(\chi_{\underline{s},\nu})}$, defined in Section~\ref{subsubsection:representations of the torus}.

The spaces $V_X(\chi)$ and $V_{\Xi}(\chi)$ are isomorphic (see \cite{KP} p.~109). Specifically, for
$f\in V_X(\chi)$ define $f'\in V_{\Xi}(\chi)$ via
\begin{align}\label{eq:isomorphism X and Xi}
f'(g)=\sum_{\delta'\in\ \lmodulo{(C_{\cover{T}_{\glnidx+1}(\Adele)}\cap\mathfrak{s}(T_{\glnidx+1}(F))) }{T_{\glnidx+1}(F)}}f(\mathfrak{s}(\delta') g),\qquad g\in\cover{G}_{\glnidx}(\Adele).
\end{align}

For $\mathbf{w}\in W_{\glnidx}$, let $M'(w,\chi_{\underline{s}}):V_{\Xi}(\chi_{\underline{s}})\rightarrow
V_{\Xi}(\rconj{\mathbf{w}}\chi_{\underline{s}})$ be the standard intertwining operator defined by (the meromorphic continuation of)
\begin{align}\label{eq:global intertwining def}
M'(w,\chi_{\underline{s}})f'(g)=\int_{\lmodulo{N_{\glnidx}^w(\Adele)}{N_{\glnidx}(\Adele)}}f'_{\underline{s}}
(\mathfrak{s}(w)^{-1}\mathfrak{s}(u)g)du.
\end{align}
Note that $\mathfrak{W}_{\glnidx}\subset G_{\glnidx}(F)$. Using the isomorphism $V_{\Xi}\isomorphic V_{X}$ we can define the corresponding intertwining operator $M(w,\chi_{\underline{s}}):V_{X}(\chi_{\underline{s}})\rightarrow
V_{X}(\rconj{\mathbf{w}}\chi_{\underline{s}})$. The poles and zeros (with multiplicities) of these operators coincide.


Now we are ready to define the Eisenstein series. For $f\in V_X(\chi)$ let
\begin{align*}
E_{B_{\glnidx}}(g;f,\underline{s})=\sum_{\delta\in\lmodulo{B_{\glnidx}(F)}{G_{\glnidx}(F)}}f'_{\underline{s}}(\mathfrak{s}(\delta)g).
\end{align*}
The sum is absolutely convergent in a suitable cone and has a meromorphic continuation. Analogously to
\cite{KP} (Proposition~II.1.2), or using the general formulation of M{\oe}glin and Waldspurger \cite{MW2}
(Section~II.1.7), if
\begin{align*}
E_{B_{\glnidx}}^{N_{\glnidx}}(g;f,\underline{s})=\int_{\lmodulo{N_{\glnidx}(F)}{N_{\glnidx}(\Adele)}}
E_{B_{\glnidx}}(\mathfrak{s}(u)g;f,\underline{s})du
\end{align*}
is the constant term of $E_{B_{\glnidx}}(g;f,\underline{s})$ along $N_{\glnidx}$,
\begin{align*}
E_{B_{\glnidx}}^{N_{\glnidx}}(g;f,\underline{s})=\sum_{\mathbf{w}\in W_{\glnidx}}M'(w,\chi_{\underline{s}})f'_{\underline{s}}(g).
\end{align*}

\subsection{Induction from $M_{\glnidx}$}\label{subsubsection:global induction transitivity}
Throughout this section, the groups $\GL{\glnidx}$ and $G_0$ are regarded as subgroups of $M_{\glnidx}$. Since the local subgroups $\CGLF{\glnidx}{F_{\nu}}$ and $\cover{G}_0(F_{\nu})$ commute,
$\CGLF{\glnidx}{\Adele}$ and $\cover{G}_0(\Adele)$ commute. Therefore, genuine automorphic representations of $\cover{M}_{\glnidx}(\Adele)$ can be described using the usual tensor product.

Let $\chi$ be a genuine character of $C_{\cover{T}_{\glnidx+1}(\Adele)}$ which is trivial on
$C_{\cover{T}_{\glnidx+1}(\Adele)}\cap\mathfrak{s}(T_{\glnidx+1}(F))$. According to the results of Section~\ref{subsection:Subgroups of the torus},
\begin{align*}
C_{\cover{T}_{\glnidx+1}(\Adele)}=\lmodulo{\setof{(\zeta,\zeta)}{\zeta\in\mu_2}}{(C_{\cover{T}_{\GL{\glnidx}}(\Adele)}\times\cover{G}_0(\Adele))}.
\end{align*}
Therefore we can write uniquely $\chi=\chi^{(1)}\otimes\chi^{(2)}$, where $\chi^{(1)}$ and $\chi^{(2)}$ are genuine characters of $C_{\cover{T}_{\GL{\glnidx}}(\Adele)}$ and
$\cover{G}_0(\Adele)$. The character $\chi^{(1)}$ is trivial on $C_{\cover{T}_{\GL{\glnidx}}(\Adele)}\cap\mathfrak{s}(T_{\GL{\glnidx}}(F))$ and $\chi^{(2)}$ is
trivial on $\mathfrak{s}(G_0(F))$.

The description of Section~\ref{subsubsection:global reps} was adapted from the exposition in 
\cite{KP} (Section~II.1), which we briefly recall. For $\underline{s}\in\C^{\glnidx}$ define a nongenuine character $\mathe_{\underline{s}}$ of $\cover{T}_{\GL{\glnidx}}(\Adele)$ by requiring ${\mathe}_{\underline{s}}(\alpha_i^{\vee}(x))=|x|^{(\alpha_i,\underline{s})/2}$ for all $2\leq i\leq\glnidx$ and $x\in\Adele^*$. Let $\sigma$ be a genuine character of $C_{\cover{T}_{\GL{\glnidx}}(\Adele)}$.
For $C_{\cover{T}_{\GL{\glnidx}}(\Adele)}<H<\cover{T}_{\GL{\glnidx}}(\Adele)$, the space $V_H(\sigma_{\underline{s}})$ consists of functions on $\CGLF{\glnidx}{\Adele}$.

Denote $K_{\GLF{\glnidx}{F_{\nu}}}=K_{\nu}\cap \GLF{\glnidx}{F_{\nu}}$ and for $\nu\notin S_0$, $K_{\GLF{\glnidx}{F_{\nu}}}^*=\kappa_{\nu}(K_{\GLF{\glnidx}{F_{\nu}}})$. We have the subgroup  $X^{\GL{\glnidx}}=\lmodulo{\mu_2^{\times}}{\prod'_{\nu}X_{\nu}^{\GL{\glnidx}}}$, where for almost all $\nu$, $X_{\nu}^{\GL{\glnidx}}=C(\cover{T}_{\GL{\glnidx}}(F_{\nu}),\cover{T}_{\GL{\glnidx}}(F_{\nu})\cap K_{\GLF{\glnidx}{F_{\nu}}}^*)$.
Also let $\Xi^{\GL{\glnidx}}=C(\cover{T}_{\GL{\glnidx}}(\Adele),\mathfrak{s}(T_{\GL{\glnidx}}(F)))$.
We alter between $V_{X^{\GL{\glnidx}}}(\sigma)$ and $V_{\Xi^{\GL{\glnidx}}}(\sigma)$, where
$f\in V_{X^{\GL{\glnidx}}}(\sigma)$ is mapped to
$f'\in V_{\Xi^{\GL{\glnidx}}}(\sigma)$ using a summation over $\lmodulo{(C_{\cover{T}_{\GL{\glnidx}}(\Adele)}\cap\mathfrak{s}(T_{\GL{\glnidx}}(F)))}
{T_{\GL{\glnidx}}(F)}$.

\begin{claim}\label{claim:global transitivity of induction}
Let $\chi$ be as above and denote by $V_{Q_{\glnidx}}(V_{\Xi^{\GL{\glnidx}}}(\chi^{(1)})\otimes \chi^{(2)})$ the subspace of right $\cover{K}$-smooth functions in $\Induced{\cover{Q}_{\glnidx}(\Adele)}{\cover{G}_{\glnidx}(\Adele)}{\Induced{\Xi^{\GL{\glnidx}}N_{\GL{\glnidx}}(\Adele)}{\CGLF{\glnidx}{\Adele}}
{\chi^{(1)}}\otimes\chi^{(2)}}$. Then
\begin{align*}
V_{\Xi}(\chi)\isomorphic V_{Q_{\glnidx}}(V_{\Xi^{\GL{\glnidx}}}(\chi^{(1)})\otimes \chi^{(2)}).
\end{align*}
\end{claim}
\begin{proof}
According to Section~\ref{subsection:Subgroups of the torus} (and \cite{KP} p.~59),
\begin{align}\label{eq:center of the torus split into GLn and G0}
&C_{\cover{T}_{\glnidx+1}(F_{\nu})}=\lmodulo{\setof{(\zeta,\zeta)}{\zeta\in\mu_2}}{(C_{\cover{T}_{\GL{\glnidx}}(F_{\nu})}
\times\cover{G}_0(F_{\nu}))},\\\label{eq:maximal abelian subgroup of the torus split into GLn and G0}
&T_{\glnidx+1}(F_{\nu})^{\mathm}=
\lmodulo{\setof{(\zeta,\zeta)}{\zeta\in\mu_2}}{(T_{\GL{\glnidx}}(F_{\nu})^{\mathm}\times\cover{G}_0(F_{\nu}))},
\end{align}
and for $\nu\notin S_0$,
\begin{align}\label{eq:unramified maximal abelian subgroup in an unramified place}
&C(\cover{T}_{\glnidx+1}(F_{\nu}),\cover{T}_{\glnidx+1}(F_{\nu})\cap K_{\nu}^*)\\\nonumber
&=\lmodulo{\setof{(\zeta,\zeta)}{\zeta\in\mu_2}}{(C(\cover{T}_{\GL{\glnidx}}(F_{\nu}),\cover{T}_{\GL{\glnidx}}(F_{\nu})\cap K_{\GLF{\glnidx}{F_{\nu}}}^*)\times \cover{G}_0(F_{\nu}))}.
\end{align}
In general let $\sigma$ be a genuine character of $C_{\cover{T}_{\glnidx+1}(F_{\nu})}$. Using \eqref{eq:center of the torus split into GLn and G0} we can write $\sigma_{\nu}=\sigma_{\nu}^{(1)}\otimes\sigma_{\nu}^{(2)}$ for unique genuine characters
$\sigma_{\nu}^{(1)}$ and $\sigma_{\nu}^{(2)}$ of $C_{\cover{T}_{\GL{\glnidx}}(F_{\nu})}$ and $\cover{G}_0(F_{\nu})$. Then
\begin{align}\label{eq:local isomorphism}
\induced{\cover{B}_{\glnidx}(F_{\nu})}{\cover{G}_{\glnidx}(F_{\nu})}{\rho_{\nu}(\sigma_{\nu})}
\isomorphic\induced{\cover{Q}_{\glnidx}(F_{\nu})}{\cover{G}_{\glnidx}(F_{\nu})}{
\induced{\cover{B}_{\GL{\glnidx}}(F_{\nu})}{\CGLF{\glnidx}{F_{\nu}}}{\rho_{\nu}(\sigma_{\nu}^{(1)})}\otimes
\sigma_{\nu}^{(2)}}.
\end{align}

Assume
\begin{align*}
X_{\nu}=\begin{cases}
C(\cover{T}_{\glnidx+1}(F_{\nu}),\cover{T}_{\glnidx+1}(F_{\nu})\cap K_{\nu}^*)&\nu\notin S_0,\\
\cover{T}_{\glnidx+1}(F_{\nu})^{\mathm}&\text{otherwise.}
\end{cases}
\end{align*}
Define $X^{\GL{\glnidx}}=\lmodulo{\mu_2^{\times}}{\prod'_{\nu}X_{\nu}^{\GL{\glnidx}}}$ where
\begin{align*}
X_{\nu}^{\GL{\glnidx}}=\begin{cases}
C(\cover{T}_{\GL{\glnidx}}(F_{\nu}),\cover{T}_{\GL{\glnidx}}(F_{\nu})\cap K_{\GLF{\glnidx}{F_{\nu}}}^*)&\nu\notin S_0,\\
\cover{T}_{\GL{\glnidx}}(F_{\nu})^{\mathm}&\text{otherwise.}
\end{cases}
\end{align*}
Now \eqref{eq:maximal abelian subgroup of the torus split into GLn and G0}-\eqref{eq:local isomorphism} yield
\begin{align*}
\Induced{XN_{\glnidx}(\Adele)}{\cover{G}_{\glnidx}(\Adele)}{\chi}\isomorphic
\Induced{\cover{Q}_{\glnidx}(\Adele)}{\cover{G}_{\glnidx}(\Adele)}{\Induced{X^{\GL{\glnidx}}N_{\GL{\glnidx}}(\Adele)}{\CGLF{\glnidx}{\Adele}}
{\chi^{(1)}}\otimes\chi^{(2)}}
\end{align*}
and the isomorphism between the spaces follows because
\[
\lmodulo{(C_{\cover{T}_{\glnidx+1}(\Adele)}\cap\mathfrak{s}(T_{\glnidx+1}(F)))}{T_{\glnidx+1}(F)}\isomorphic
\lmodulo{(C_{\cover{T}_{\GL{\glnidx}}(\Adele)}\cap\mathfrak{s}(T_{\GL{\glnidx}}(F)))}{T_{\GL{\glnidx}}(F)}.\qedhere
\]
\end{proof}

\subsection{Global exceptional (small) representations}\label{subsection:global small representations}

\subsubsection{Definition and basic properties}\label{subsubsection:global defs and basic properties}
Let $\chi$ be a genuine character of $C_{\cover{T}_{\glnidx+1}(\Adele)}$, which is trivial on $C_{\cover{T}_{\glnidx+1}(\Adele)}\cap\mathfrak{s}(T_{\glnidx+1}(F))$.
Analogously to the local case, call $\chi$ exceptional if for all
$\alpha\in\Delta_{G_{\glnidx}}$, $\chi({{\alpha^{\vee}}^*(x^{\mathfrak{l}(\alpha)})})=|x|$ for all $x\in \Adele^*$. 
We have $\chi=\otimes'\chi_{\nu}$ where $\chi_{\nu}$ is a local exceptional character of $C_{\cover{T}_{\glnidx+1}(F_{\nu})}$.

Put
$\Res_{\underline{s}=\underline{0}}=\lim_{\underline{s}\rightarrow\underline{0}}\prod_{i=2}^{\glnidx}(s_i-s_{i+1})s_{\glnidx+1}$.
For $f\in V_X(\chi)$, define
\begin{align*}
\theta_f(g)=\Res_{\underline{s}=\underline{0}}E_{B_{\glnidx}}^{N_{\glnidx}}(g;f,\underline{s})
,\qquad g\in\cover{G}_{\glnidx}(\Adele).
\end{align*}
\begin{proposition}\label{proposition:global small rep}
Let $\chi$ be an exceptional character and let $f\in V_X(\chi)$.
\begin{enumerate}[leftmargin=*]
\item The residue $\Res_{\underline{s}=\underline{0}}M'(w_0,\chi_{\underline{s}})f'_{\underline{s}}(g)$ is finite and nonzero.
If $w\ne w_0$, $\Res_{\underline{s}=\underline{0}}M'(w,\chi_{\underline{s}})f'_{\underline{s}}=0$.
\item The constant term $\theta_f^{N_{\glnidx}}$ of $\theta_f$ along $N_{\glnidx}$ is the image of $M'(w_0,\chi)$:
\begin{align*}
\int_{\lmodulo{N_{\glnidx}(F)}{N_{\glnidx}(\Adele)}}\theta_f(\mathfrak{s}(u)g)du=M'(w_0,\chi)f'(g).
\end{align*}
Here we normalize the measure by requiring the volume of $\lmodulo{F}{\Adele}$ to be $1$.
\item The representation $\Theta_{G_{\glnidx},\chi}$ spanned by 
$\theta_f$ as $f$ varies in $V_X(\chi)$ is irreducible, automorphic, genuine and belongs to
$L^2(\lmodulo{\mathfrak{s}(G_{\glnidx}(F))}{\cover{G}_{\glnidx}(\Adele)})$. It is isomorphic to $\otimes'_{\nu}\Theta_{G_{\glnidx},\chi_{\nu}}$.
\end{enumerate}
\end{proposition}
\begin{proof}
The proof follows exactly as in 
\cite{BFG} (Proposition~3.1 and Theorem~3.2) and 
\cite{KP} (Proposition~II.1.2 and
Theorems~II.1.4 and II.2.1), we briefly reproduce the argument. The assertions follow from the calculation of the poles of $M'(w,\chi_{\underline{s}})f'_{\underline{s}}(g)$. When we take a pure tensor $f$, $M(w,\chi_{\underline{s}})f_{\underline{s}}=(\otimes_{\nu\in S}\varphi_{\nu})\otimes(\otimes'_{\nu\notin S}\varphi_{\nu})$ where $S\supset S_0$ is a finite set of places depending on $\chi$ and $f$.
The infinite unramified part is $\otimes'_{\nu\notin S}\varphi_{\nu}$. For any fixed place $\nu$, the local intertwining operator is holomorphic since $\chi_{\underline{s},\nu}$ belongs to the positive Weyl chamber when $\underline{s}$ tends to $\underline{0}$. Thus the poles are located in the infinite unramified part, they are determined by the product $\prod_{\nu\notin S}c(\mathbf{w},\chi_{\underline{s},\nu})$.
If $\mathbf{w}=\mathbf{w_0}$, according to Claim~\ref{claim:values of exceptional unramified character on a alpha} the poles belong to
\begin{align*}
\frac{\prod_{2\leq i<j\leq\glnidx+1}\zeta(j-i+s_j-s_i)\zeta(-j-i+2(\glnidx+2)+s_i+s_j)\prod_{i=2}^{\glnidx+1}\zeta(\glnidx+2-i+s_i)}
{\prod_{2\leq i<j\leq\glnidx+1}\zeta(1+j-i+s_j-s_i)\zeta(1-j-i+2(\glnidx+2)+s_i+s_j)\prod_{i=2}^{\glnidx+1}\zeta(\glnidx+3-i+s_i)}.
\end{align*}
Here $\zeta$ denotes the partial Dedekind zeta function of $F$ with respect to $S$.
Recall that $\zeta$ is holomorphic except at $s=1$, where it has a simple pole, and is nonzero on the right half-plane $\Re(s)\geq1$. When $\underline{s}\rightarrow\underline{0}$, precisely $\glnidx$ zeta functions contribute a pole. If $\mathbf{w}\ne\mathbf{w_0}$, one of these functions will be omitted. Finally, the image $M'(w_0,\chi)V_{\Xi}(\chi)$ is isomorphic
to $\otimes'_{\nu}M_{\nu}(w_0,\chi_{\nu})V(\chi_{\nu})$ which is $\otimes'_{\nu}\Theta_{G_{\glnidx},\chi_{\nu}}$.
\end{proof}
The representation $\Theta_{G_{\glnidx},\chi}$ is the global exceptional representation.

\subsubsection{Explicit construction of exceptional characters}\label{subsubsection:Explicit construction of an exceptional representation global}
Let $\eta$ be a Hecke character, i.e., a character of $\lmodulo{F^*}{\Adele^*}$,
and let $\psi$ be a nontrivial character of $\lmodulo{F}{\Adele}$. Let $\gamma_{\psi}$  be the corresponding global Weil factor. One can mimic the local construction of Section~\ref{subsubsection:Explicit construction of an exceptional representation} and construct a global exceptional character. Define for
$t=\prod_{i=1}^{\glnidx}\eta_i^{\vee}(a_i^2)\beta_1^{\vee}(t_1)\in T_{\glnidx+1}(\Adele)^2$, $d\in(\Adele^*)^{2/\gcd(2,\glnidx+1)}$, $z=\prod_{i=1}^{\glnidx}\eta_i^{\vee}(d)$ and $\zeta\in\mu_2$,
\begin{align}\label{eq:global form of exceptional character}
\chi(\zeta\mathfrak{s}(t)\mathfrak{s}(z))=\zeta
\prod_{i=1}^{\glnidx}|a_i|^{\glnidx-i+1}\eta(\Upsilon(t))|d|^{\glnidx(\glnidx+1)/4}\eta(d)^{\glnidx}
\gamma_{\psi}^{\lceil\glnidx/2\rceil}(d).
\end{align}
This is a global exceptional character and it decomposes as $\otimes'\chi_{\nu}$, where
$\chi_{\nu}$ is given by \eqref{eq:explicit exceptional character}. 
For almost all $\nu$, $\chi_{\nu}$ is unramified.

If $\glnidx$ is odd, in contrast with the local case, this description does not exhaust all exceptional characters, as one may replace $\gamma_{\psi}$ with a more general class of functions. However, $\chi|_{\mathfrak{s}(T_{\glnidx+1}(\Adele)^2)}$ is uniquely described by $\eta$ and \eqref{eq:global form of exceptional character}.

If we write $\chi=\chi^{(1)}\otimes \chi^{(2)}$ in the notation of
Section~\ref{subsubsection:global induction transitivity},
\begin{align}\label{eq:global form of exceptional character GLn part}
&\chi^{(1)}(\zeta\mathfrak{s}(\prod_{i=1}^{\glnidx}\eta_i^{\vee}(a_i^2))\mathfrak{s}(\prod_{i=1}^{\glnidx}\eta_i^{\vee}(d)))=\zeta
\prod_{i=1}^{\glnidx}|a_i|^{\glnidx-i+1}\eta(d^{\glnidx}\prod_{i=1}^{\glnidx}a_i^2)|d|^{\glnidx(\glnidx+1)/4}
\gamma_{\psi}^{\lceil\glnidx/2\rceil}(d),\\\nonumber
&\chi^{(2)}(\zeta\mathfrak{s}(\beta_1^{\vee}(t_1)))=\zeta\eta(t^{-2}).
\end{align}


\subsubsection{The constant term}\label{subsubsection:the constant term}
Let $\chi$ be an exceptional character and let $\theta$ be an automorphic form in the space of
$\Theta_{G_{\glnidx},\chi}$. For any $1\leq k\leq\glnidx$ define the constant term of $\theta$ along $U_{k}$ by
\begin{align}\label{int:the constant term of theta along a unipotent subgroup}
\theta^{U_{k}}(g)=\int_{\lmodulo{U_{k}(F)}{U_{k}(\Adele)}}
\theta(\mathfrak{s}(u)g)du,\qquad g\in\cover{G}_{\glnidx}(\Adele).
\end{align}
%
Let $\chi$ be an exceptional character. We prove Theorem~\ref{theorem:the constant term}. Namely, the function $m\mapsto\theta^{U_{\glnidx}}(m)$ ($m\in\cover{M}_{\glnidx}(\Adele)$) belongs to the space of $\Theta_{\GL{\glnidx},\absdet{}^{-1/2}\chi^{(1)}}\otimes\Theta_{G_{0},\chi^{(2)}}$, where $\chi^{(1)}$ and $\chi^{(2)}$ are exceptional characters such that $\chi=\chi^{(1)}\otimes\chi^{(2)}$.
\begin{proof}[Proof of Theorem~\ref{theorem:the constant term}]
Write $\theta=\theta_{f}$ for some $f\in V_{X}(\chi)$, $\theta_f(g)=\Res_{\underline{s}=\underline{0}}E_{B_{\glnidx}}^{N_{\glnidx}}(g;f,\underline{s})$
(see Section~\ref{subsubsection:global defs and basic properties}). Let $\mathcal{W}\subset W_{\glnidx}$ be a set satisfying
$G_{\glnidx}=\coprod_{\mathbf{w}\in\mathcal{W}}B_{\glnidx}w^{-1}Q_{\glnidx}$. For $X<Q_{\glnidx}$ and $\mathbf{w}\in\mathcal{W}$, set $X^w=\rconj{w}B_{\glnidx}
\cap X$. 
According to M{\oe}glin and Waldspurger \cite{MW2} (Section~II.1.7),
\begin{align*}
\int_{\lmodulo{U_{\glnidx}(F)}{U_{\glnidx}(\Adele)}}\sum_{\delta\in\lmodulo{B_{\glnidx}(F)}{G_{\glnidx}(F)}}
f'_{\underline{s}}(\mathfrak{s}(\delta)\mathfrak{s}(u)g)du=
\sum_{\mathbf{w}\in\mathcal{W}}\ \sum_{m\in\lmodulo{M_{\glnidx}^w(F)}{M_{\glnidx}(F)}}M'(w,\chi_{\underline{s}})f'_{\underline{s}}(\mathfrak{s}(m)g).
\end{align*}
Here $M'(w,\chi_{\underline{s}})$ is defined by meromorphic continuation of an integral over $\lmodulo{U_{\glnidx}^w(F)}{U_{\glnidx}(\Adele)}$. As we argue below, the domain of integration can be changed and $M'(w,\chi_{\underline{s}})$ is actually given by \eqref{eq:global intertwining def}.

We describe a choice of the set $\mathcal{W}$. There are $2^{\glnidx}$ representatives to consider. For $\underline{b}\in\{0,1\}^{\glnidx}$, set $\mathbf{w}_{\underline{b}}=\prod_{i=1}^{\glnidx}\mathbf{w_{\epsilon_{i+1}}}^{b_i}$ (a product of commuting reflections). The elements $\{\mathbf{w}_{\underline{b}}\}_{\underline{b}}$ are a set of representatives, but
$w_{\underline{b}}N_{\glnidx}w_{\underline{b}}^{-1}$ might not contain $N_{\GL{\glnidx}}$. It will be convenient to select elements $\mathbf{w}$ which satisfy
\begin{align}\label{containment:convenient representative condition}
N_{\GL{\glnidx}}<\rconj{w}N_{\glnidx}
\end{align}
(see below).
Write $\underline{b}=(1^{m_1}0^{l_1}\ldots1^{m_k}0^{l_k})$, where $m_i,l_i\geq0$,
$m=\sum_{i=1}^km_i$, $l=\sum_{i=1}^kl_i$, $m+l=\glnidx$ and we assume $k$ is minimal such that $\underline{b}$ can be written in this form (e.g., if $m_i,m_{i+1}>0$, $l_i>0$). For any $a,b,c\geq0$ such that $a+b+c\leq\glnidx$, put
\begin{align*}
&\omega_{a,b,c}=diag(I_a,\left(\begin{array}{cc}0&I_{\glnidx-a-b-c}\\J_b&0\end{array}\right),I_c)\in\GLF{\glnidx}{F},\\
&\alpha_{\underline{b},j}=\omega_{\sum_{i=1}^{j-1}{l_i},m_j,\sum_{i=1}^{j-1}{m_i}}\qquad \forall 1\leq j\leq k
\end{align*}
(if $m_j=0$, $\alpha_{\underline{b},j}=I_{\glnidx}$), $\mathbf{z}_{\underline{b}}=\alpha_{\underline{b},k}\ldots\alpha_{\underline{b},1}$.  Our set of representatives is
$\mathcal{W}=\setof{\mathbf{z}_{\underline{b}}\mathbf{w}_{\underline{b}}}{\underline{b}\in\{0,1\}^{\glnidx}}$.
The elements $\mathbf{w}\in \mathcal{W}$ satisfy \eqref{containment:convenient representative condition}.

Since $U_{\glnidx}^w<N_{\glnidx}$, we may write the $du$-integration of $M'(w,\chi_{\underline{s}})$ over
$\lmodulo{U_{\glnidx}^w(\Adele)}{U_{\glnidx}(\Adele)}$. The condition \eqref{containment:convenient representative condition} implies $\rconj{w}N_{\glnidx}\cap N_{\glnidx}=N_{\GL{\glnidx}}(\rconj{w}N_{\glnidx}\cap U_{\glnidx})$,
whence
$\lmodulo{U_{\glnidx}^w}{U_{\glnidx}}=
\lmodulo{(\rconj{w}N_{\glnidx}\cap N_{\glnidx})}{N_{\glnidx}}$.
Therefore $M'(w,\chi_{\underline{s}})f'_{\underline{s}}\in V_{\Xi}(\rconj{\mathbf{w}}\chi)$.

We use the notation of Section~\ref{subsubsection:global induction transitivity}. Write
$\rconj{\mathbf{w}}\chi=\chi_{\mathbf{w}}^{(1)}\otimes\chi_{\mathbf{w}}^{(2)}$, where $\chi_{\mathbf{w}}^{(1)}$ and $\chi_{\mathbf{w}}^{(2)}$ are genuine characters of $C_{\cover{T}_{\GL{\glnidx}}(\Adele)}$ and $\cover{G}_0(\Adele)$. According to Claim~\ref{claim:global transitivity of induction} we can regard $M'(w,\chi_{\underline{s}})f'_{\underline{s}}$ as an element of $V_{Q_{\glnidx}}(V_{\Xi^{\GL{\glnidx}}}(\chi_{\mathbf{w}}^{(1)})_{\underline{z}}\otimes (\chi_{\mathbf{w}}^{(2)})_{s_{\glnidx+1}})$,
where $\underline{z}\in\C^{\glnidx}$ satisfies $\mathe_{\underline{z}}=\rconj{\mathbf{w}}\mathe_{\underline{s}}|_{T_{\GL{\glnidx}}(\Adele)}$. Put $\underline{s}^{(\glnidx)}=(s_1,\ldots,s_{\glnidx})$.

The mapping
\begin{align*}
\{f,w,g,s_{\glnidx+1}\}(\underline{s}^{(\glnidx)},b)=M'(w,\chi_{\underline{s}})f'_{\underline{s}}(bg)
\qquad (b\in \CGLF{\glnidx}{\Adele})
\end{align*}
belongs to $V_{\Xi^{\GL{\glnidx}}}(\absdet{}^{\glnidx/2}(\chi_{\mathbf{w}}^{(1)})_{\underline{z}})$ (we used $\delta_{Q_{\glnidx}(\Adele)}^{1/2}=\absdet{}^{\glnidx/2}$).

By \eqref{containment:convenient representative condition}, $\GLF{\glnidx}{F}^w=B_{\GL{\glnidx}}(F)$ hence $\lmodulo{M_{\glnidx}^w(F)}{M_{\glnidx}(F)}\isomorphic
\lmodulo{\GLF{\glnidx}{F}^w}{\GLF{\glnidx}{F}}$ and the summation is an Eisenstein series with respect to $B_{\GL{\glnidx}}$, applied to $\{f,w,g,s_{\glnidx+1}\}$.
Thus
\begin{align*}
\theta_f(g)=\sum_{\mathbf{w}\in\mathcal{W}}\Res_{\underline{s}=\underline{0}}
E_{B_{\GL{\glnidx}}}(I_{\glnidx};\{f,w,g,s_{\glnidx+1}\},\underline{s}^{(\glnidx)}).
\end{align*}

Fix $\mathbf{w}=\mathbf{z}_{\underline{b}}\mathbf{w}_{\underline{b}}$ and let $m_i,l_i,m,l$ be as above.
Let $\mathcal{U}_{\alpha_{\glnidx+1}}$ be the unipotent subgroup of $N_{\glnidx}(\Adele)$ generated by $n_{\alpha_{\glnidx+1}}$. Then $\mathcal{U}_{\alpha_{\glnidx+1}}<(\rconj{w}N_{\glnidx}\cap N_{\glnidx})(\Adele)$ if and only if $m=0$.
\begin{claim}\label{claim:constant term pole of intertwiner}
The function $M'(w,\chi_{\underline{s}})f'_{\underline{s}}$ is holomorphic at $\underline{s}$ except perhaps for a simple pole at $s_{\glnidx+1}=0$. The pole exists for some data $f'$ if and only if $m>0$.
\end{claim}

\begin{claim}\label{claim:constant term pole of series}
The series $E_{B_{\GL{\glnidx}}}(I_{\glnidx};\{f,w,g,s_{\glnidx+1}\},\underline{s}^{(\glnidx)})$ is holomorphic at $\underline{s}^{(\glnidx)}$ except perhaps for simple poles at $s_i-s_{i+1}$ for $1\leq i<\glnidx$. If $l>0$ and $m>0$, it has at most $\glnidx-2$ poles. If $l=0$, the function $y\mapsto E_{B_{\GL{\glnidx}}}(I_{\glnidx};\{f,w,yg,s_{\glnidx+1}\},\underline{s}^{(\glnidx)})$ on $\CGLF{\glnidx}{\Adele}$ belongs to the space of $\Theta_{\GL{\glnidx},\absdet{}^{-1/2}\chi^{(1)}}$.
\end{claim}

Before proving these claims, we use them to complete the proof of the theorem. According to Claim~\ref{claim:constant term pole of intertwiner}, $s_{\glnidx+1}\cdot \{f,w,g,s_{\glnidx+1}\}$ is holomorphic and vanishes at $s_{\glnidx+1}=0$ if $m=0$. Hence we may assume $m>0$.  Then by Claim~\ref{claim:constant term pole of intertwiner} the series $E_{B_{\GL{\glnidx}}}(I_{\glnidx};\{f,w,g,s_{\glnidx+1}\},\underline{s}^{(\glnidx)})$ multiplied by $\prod_{i=1}^{\glnidx-1}(s_i-s_{i+1})$ vanishes for $\underline{s}\rightarrow\underline{0}$, unless $l=0$ and then $\mathbf{w}=\mathbf{z}_{\underline{1}}\mathbf{w}_{\underline{1}}$. Therefore
\begin{align*}
\theta_f(g)=\Res_{\underline{s}=\underline{0}}
E_{B_{\GL{\glnidx}}}(I_{\glnidx};\{f,w,g,s_{\glnidx+1}\},\underline{s}^{(\glnidx)}).
\end{align*}
Now Claim~\ref{claim:constant term pole of intertwiner} shows that $y\mapsto\theta_f(yg)$ belongs to the space of $\Theta_{\GL{\glnidx},\absdet{}^{-1/2}\chi^{(1)}}$. Clearly
$h\mapsto\theta_f(hg)$ ($h\in\cover{G}_0(\Adele)$) lies in the space of $\Theta_{G_0,\chi^{(2)}}$. The theorem is proved.

\begin{proof}[Proof of Claim~\ref{claim:constant term pole of intertwiner}]
As in the proof of Proposition~\ref{proposition:global small rep} and with the same notation, we analyze the poles of
$M'(w,\chi_{\underline{s}})f'_{\underline{s}}$ by looking at $M(w,\chi_{\underline{s}})f_{\underline{s}}$ for a pure tensor
$f\in V_{X}(\chi)$. Denote
$M(w,\chi_{\underline{s}})f_{\underline{s}}=(\otimes_{\nu\in S}\varphi_{\nu})\otimes(\otimes'_{\nu\notin S}\varphi_{\nu})$.
At each $\nu\in S$, the local intertwining operator is holomorphic
around $\underline{0}$, because when $\underline{s}\rightarrow\underline{0}$, $\chi_{\underline{s},\nu}$ belongs to the positive Weyl chamber. The poles of the intertwining operator coincide with the poles of the infinite product. If $m=0$, the integration is trivial hence clearly there is no pole. Otherwise, the poles are contained in the set of poles of
\begin{align*}
\frac{\prod_{2\leq i<j\leq\glnidx+1}\zeta(-j-i+2(\glnidx+2)+s_i+s_j)\prod_{i=2}^{\glnidx+1}\zeta(\glnidx+2-i+s_i)}
{\prod_{2\leq i<j\leq\glnidx+1}\zeta(1-j-i+2(\glnidx+2)+s_i+s_j)\prod_{i=2}^{\glnidx+1}\zeta(\glnidx+3-i+s_i)}.
\end{align*}
This product is holomorphic except for a simple pole at $s_{\glnidx+1}=0$. Furthermore, since
$\mathcal{U}_{\alpha_{\glnidx+1}}$ is not a subgroup of $(\rconj{w}N_{\glnidx}\cap N_{\glnidx})(\Adele)$, the factor $\zeta(\glnidx+2-i+s_i)$ with $i=\glnidx+1$ contributes to the poles of the intertwining operator. Hence there is a pole at $s_{\glnidx+1}=0$.
\end{proof}

\begin{proof}[Proof of Claim~\ref{claim:constant term pole of series}]
To compute the poles of $E_{B_{\GL{\glnidx}}}(I_{\glnidx};\{f,w,g,s_{\glnidx+1}\},\underline{s}^{(\glnidx)})$ we consider the constant term of $\{f,w,g,s_{\glnidx+1}\}$ along $N_{\GL{\glnidx}}$. According to Kazhdan and Patterson \cite{KP} (Proposition~II.1.2),
\begin{align*}
&\int_{\lmodulo{N_{\GL{\glnidx}}(F)}{N_{\GL{\glnidx}}(\Adele)}}
E_{B_{\GL{\glnidx}}}(\mathfrak{s}(u);\{f,w,g,s_{\glnidx+1}\},\underline{s}^{(\glnidx)})du\\&
=\sum_{\omega\in W_{\GL{\glnidx}}}(M^{\GL{\glnidx}})'(\omega,\underline{s}^{(\glnidx)})
\{f,w,g,s_{\glnidx+1}\}'_{\underline{s}^{(\glnidx)}}(I_{\glnidx}).
\end{align*}
Here $W_{\GL{\glnidx}}$ is the Weyl subgroup of $\GL{\glnidx}$ and
\begin{align*}
(M^{\GL{\glnidx}})'(\omega,\underline{s}^{(\glnidx)}):V_{\Xi^{\GL{\glnidx}}}(\absdet{}^{\glnidx/2}(\chi_{\mathbf{w}}^{(1)})_{\underline{z}})\rightarrow V_{\Xi^{\GL{\glnidx}}}(\absdet{}^{\glnidx/2}\ \rconj{\mathfrak{s}(\omega)}((\chi_{\mathbf{w}}^{(1)})_{\underline{z}}))
\end{align*}
is the corresponding intertwining operator.

We may assume that $\chi|_{\mathfrak{s}(T_{\glnidx+1}(\Adele)^2)}$ is defined with respect to a Hecke character $\eta$ according to \eqref{eq:global form of exceptional character}.
Since 
$g\mapsto\eta(\Upsilon(g))$ is an automorphic character of $G_{\glnidx}(\Adele)$ which extends to a nongenuine automorphic character of $\cover{G}_{\glnidx}(\Adele)$, and $(\eta\Upsilon\cdot\chi)_{\mathbf{w}}^{(1)}=\eta\Upsilon\cdot(\chi_{\mathbf{w}}^{(1)})$, we may prove the claim for $(\eta^{-1}\Upsilon)\cdot\chi$. This means that we assume $\eta=1$, then
for $t=\prod_{i=1}^{\glnidx}\eta_i^{\vee}(a_i^2)\beta_1^{\vee}(t_1)\in T_{\glnidx+1}(\Adele)^2$, $d\in(\Adele^*)^{2/\gcd(2,\glnidx+1)}$, $z=\prod_{i=1}^{\glnidx}\eta_i^{\vee}(d)$ and $\zeta\in\mu_2$, 
\begin{align}\label{eq:simplifying assumption constant term chi1 on squares}
&\chi(\zeta\mathfrak{s}(t)\mathfrak{s}(z))=\zeta
\prod_{i=1}^{\glnidx}|a_i|^{\glnidx-i+1}|d|^{\glnidx(\glnidx+1)/4}
\gamma(d).
\end{align}
Here $\gamma$ is a complex-valued function on the quotient of $(\Adele^*)^{2/\gcd(2,\glnidx+1)}$ by
$((\Adele^*)^{2/\gcd(2,\glnidx+1)}\cap F^*)\Adele^{*2}$,
which satisfies $\gamma(dd')=c(d,d')^{\lceil\glnidx/2\rceil}\gamma(d)\gamma(d')$ and is nonzero.
Additionally, by Claim~\ref{claim:constant term pole of intertwiner} we may assume $s_{\glnidx+1}=0$.

The element $\mathbf{z}_{\underline{b}}$ acts on the cocharacters $\eta_i^{\vee}$. Let $\tau$ be the permutation such that
\begin{align*}
\rconj{\mathbf{z}_{\underline{b}}^{-1}}(\prod_{i=1}^{\glnidx}\eta_i^{\vee}(a_i))=
\prod_{i=1}^{\glnidx}\eta_i^{\vee}(a_{\tau^{-1}(i)}).
\end{align*}
Also define $\underline{c}\in\{-1,1\}^{\glnidx}$ by $c_i=1-2b_i$.
Then
\begin{align*}
(\chi_{\mathbf{w}}^{(1)})_{\underline{z}}(\mathfrak{s}(\prod_{i=1}^{\glnidx}\eta_i^{\vee}(a_i^2)))=\prod_{i=1}^{\glnidx}|a_i|^{c_{\tau(i)}(\glnidx-\tau(i)+1+s_{\tau(i)})}.
\end{align*}
Note that here we used the fact that $\chi(\mathfrak{s}(\beta_1^{\vee}(t_1)))=1$, because (see e.g. \eqref{eq:formula for conjugation by w_0 with convenient coordinates})
\begin{align*}
\rconj{\mathbf{w}_{\underline{b}}}(\prod_{i=1}^{\glnidx}\eta_i^{\vee}(a_i))=\prod_{i=1}^{\glnidx}\eta_i^{\vee}(a_i^{\mp1})\beta_1^{\vee}(\cdots).
\end{align*}

We show that $\chi_{\mathbf{w}}^{(1)}$ belongs to the positive Weyl chamber. 
Denote by $\sqcup$ the concatenation of tuples, e.g., $(2,5,1)\sqcup(2,8,-1,0)=(2,5,1,2,8,-1,0)$.
For $r$ integers $i_1,\ldots,i_r$, set $\rconj{J}(i_1,i_2,\ldots,i_r)=(i_r,\ldots,i_2,i_1)$. Let
$(1,\ldots,\glnidx)=[l_1]\sqcup\ldots\sqcup[l_k]\sqcup[m_k]\sqcup\ldots\sqcup[m_1]$, where for $r\geq0$, $[r]$ is an ascending sequence of $r$ integers. Then
\begin{align*}
(\tau^{-1}(1),\ldots,\tau^{-1}(\glnidx))=\rconj{J}[m_1]\sqcup[l_1]\sqcup\ldots\sqcup\rconj{J}[m_k]\sqcup[l_k].
\end{align*}
It follows that if $1\leq i<j<l$, $\tau(i)<\tau(j)$ and $c_{\tau(i)}=c_{\tau(j)}=1$. Thus for $x\in\Adele^*$,
\begin{align}\label{eq:poles inside the l part}
\chi_{\mathbf{w}}^{(1)}(\mathfrak{s}(\eta_i^{\vee}(x^2)\eta_j^{\vee}(x^{-2})))=|x|^{\glnidx-\tau(i)+1-(\glnidx-\tau(j)+1)}=|x|^{\tau(j)-\tau(i)}.
\end{align}
For $l+1\leq i<j<\glnidx$, we have $\tau(i)>\tau(j)$ and $c_{\tau(i)}=c_{\tau(j)}=-1$, whence
\begin{align}\label{eq:poles inside the m part}
\chi_{\mathbf{w}}^{(1)}(\mathfrak{s}(\eta_i^{\vee}(x^2)\eta_j^{\vee}(x^{-2})))=|x|^{\tau(i)-\tau(j)}.
\end{align}
If $l\leq i\leq l$ and $l+1\leq j\leq\glnidx$, then $c_{\tau(i)}=1$ and $c_{\tau(j)}=-1$,
\begin{align}\label{eq:pole between l and m part}
\chi_{\mathbf{w}}^{(1)}(\mathfrak{s}(\eta_l^{\vee}(x^2)\eta_{l+1}^{\vee}(x^{-2})))=|x|^{2\glnidx+2-\tau(i)-\tau(j)}.
\end{align}
We conclude that $\chi_{\mathbf{w}}^{(1)}$ belongs to the positive Weyl chamber. Therefore, for decomposable data, at any place $\nu$ the local intertwining operator is holomorphic. Then the poles of the intertwining operator $(M^{\GL{\glnidx}})'(\omega,\underline{s}^{(\glnidx)})$ appear only in the infinite unramified part.

Consider such poles. For the purpose of a bound, we may assume $\omega=J_{\glnidx}$. Let $S\supset S_0$ be a finite set. For a positive root $(i,j)$, where $1\leq i<j\leq\glnidx$, and $\nu\notin S$,
let
\begin{align*}
d_{(i,j),\nu}=(\chi_{\mathbf{w}}^{(1)})_{\nu}(\mathfrak{s}_{\nu}(\eta_i^{\vee}(\varpi_{\nu}^2)\eta_j^{\vee}(\varpi_{\nu}^{-2})))=q_{\nu}^{-C_{(i,j)}}.
\end{align*}
Because
$(\chi_{\mathbf{w}}^{(1)})_{\nu}$ belongs to the positive Weyl chamber, $C_{(i,j)}>0$. Hence the poles appear (with multiplicities) in the product
\begin{align*}
\prod_{1\leq i<j\leq\glnidx}\zeta(C_{(i,j)}+s_{\tau(j)}-s_{\tau(i)}).
\end{align*}
For a pair $(i,j)$ appearing in this product, there is a pole at $\underline{s}^{(\glnidx)}=\underline{0}$ if and only if
$C_{(i,j)}=1$. For any $i$, there is at most one $j$ such that $\tau(j)-\tau(i)=1$ and similarly, at most one $j$ with $\tau(i)-\tau(j)=1$. Looking at
\eqref{eq:poles inside the l part}-\eqref{eq:pole between l and m part} we see that the only possible poles are simple poles
at $s_i-s_{i+1}=0$ for each
$1\leq i<\glnidx$ such that $i\ne l$. If $l=0$ or $m=0$ (in which case
$l=\glnidx$), these are $\glnidx-1$ poles. Otherwise we have at most $\glnidx-2$ poles.

Now assume $l=0$. Then
\begin{align*}
&\absdet{}^{\glnidx/2}\chi_{\mathbf{w}}^{(1)}(\mathfrak{s}(\prod_{i=1}^{\glnidx}\eta_i^{\vee}(a_i^2))\mathfrak{s}(\prod_{i=1}^{\glnidx}\eta_i^{\vee}(d)))=\prod_{i=1}^{\glnidx}|a_i|^{\glnidx-i}|d|^{\glnidx(\glnidx-1)/4}\gamma(d^{-1}).
\end{align*}
This is an exceptional character of $C_{\cover{T}_{\GL{\glnidx}}(\Adele)}$ and therefore according to the results of Kazhdan and Patterson \cite{KP} (Proposition~II.1.2 and
Theorems~II.1.4 and II.2.1), the mapping
\begin{align*}
y\mapsto \lim_{\underline{s}^{(\glnidx)}\rightarrow\underline{0}}\prod_{i=1}^{\glnidx-1}(s_i-s_{i+1})E_{B_{\GL{\glnidx}}}(y;\{f,w,g,0
\},\underline{s}^{(\glnidx)})
\end{align*}
(we assumed $s_{\glnidx+1}=0$)
belongs to the space of $\Theta_{\GL{\glnidx},\absdet{}^{\glnidx/2}\chi_{\mathbf{w}}^{(1)}}$. Since
\begin{align*}
E_{B_{\GL{\glnidx}}}(y;\{f,w,g,s_{\glnidx+1}\},\underline{s}^{(\glnidx)})
=E_{B_{\GL{\glnidx}}}(I_{\glnidx};\{f,w,yg,0
,\underline{s}^{(\glnidx)}),
\end{align*}
it is left to show
$\absdet{}^{\glnidx/2}\chi_{\mathbf{w}}^{(1)}=\absdet{}^{-1/2}\chi^{(1)}$.
Equality~\eqref{eq:simplifying assumption constant term chi1 on squares} implies
\begin{align*}
\chi^{(1)}(\mathfrak{s}(\prod_{i=1}^{\glnidx}\eta_i^{\vee}(a_i^2))\mathfrak{s}(z))=\prod_{i=1}^{\glnidx}|a_i|^{\glnidx-i+1}
|d|^{\glnidx(\glnidx+1)/4}\gamma(d).
\end{align*}
Because $\gamma(d^{-1})=\gamma(d)$, the result follows. 
\end{proof}
\end{proof}

\begin{remark}\label{remark:constant term method for SO(2n+1) fails here}
The analog result for $SO_{2\glnidx+1}$ was proved in \cite{BFG,me7} by induction, using a parabolic subgroup with a Levi part isomorphic to $\GL{1}\times SO_{2(\glnidx-1)+1}$. This approach does not seem to work, because the representation of $M_1(\Adele)$ cannot be described using the tensor product. Note that the twist of $\chi^{(1)}$ by $\absdet{}^{-1/2}$ is compatible with the result on $SO_{2\glnidx+1}$.
\end{remark}

\subsubsection{Vanishing results}\label{subsubsection:global vanishing results}
Let $\psi$ be a nontrivial character of $\lmodulo{F}{\Adele}$. We use the notation of Section~\ref{subsubsection:vanishing and results}.
The global counterpart of the twisted Jacquet modules
is a class of Fourier coefficients, for generic characters.
Any character of $V_{\UnipotentOrbit}(\Adele)$ which is trivial on
$V_{\UnipotentOrbit}(F)$ is the lift of a character of $V_{\UnipotentOrbit}^{\matha}(\Adele)$, trivial on $V_{\UnipotentOrbit}^{\matha}(F)$. Such a character can be identified with a point $b\in V_{\UnipotentOrbit}^{\matha}(F)$. The character $\psi_b$ is generic if $b$ belongs to the open orbit with respect to the action of $C(G_{\glnidx}(\overline{F}),h_{\UnipotentOrbit}(\overline{F}^*))$ on $V_{\UnipotentOrbit}^{\matha}(\overline{F})$. The Fourier coefficient of an automorphic function $\varphi$ on $\cover{G}_{\glnidx}(\Adele)$ or $G_{\glnidx}(\Adele)$, with respect to $\UnipotentOrbit$ and $\psi_b$, is given by
\begin{align*}
\int_{\lmodulo{V_{\UnipotentOrbit}(F)}{V_{\UnipotentOrbit}(\Adele)}}\varphi(\mathfrak{s}(v))\psi_b(v)dv.
\end{align*}
This is the Fourier coefficient of Theorem~\ref{theorem:global vanishing result}.
This theorem follows immediately from Theorem~\ref{theorem:local vanishing result} using a local-global principle (see e.g. \cite{JR} Propositon~1). In particular, as in \cite{BFG} (Proposition~4.3), we have the following result (directly implied by Lemma~\ref{lemma:twisted Jacquet modules vanish on small representations}):
\begin{proposition}\label{proposition:global basic Fourier vanishing result}
Let $\chi$ be an exceptional character and $b\in F^{2\glnidx-1}$ be with $\ell(b)\ne0$. Then for any $\theta$ in the space of $\Theta_{G_{\glnidx},\chi}$,
$\int_{\lmodulo{U_1(F)}{U_1(\Adele)}}\theta(\mathfrak{s}(u))\psi(r(u)b)du=0$.
\end{proposition}

\section{Application - the theta period}\label{section:application 1 coperiod}
In this section we prove Theorem~\ref{theorem:co period GSpin}.
Let $F$ be a number field with a ring of ad\`{e}les $\Adele$. Denote by $A^+$ the subgroup of id\`{e}les of $F$ whose finite components are trivial and \archimedean\ components are equal, real and positive, it can be identified with $\R_{>0}$.

Regard $\GL{\glnidx}$ and $G_0$ as subgroups of $M_{\glnidx}$. Let $\GLF{\glnidx}{\Adele}^1<\GLF{\glnidx}{\Adele}$ be the subgroup of matrices $b$ such that $\absdet{b}=1$.
If $g=bhuk\in G_{\glnidx}(\Adele)$ with $b\in\GLF{\glnidx}{\Adele}$, $h\in G_0(\Adele)$,
$u\in U_{\glnidx}(\Adele)$ and $k\in K$, define $H(g)=|\det{b}|$.

Let $\cover{G}_{\glnidx}(\Adele)$ be the double cover of $G_{\glnidx}(\Adele)$, defined in Section~\ref{subsection:the double cover of G_n Adele}.
We identify $G_{\glnidx}(F)$ and $N_{\glnidx}(\Adele)$ with their images in $\cover{G}_{\glnidx}(\Adele)$ under $\mathfrak{s}$ (see Claim~\ref{claim:good ol section splits G(F) and N(A)}). Furthermore, regard $G_{\glnidx}(\Adele)$ as a subgroup of $\cover{G}_{\glnidx}(\Adele)$ by fixing a section $G_{\glnidx}(\Adele)\rightarrow\cover{G}_{\glnidx}(\Adele)$. E.g., one can take the section $\kappa$ defined in Section~\ref{subsection:the double cover of G_n Adele}. We suppress the actual section from the notation.

Let $\glnrep$ be an irreducible unitary automorphic cuspidal representation of $\GLF{\glnidx}{\Adele}$ with a central character $\omega_{\tau}$ and let $\eta$ be a unitary Hecke character. 
As in the statement of the theorem, assume
$\omega_{\tau}^2(t)\eta^{\glnidx}(t)=1$ for all $t\in A^+$.

Let $\rho$ be a smooth complex-valued function on $G_{\glnidx}(\Adele)$, satisfying the following properties: it is $K$ finite; for any $b^1\in\GLF{\glnidx}{\Adele}^1$, $h\in G_0(\Adele)$, $u\in U_{\glnidx}(\Adele)$, $t\in A^+$, $a=t\cdot I_{\glnidx}$ and $g\in G_{\glnidx}(\Adele)$,
$\rho(b^1ahug)=\delta_{Q_{\glnidx}(\Adele)}^{\half}(a)\omega_{\tau}(a)\eta(h)\rho(b^1g)$; and for any $k\in K$, the mapping $b^1\mapsto\rho(b^1k)$ is a $K\cap \GLF{\glnidx}{\Adele}^1$-finite vector in the space of $\glnrep$. The standard section corresponding to $\rho$ is defined by $\rho_s(g)=H(g)^s\rho(g)$, for any $s\in\C$.

We have the following Eisenstein series,
\begin{align*}
E_{Q_{\glnidx}}(g;\rho,s)=\sum_{\delta\in\lmodulo{Q_{\glnidx}(F)}{G_{\glnidx}(F)}}\rho_{s}(\delta g),\qquad g\in G_{\glnidx}(\Adele).
\end{align*}

According to Hundley and Sayag \cite{HS} (
Proposition~18.0.4), 
the series $E_{Q_{\glnidx}}(g;\rho,s)$ is holomorphic when $\Re(s)>0$ except perhaps at $s=1/2$, where it may have a simple pole. This pole exists (for some data $\rho,g$)
if and only if 
$L(s,\tau,Sym^2\otimes\eta)$ has a pole at $s=1$. Equivalently, one can use the partial $L$-function
$L^S(s,\tau,Sym^2\otimes\eta)$, where $S$ is any finite set of places, see \cite{HS} (Remark~2.2.2 and Section~19.3).

\begin{remark}\label{remark:eta inverse in HS}
In \cite{HS} the $L$-function was actually $L^S(s,\tau,Sym^2\otimes\eta^{-1})$.
This is because their embedding of $G_0$ in $G_{\glnidx}$ was different. In the notation of Section~\ref{subsection:properties of GSpin}, they sent $\theta_1^{\vee}$ to $-\epsilon_1^{\vee}$ (compare \eqref{eq:formula for conjugation by w_0 with convenient coordinates} to
their formula for $wmw^{-1}$ in \cite{HS} Lemma~13.2.4).
\end{remark}


Put $E_{1/2}(g;\rho)=\Res_{s=1/2}E_{Q_{\glnidx}}(g;\rho,s)$, where $\Res_{s=1/2}=\lim_{s\rightarrow1/2}(s-1/2)$.

In general if $H<G_{\glnidx}(\Adele)$ and we are given two complex-valued genuine functions $\varphi$ and $\varphi'$ on $\cover{H}$, the function $h\mapsto\varphi(h)\varphi'(h)$ is independent of the actual choice of section $H\rightarrow\cover{G}_{\glnidx}(\Adele)$. Therefore it can be regarded as a function on $H$.

Let $\chi$ and $\chi'$ be global exceptional characters, constructed as explained in Section~\ref{subsubsection:Explicit construction of an exceptional representation global}, and form the exceptional representations $\Theta_{G_{\glnidx},\chi}$ and $\Theta_{G_{\glnidx},\chi'}$. We take $\chi$ and $\chi'$ such that $\chi\cdot\chi'\cdot\eta=1$ on $C_{G_{\glnidx}(\Adele)}$, which means
\begin{align}\label{eq:condition for chi and chi' and eta}
(\chi\cdot\chi')(\beta^{\vee}(t))\eta(t)=1,\qquad\forall t\in\Adele^*.
\end{align}
This is possible (see Section~\ref{subsubsection:Explicit construction of an exceptional representation global}).
Let $\theta$ (resp. $\theta'$) be an automorphic form in the space of $\Theta_{G_{\glnidx},\chi}$ (resp. $\Theta_{G_{\glnidx},\chi'}$). 

For any complex-valued function $\Xi$ on $\lmodulo{G_{\glnidx}(F)}{G_{\glnidx}(\Adele)}$, satisfying
$\Xi(\beta^{\vee}(t)g)=\eta(t)\Xi(g)$ for all $g\in G_{\glnidx}(\Adele)$ and $t\in\Adele^*$, the co-period of $\Xi$, $\theta$ and $\theta'$ is the following integral
\begin{align*}
\mathcal{I}(\Xi,\theta,\theta')=\int_{\lmodulo{C_{G_{\glnidx}(\Adele)}G_{\glnidx}(F)}{G_{\glnidx}(\Adele)}}\Xi(g)\theta(g)\theta'(g)dg,
\end{align*}
provided it is absolutely convergent. Note that $H(\beta^{\vee}(t)g)=H(g)$, hence $\rho_s, E_{Q_{\glnidx}}(\cdot;\rho,s)$ and $E_{1/2}(\cdot;\rho)$ are, at least formally, possible candidates for $\Xi$. 

Let $dt$ be the standard Lebesgue measure on $\R$. We fix measures such that the following formulas hold.
\begin{align}\nonumber
&\int_{\lmodulo{Q_{\glnidx}(F)}{G_{\glnidx}(\Adele)}}\varphi(g)dg=
\int_{K}\int_{\lmodulo{M_{\glnidx}(F)}{M_{\glnidx}(\Adele)}}\int_{\lmodulo{U_{\glnidx}(F)}{U_{\glnidx}(\Adele)}}
\varphi(umk)du\delta_{Q_{\glnidx}(\Adele)}^{-1}(m)dmdk,\\
\label{eq:equality defining the constant}
&\int_{\lmodulo{\GLF{\glnidx}{F}}{\GLF{\glnidx}{\Adele}}}\varphi(b)db=\glnidx^{-1}
\int_{\lmodulo{\GLF{\glnidx}{F}}{\GLF{\glnidx}{\Adele}^1}}\int_{\R_{>0}}\varphi((t\cdot I_{\glnidx})b)t^{-1}
dtdb.
\end{align}

The co-period $\mathcal{I}(E_{1/2}(\cdot;\rho),\theta,\theta')$ is absolutely convergent (see \cite{me7} Claim~3.1).
In order to compute it we apply the truncation operator of Arthur \cite{A1,A2} to $E_{Q_{\glnidx}}(g;\rho,s)$ as in, e.g.,
\cite{JR,Jng,GRS7,GJR2}. For the complete argument see \cite{me7}, we provide a sketch. For a real number $d>1$, let $ch_{>d}:\R_{>0}\rightarrow\{0,1\}$ be the characteristic function of $\R_{>d}$ and set $ch_{\leq d}=1-ch_{>d}$. Denote
\begin{align*}
&\mathcal{E}_1^d(g;s)=\sum_{\delta\in\lmodulo{Q_{\glnidx}(F)}{G_{\glnidx}(F)}}\rho_s(\delta g)ch_{\leq d}(H(\delta g)),\\
&\mathcal{E}_2^d(g;s)=\sum_{\delta\in\lmodulo{Q_{\glnidx}(F)}{G_{\glnidx}(F)}}M(w,s)\rho_s(\delta g)ch_{>d}(H(\delta g)).
\end{align*}
Here $M(w,s)$ is the global intertwining operator corresponding to a representative $w$ such that $wU_{\glnidx}w^{-1}=U_{\glnidx}^-$. The co-periods $\mathcal{I}(\mathcal{E}_1^d(\cdot;s),\theta,\theta')$ and $\mathcal{I}(\mathcal{E}_2^d(\cdot;s),\theta,\theta')$ are absolutely convergent for $\Re(s)\gg0$ and
\begin{align}\label{eq:gspin truncation formula 1}
\mathcal{I}(\Lambda_dE_{Q_{\glnidx}}(\cdot;\rho,s),\theta,\theta')=\mathcal{I}(\mathcal{E}_1^d(\cdot;s)-\mathcal{E}_2^d(\cdot;s),\theta,\theta'),
\end{align}
where $\Lambda_dE_{Q_{\glnidx}}(\cdot;\rho,s)$ is the application of the truncation operator to
$E_{Q_{\glnidx}}(\cdot;\rho,s)$.

The following proposition describes 
$\mathcal{I}(\mathcal{E}_i^d(\cdot;s),\theta,\theta')$ and essentially completes the proof.
\begin{proposition}\label{proposition:gspin computation of the co-periods}
The co-periods $\mathcal{I}(\mathcal{E}_1^d(\cdot;s),\theta,\theta')$ and $\mathcal{I}(\mathcal{E}_2^d(\cdot;s),\theta,\theta')$ have meromorphic continuation to the whole plane. Moreover, 
\begin{align}\label{int:gspin co period 1 meromorphic continuation formula}
&\mathcal{I}(\mathcal{E}_1^d(\cdot;s),\theta,\theta')=
\frac{d^{s-1/2}}{s-1/2}
\int_{K}\int_{\lmodulo{\GLF{\glnidx}{F}}{\GLF{\glnidx}{\Adele}^1}}\rho(bk)\theta^{U_{\glnidx}}(bk){\theta'}^{U_{\glnidx}}(bk)dbdk,\\
\label{int:gspin co period 2 meromorphic continuation formula}
&\mathcal{I}(\mathcal{E}_2^d(\cdot;s),\theta,\theta')=
\frac{d^{-s-1/2}}{s+1/2}
\int_{K}\int_{\lmodulo{\GLF{\glnidx}{F}}{\GLF{\glnidx}{\Adele}^1}}M(w,s)\rho_s(bk)\theta^{U_{\glnidx}}(bk){\theta'}^{U_{\glnidx}}(bk)dbdk.
\end{align}
Here $\theta^{U_{\glnidx}}$ and ${\theta'}^{U_{\glnidx}}$ are given by \eqref{int:the constant term of theta along a unipotent subgroup}.
In particular, $\mathcal{I}(\mathcal{E}_1^d(\cdot;s),\theta,\theta')$ is holomorphic except perhaps at $s=1/2$, where it may have a simple pole.
\end{proposition}


Now taking the residue in \eqref{eq:gspin truncation formula 1} yields
\begin{align*}
\mathcal{I}(\Lambda_dE_{1/2}(\cdot;\rho),\theta,\theta')=\Res_{s=1/2}\mathcal{I}(\mathcal{E}_1^d(\cdot;s),\theta,\theta')-
\Res_{s=1/2}\mathcal{I}(\mathcal{E}_2^d(\cdot;s),\theta,\theta').
\end{align*}
Here $\Lambda_dE_{1/2}(\cdot;\rho)$ is the application of the truncation operator to $E_{1/2}(\cdot;\rho)$.
Taking $d\rightarrow\infty$ (and using \cite{LW} Proposition~4.3.3), we obtain
\begin{align*}
\mathcal{I}(E_{1/2}(\cdot;\rho),\theta,\theta')=
\int_{K}\int_{\lmodulo{\GLF{\glnidx}{F}}{\GLF{\glnidx}{\Adele}^1}}\rho(bk)\theta^{U_{\glnidx}}(bk){\theta'}^{U_{\glnidx}}(bk)dbdk.
\end{align*}
This proves part~\eqref{thm:gspin part 2} of the theorem. Part~\eqref{thm:gspin part 3} follows from Theorem~3.2 of \cite{GJR2} (see also \cite{JR} Proposition~2 and \cite{Jng}).

\begin{proof}[Proof of Proposition~\ref{proposition:gspin computation of the co-periods}]
The proof is a simple modification of the proof of Proposition~3.2 of \cite{me7}. We briefly present the argument. Consider $\mathcal{I}(\mathcal{E}_1^d(\cdot;s),\theta,\theta')$. Collapsing the summation into the integral gives
\begin{align}\label{int:gspin co period proposition proof 1}
\int_{\lmodulo{C_{G_{\glnidx}(\Adele)}Q_{\glnidx}(F)}{G_{\glnidx}(\Adele)}}\rho_s(g)ch_{\leq d}(H(g))\theta(g)\theta'(g)dg.
\end{align}
We write the Fourier expansion of $\theta'$ along the derived group $C_{U_{\glnidx}}$ of $U_{\glnidx}$.
Assume $\glnidx>1$, the case $\glnidx=1$ is trivial because the cover of $G_1(\Adele)$ splits (by Claim~\ref{claim:cover minimal cases}). The group $M_{\glnidx}(F)$ acts on the set of characters of $\lmodulo{C_{U_{\glnidx}}(F)}{C_{U_{\glnidx}}(\Adele)}$ with $\lfloor\glnidx/2\rfloor+1$ orbits.
For $0\leq j\leq\lfloor\glnidx/2\rfloor$, define a character $\psi_{j}$ of $\lmodulo{C_{U_{\glnidx}}(F)}{C_{U_{\glnidx}}(\Adele)}$ by
$\psi_{j}(c)=\psi(\sum_{i=1}^jc_{\glnidx-2i+1,\glnidx+2i})$, 
where $\psi$ is a nontrivial character of $\lmodulo{F}{\Adele}$. The stabilizer $St_{\psi_j}(F)$ of $\psi_{j}$ in $M_{\glnidx}(F)$ is equal to
$St'_{\psi_j}(F)\times G_0(F)$, where $St'_{\psi_j}(F)$ is the stabilizer computed in \cite{me7},
\begin{align*}
St'_{\psi_j}(F)=\left\{\left(\begin{array}{cc}a&z\\0&b\end{array}\right):a\in\GLF{\glnidx-2j}{F},b\in Sp_{j}(F),\text{$z$ is any $(\glnidx-2j)\times2j$ matrix}\right\}<\GLF{\glnidx}{F}
\end{align*}
and $Sp_{j}$ is the symplectic group in $2j$ variables. Then
\begin{align*}
\theta'(g)=\sum_{j=0}^{\lfloor\glnidx/2\rfloor}\sum_{\lambda\in\lmodulo{St_{\psi_j}(F)}{\GLF{\glnidx}{F}}}\int_{\lmodulo{C_{U_{\glnidx}}(F)}{C_{U_{\glnidx}}(\Adele)}}\theta'(c\lambda g)\psi_{j}(c)dc.
\end{align*}

Plugging this 
into \eqref{int:gspin co period proposition proof 1} gives a sum of integrals $\sum_{j=0}^{\lfloor\glnidx/2\rfloor}\mathcal{I}_j$
with
\begin{align*}
\mathcal{I}_j=&
\int_{K}\int_{\lmodulo{C_{G_{\glnidx}(\Adele)}St_{\psi_j}(F)}{M_{\glnidx}(\Adele)}}\rho_s(mk)ch_{\leq d}(H(m))\\\nonumber&
\int_{\lmodulo{U_{\glnidx}(F)}{U_{\glnidx}(\Adele)}}\theta(umk)
\left(\int_{\lmodulo{C_{U_{\glnidx}}(F)}{C_{U_{\glnidx}}(\Adele)}}\theta'(cumk)\psi_{j}(c)dc\right)du\delta_{Q_{\glnidx}(\Adele)}^{-1}(m)dmdk.
\end{align*}
Set $V_{j}=\lmodulo{(C_{U_{\glnidx}}\cap U_{\glnidx}\cap U_{\glnidx-2j})}{(U_{\glnidx}\cap U_{\glnidx-2j})}$ and
\begin{align*}
X_j(F)=\left\{\left(\begin{array}{ccccc}I_{\glnidx-2j}&0&v_1&c_1&c_2\\&I_{2j}&0&c_3&c_1'\\&&1&0&v_1'\\&&&I_{2j}&0\\&&&&I_{\glnidx-2j}\end{array}\right)\in U_{\glnidx}(F)\right\}.
\end{align*}
The constant term of $\theta'$ along $C_{U_{\glnidx}}$ is a function on
$\lmodulo{V_j(F)}{V_j(\Adele)}\isomorphic(\lmodulo{F}{\Adele})^{\glnidx-2j}$ and by considering a
Fourier expansion of this function and using Proposition~\ref{proposition:global basic Fourier vanishing result}, one sees that \begin{align*}
\int_{\lmodulo{C_{U_{\glnidx}}(F)}{C_{U_{\glnidx}}(\Adele)}}\theta'(c)\psi_{j}(c)dc
=\int_{\lmodulo{X_j(F)}{X_j(\Adele)}}\theta'(c)\psi_{j}(c)dc.
\end{align*}
Here $\psi_j$ was extended to a character of $X_j(\Adele)$, trivially on the coordinates of $v_1$ (and $v_1'$). Applying this to $\mathcal{I}_j$ and factoring the $du$-integration through $X_j$ yields
\begin{align*}
\mathcal{I}_j=&
\int_{K}\int_{\lmodulo{C_{G_{\glnidx}(\Adele)}St_{\psi_j}(F)}{M_{\glnidx}(\Adele)}}\rho_s(mk)ch_{\leq d}(H(m))
\int_{\lmodulo{X_j(\Adele)}{(\lmodulo{U_{\glnidx}(F)}{U_{\glnidx}(\Adele)})}}\\\nonumber&\left(\int_{\lmodulo{X_j(F)}{X_j(\Adele)}}
\theta(xu_1mk)\psi_{j}^{-1}(x)dx\right)
\left(\int_{\lmodulo{X_j(F)}{X_j(\Adele)}}\theta'(xu_1mk)\psi_{j}(x)dx\right)du_1\delta_{Q_{\glnidx}(\Adele)}^{-1}(m)dmdk.
\end{align*}

The following lemma is the heart of the unfolding argument.
\begin{lemma}\label{lemma:I_k vanishes for k > 0}
For any $j>0$, $\mathcal{I}_j=0$.
\end{lemma}
\begin{proof}
The proof of Lemma~3.5 in \cite{me7} is separated into two cases, $j<\glnidx/2$ and $j=\glnidx/2$. In the case $j<\glnidx/2$ one introduces an inner integration along a unipotent radical of $\GL{\glnidx}$, into $\mathcal{I}_j$, then uses the fact that $\rho_s|_{\GLF{\glnidx}{\Adele}}$ is a cusp form to show that $\mathcal{I}_j$ vanishes. The arguments include a Fourier expansion, an ``exchange of roots" and a few vanishing results. The vanishing results we use here are
Proposition~\ref{proposition:global basic Fourier vanishing result}, Claim~\ref{claim:twisted Jacquet modules vanish on small representations claim 1} and Proposition~\ref{proposition:vanishing along GLn of the Whittaker Jacquet for small representations}. Note that in \cite{me7} we also used a result on
a global constant term (with respect to $U_k$ and any $k\geq3$), this can be replaced with a local result - Proposition~\ref{proposition:vanishing along GLn of the Whittaker Jacquet for small representations} (as already observed in \cite{me7}, Remark~3.8).

Assume $j=\glnidx/2$ (in particular, $\glnidx$ is even). In this case $St_{\psi_{\glnidx/2}}=Sp_{\glnidx/2}$. Utilizing a result of Ikeda \cite{Ik3} (Proposition~1.3) on functions on Jacobi groups along with Proposition~\ref{proposition:global basic Fourier vanishing result}, 
one can introduce an inner integration
$\int_{\lmodulo{Sp_{\glnidx/2}(F)}{Sp_{\glnidx/2}(\Adele)}}\rho_s(ymk)dy$ into $\mathcal{I}_j$, which vanishes
according to Jacquet and Rallis \cite{JR} (Proposition~1). 
The arguments in \cite{me7} are applicable. We need to observe that, as in the case of the cover of $SO_{2\glnidx+1}(\Adele$), the restriction of the cover of $G_{\glnidx}(\Adele)$ to $Sp_{\glnidx/2}(\Adele)$ is a nontrivial double cover, which is therefore isomorphic to the usual metaplectic cover of $Sp_{\glnidx/2}(\Adele)$.
\end{proof}

Therefore, Integral~\eqref{int:gspin co period proposition proof 1} is equal to $\mathcal{I}_0$ and since $St_{\psi_0}(F)=M_{\glnidx}(F)$ and \begin{align*}
\lmodulo{(C_{G_{\glnidx}(\Adele)}M_{\glnidx}(F))}{M_{\glnidx}(\Adele)}\isomorphic\lmodulo{\GLF{\glnidx}{F}}{\GLF{\glnidx}{\Adele}},
\end{align*}
Integral $\mathcal{I}_0$ becomes
\begin{align}\label{int:gspin co period proposition proof last before inner dt}
\int_{K}\int_{\lmodulo{\GLF{\glnidx}{F}}{\GLF{\glnidx}{\Adele}}}\rho_s(bk)ch_{\leq d}(H(b))\theta^{U_{\glnidx}}(bk){\theta'}^{U_{\glnidx}}(bk)\delta_{Q_{\glnidx}(\Adele)}^{-1}(b)dbdk.
\end{align}

Regarding the case $\glnidx=1$, Lemma~\ref{lemma:minimal cases vanishing} immediately implies ${\theta'}={\theta'}^{U_1}$ and we also get that \eqref{int:gspin co period proposition proof 1} equals \eqref{int:gspin co period proposition proof last before inner dt}.
Now assume $\glnidx\geq1$.

Finally to extract the dependency on $s$ one uses \eqref{eq:equality defining the constant}. Assume $\chi$ (resp. $\chi'$) is given by \eqref{eq:global form of exceptional character} with respect to $\eta_{\chi},\psi_{\chi}$ (resp. $\eta_{\chi'},\psi_{\chi'}$). Then
\eqref{eq:global form of exceptional character} and \eqref{eq:condition for chi and chi' and eta} imply
\begin{align}\label{eq:application last calc 1}
\eta_{\chi}(t^{-2})\eta_{\chi'}(t^{-2})\eta(t)=1,\qquad\forall t\in\Adele^*.
\end{align}
Let $t\in A^+$. Then $\gamma_{\psi}(t)=1$ and according to \eqref{eq:global form of exceptional character GLn part},
\begin{align}\label{eq:application last calc 2}
\chi^{(1)}(\mathfrak{s}(\prod_{i=1}^{\glnidx}\eta_i^{\vee}(t)))=t^{\glnidx(\glnidx+1)/4}\eta_{\chi}(t)^{\glnidx}.
\end{align}
By Theorem~\ref{theorem:the constant term}, \eqref{eq:application last calc 1} and \eqref{eq:application last calc 2},
\begin{align*}
\theta^{U_{\glnidx}}(t\cdot I_{\glnidx}){\theta'}^{U_{\glnidx}}(t\cdot I_{\glnidx})=t^{\glnidx(\glnidx-1)/2}\eta(t)^{\glnidx/2}\theta^{U_{\glnidx}}(I_{\glnidx}){\theta'}^{U_{\glnidx}}(I_{\glnidx}).
\end{align*}
Since also
\begin{align*}
\rho_s(t\cdot I_{\glnidx})\delta_{Q_{\glnidx}(\Adele)}^{-1}(t\cdot I_{\glnidx})=t^{\glnidx(s-\glnidx/2)}\omega_{\tau}(t)\rho(I_{\glnidx}),
\end{align*}
when we apply
\eqref{eq:equality defining the constant} to \eqref{int:gspin co period proposition proof last before inner dt} we get
\begin{align*}
\glnidx^{-1}\int_{K}\int_{\lmodulo{\GLF{\glnidx}{F}}{\GLF{\glnidx}{\Adele}^1}}\rho(bk)\theta^{U_{\glnidx}}(bk)
{\theta'}^{U_{\glnidx}}(bk)\int_{0<t^{\glnidx}\leq d}\omega_{\tau}(t)\eta(t)^{\glnidx/2}t^{\glnidx(s-3/2)}dtdbdk.
\end{align*}
Because $\omega_{\tau}(t)\eta(t)^{\glnidx/2}=1$,
we reach \eqref{int:gspin co period 1 meromorphic continuation formula}.

Similar arguments apply to $\mathcal{I}(\mathcal{E}_2^d(\cdot;s),\theta,\theta')$ and we obtain \eqref{int:gspin co period 2 meromorphic continuation formula}. Note that
\[
M(w,s)\rho_s(t\cdot I_{\glnidx})\delta_{Q_{\glnidx}(\Adele)}^{-1}(t\cdot I_{\glnidx})=t^{\glnidx(s-\glnidx/2)}\omega_{\tau}^{-1}(t)\eta(t)^{-\glnidx}\rho(I_{\glnidx}).\qedhere
\]
\end{proof}

\bibliographystyle{plain}

\end{document}